\documentclass[11pt,letterpaper]{article}

\usepackage[T1]{fontenc}
\usepackage{lmodern}

\usepackage[obeyclassoptions=true,subpreambles=true]{standalone}
\usepackage{import}

\usepackage[margin=1in,marginpar=0.8in]{geometry}
\usepackage{mathtools,amssymb,amstext,mathrsfs}
\usepackage{amsthm}
\usepackage{authblk}
\usepackage{enumerate}
\usepackage[shortlabels]{enumitem}

\usepackage{dirtytalk}
\usepackage{tikz}
\usepackage{thmtools}
\usepackage{thm-restate}
\usepackage[normalem]{ulem}

\usepackage{thmtools}
\usepackage{thm-restate}

\usepackage{float,subcaption}
\usepackage{bookmark}

\newtheorem{thm}{Theorem}[section]

\newtheorem{prop}[thm]{Proposition}
\newtheorem{lemma}[thm]{Lemma}
\newtheorem{conj}[thm]{Conjecture}
\newtheorem{cor}[thm]{Corollary}
\newtheorem{obs}[thm]{Observation}
\newtheorem{claim}{Claim}
\newtheorem{subclaim}{Subclaim}

\numberwithin{equation}{section}
\theoremstyle{definition} 
\newtheorem{definition}[thm]{Definition}

\newenvironment{cproof}[1][\proofname]
{\proof[#1]}
{\endproof}
\newenvironment{scproof}[1][\proofname]
{\proof[#1]}
{\endproof}

\renewcommand{\subset}{\subseteq}
\newcommand{\g}{g}

\newcommand{\del}{\mathsf{Del}}
\newcommand{\delsave}{\mathsf{DelSave}}
\renewcommand{\bar}{\overline}

\definecolor{asparagus}{rgb}{0.53, 0.66, 0.42}

\definecolor{cerulean}{rgb}{0.0, 0.48, 0.65}
\definecolor{cornellred}{rgb}{0.7, 0.11, 0.11}
\definecolor{darklavender}{rgb}{0.45, 0.31, 0.59}
\definecolor{darkslateblue}{rgb}{0.28, 0.24, 0.55}
\definecolor{burntorange}{rgb}{0.8, 0.33, 0.0}
\date{}

\definecolor{nicelavender}{RGB}{153, 128, 250}

\begin{document}


\author{Ewan Davies\thanks{Supported in part by NSF grant CCF-2309707}}
\affil{Department of Computer Science, Colorado State University \\ \texttt{ewan.davies@colostate.edu}}
\author{Evelyne Smith-Roberge}
\affil{School of Mathematics, Georgia Institute of Technology \\ \texttt{evelyne.smithroberge@gmail.com}}
\title{Local Weak Degeneracy of Planar Graphs}
\date{\today}
\maketitle
\date{}

\begin{abstract}
Thomassen showed that planar graphs are 5-list-colourable, and that planar graphs of girth at least five are 3-list-colourable. An easy degeneracy argument shows that planar graphs of girth at least four are 4-list-colourable. 
In 2022, Postle and Smith-Roberge proved a common strengthening of these three results: with $g(v)$ denoting the length of a shortest cycle containing a vertex $v$, they showed that if $G$ is a planar graph and $L$ a list assignment for $G$ where $|L(v)| \geq \max\{3,8-g(v)\}$ for all $v \in V(G)$, then $G$ is $L$-colourable. 
Moreover, they conjectured that an analogous theorem should hold for correspondence colouring. 
We prove this conjecture; in fact, our main theorem holds in the still more restrictive setting of \emph{weak degeneracy}, and moreover acts as a joint strengthening of the fact that planar graphs are weakly 4-degenerate (originally due to Bernshteyn, Lee, and Smith-Roberge), and that planar graphs of girth at least five are weakly 2-degenerate (originally due to Han et al.). 
\end{abstract}
\maketitle

\section{Introduction}\label{sec:intro}
Colouring planar graphs is one of the oldest and most well-known problems in graph theory. 
Perhaps the most famous theorem in graph colouring is the \emph{Four Colour Theorem} of Appel and Haken~\cite{AH77a,AHK77}, which states that the vertices of every planar graph can be coloured with only four colours such that the endpoints of each edge receive different colours.

Graphs in this paper are finite, and contain neither loops nor parallel edges. Given a graph $G$, we refer to its vertex- and edge-sets as $V(G)$ and $E(G)$, respectively.
The \emph{chromatic number} of a graph $G$ is the minimum number $k$ such that there exists a proper \emph{$k$-colouring} of the kind above: a function $\varphi:V(G) \rightarrow \{1, 2, \dots, k\}$ where for each edge $uv \in E(G)$, we have that $\varphi(u) \neq \varphi(v)$. We say that $G$ is \emph{$k$-colourable} if it admits a $k$-colouring.

As there exist planar graphs that require four colours in any proper colouring (e.g.\ $K_4$, the complete graph on four vertices), the four colour theorem is in a sense best possible. 
To gain a deeper understanding of the chromatic properties of planar graphs, it is very natural to further limit their structure, and investigate the resulting class of graphs. 
One commonly studied structural limitation is a lower bound on \emph{girth} $g(G)$, the length of a shortest cycle in $G$ (the girth of a forest is defined to be infinite).
Gr\"{o}tzsch's theorem~\cite{grotzsch1959dreifarbensatz} concerns the chromatic number of planar graphs from this perspective: it states that planar graphs with girth at least four are 3-colourable. Like the four colour theorem, Gr\"{o}tzsch's theorem is best possible: there exist planar graphs of arbitrarily large girth that require three colours in any proper colouring (e.g.\ odd cycles).
From here, to further understand the colouring properties of planar graphs, one could attempt to limit graph structure in more complex ways. For example, one could allow triangles under the condition that they be sufficiently far apart \cite{dvovrak2021three}. A different approach is to investigate different notions of colouring, beyond standard vertex colouring.

\emph{List colouring} is the most widely-studied generalization of colouring. 
It was introduced in the mid-70s by Vizing~\cite{vizing}, and independently by Erd\H{o}s, Rubin, and Taylor~\cite{erdos1979choosability}. Intuitively, list colouring is the result of localizing the colour palette available to each vertex: that is, the result of restricting the possible images of the colouring function for each vertex. Formally, it is defined as follows.

\begin{definition}
    Given a positive integer $k$ and graph $G$, a \emph{$k$-list assignment} for $G$ is a function $L$ that assigns to each vertex $v \in V(G)$ a set $L(v)$ of size at least $k$. We refer to the elements of $L(v)$ as colours. An \emph{$L$-colouring} is a function $\varphi$ that assigns to each vertex $v \in V(G)$ a colour in $L(v)$, and where for each $uv \in E(G)$, we require $\varphi(u) \neq \varphi(v)$. $G$ is \emph{$L$-colourable} if it has an $L$-colouring; it is \emph{$k$-list-colourable} if it has an $L$-colouring for every $k$-list assignment $L$. The \emph{list chromatic number} of $G$ is denoted by $\chi_\ell(G)$ and defined as the minimum $k$ such that $G$ is $k$-list-colourable.
\end{definition}

Clearly, list colouring generalizes standard vertex colouring: we recover the standard notion by setting $L(u) = L(v)$ for every pair of vertices $u,v$ in a graph $G$. In general, the list chromatic and chromatic numbers of a graph can be arbitrarily far apart: there exist bipartite (2-colourable) graphs with arbitrarily high list chromatic number. For planar graphs, the situation is a bit simpler: Thomassen showed (with a remarkably elegant proof) that planar graphs are 5-list-colourable~\cite{thomassen5LC}. As demonstrated by Voigt~\cite{voigt1993list}, this is best possible. 
Analogous to Gr\"{o}tzsch's theorem, by ruling out triangles we can lower this: every planar graph of girth at least four is 4-list-colourable.  
Unlike Gr\"{o}tzsch's theorem, this result is elementary, and follows easily the fact that triangle-free planar graphs are \emph{3-degenerate} (which itself follows easily from Euler's formula for embedded graphs). This is also the best we can do for planar graphs of girth four: Voigt \cite{voigt1995not} constructed such a graph that is not 3-list-colourable. 
We recall the definition of degeneracy below, about which we will say more later.

\begin{definition}
    A graph $G$ is \emph{$k$-degenerate} if every subgraph of $G$ contains a vertex of degree at most $k$. 
\end{definition}

Continuing a theme of restricting girth, Thomassen \cite{thomassen3LC,thomassen3LCnew} gave two proofs that planar graphs of girth at least five are 3-list-colourable. As this is best possible for ordinary colouring (e.g.\ arbitrarily long odd cycles), it follows that it is also best for list colouring. 
We summarize the state of affairs for list colouring planar graphs $G$ of given girth thus: $\chi_\ell(G)\le \max\{8 - g(G), 3\}$.

To gain a deeper understanding of the list colouring properties of planar graphs, one might try instead to limit list sizes based on \emph{local} structure, rather than a global graph parameter such as the girth.
Indeed, understanding the extent to which local structure affects global properties such as the chromatic number is a major theme in graph theory. In the colouring context, notable examples include results where list sizes depend on individual vertex degrees~\cite{borodin1997list,DJKP20a,BDLP24}, or more involved properties of neighbourhoods such as density, clique number, and Hall ratio~\cite{kelly2020local,BKNP22,PH21a,DKPS20a}. To this end, we recall the following definitions from \cite{PS22}.

\begin{definition}\label{def:girth}
    Given a graph $G$ and a vertex $v \in V(G)$, we define the \emph{girth of $v$} as the length of a shortest cycle in $G$ containing $v$. When $v$ is not contained in a cycle, we define its girth to be infinite. We denote the girth of $v$ as $g_G(v)$, dropping the subscript when the choice of $G$ is clear from context.
\end{definition}
\begin{definition}\label{def:localgirth}
    Let $G$ be a graph, and $L$ a list assignment for $G$. We say $L$ is a \emph{local girth list assignment} if 
    \begin{itemize}[itemsep=0pt]
        \item every vertex $v \in V(G)$ has $|L(v)| \geq 3$,
        \item every vertex $v \in V(G)$ with $g(v) = 4$ has $|L(v)| \ge 4$, and
        \item every vertex $v \in V(G)$ with $g(v) = 3$ has $|L(v)| \ge 5$. 
    \end{itemize}
    Equivalently, $|L(v)|\ge \max\{8-g(v),3\}$. If $G$ admits an $L$-colouring for every local girth list assignment $L$, we say $G$ is \emph{local girth list colourable}.
\end{definition}

In 2022, Postle and the second author~\cite{PS22} proved that every planar graph is local girth list colourable, uniting and strengthening both theorems of Thomassen mentioned above as well as the fact that triangle-free planar graphs are 4-list-colourable. 

List colouring, however, is merely the beginning of a sequence of generalisations of vertex colouring. 
While list colouring localises the palette of colours available to a vertex, a further generalisation due to Dvo{\v{r}}{\'a}k and Postle~\cite{dvovrak2018correspondence} known as \emph{correspondence colouring} (or DP-colouring) localises the constraint imposed by each edge of the graph on the function $\varphi$ assigning a vertex to a colour on its list. Correspondence colouring is defined as follows.

\begin{definition}\label{def:correspondence}
 Let $G$ be a graph. A \emph{$k$-correspondence assignment} $(L,M)$ for $G$ is a $k$-list assignment $L$ together with a function $M$ that assigns to every edge $e = uv \in E(G)$ a matching $M_e$ between $\{u\}\times L(u)$ and $\{v\}\times L(v)$.
An $(L,M)$-colouring of $G$ is a function $\varphi$ that assigns to each vertex $v \in V(G)$ a colour $\varphi(v) \in L(v)$ such that for every edge $e=uv \in E(G)$, the vertices $(u, \varphi(u))$ and $(v, \varphi(v))$ are non-adjacent in $M_{e}$. We say that $G$ is $(L,M)$-colourable if it admits an $(L,M)$-colouring, and that $G$ is \emph{$k$-correspondence-colourable} if $G$ is $(L,M)$-colourable for every $k$-correspondence assignment $(L,M)$ for $G$. The \emph{correspondence chromatic number} of $G$ is denoted $\chi_c(G)$ and defined to be the minimum number $k$ such that $G$ is $k$-correspondence colourable.
\end{definition}

\noindent 
Also important to our study is the following.

\begin{definition}
    If $G$ is a graph and $(L,M)$ is a correspondence assignment for $G$ where $L$ is a local girth list assignment, we say that $(L,M)$ is a \emph{local girth correspondence assignment.}
\end{definition}

Clearly, for every graph $G$ we have that
$$
\chi(G) \leq \chi_\ell(G) \leq \chi_c(G),
$$
since one can take matchings $M_{uv}$ that connect each $(u,i)$ to $(v,i)$ for each $i\in L(u)\cap L(v)$ to recover the list colouring framework. 
Relevant to our discussion of planar graphs, we note that correspondence colouring was originally introduced~\cite{dvovrak2018correspondence} in order to study the list chromatic number of planar graphs with no cycles of length four to eight, a more complex global condition than girth on the cycles in a planar graph.
Correspondence colouring quickly became the subject of sustained study in its own right; the original paper~\cite{dvovrak2018correspondence} has over 250 citations at the time of writing. 

With only minor semantic changes, Thomassen's proofs of the 5-list-colourability of planar graphs and 3-list-colourability of planar graphs of girth at least five also apply in the correspondence colouring setting. Since $k$-degeneracy implies $(k+1)$-correspondence colourability, it is moreover true that planar graphs of girth at least four are 4-correspondence colourable. 
Nevertheless, the proof of Postle and the second author on the local girth list colourability of planar graphs does \emph{not} go through in the correspondence colouring framework, though the authors of \cite{PS22} conjectured in~\cite{PS23} that the correspondence analogue of their theorem should hold.
In this paper, we prove this conjecture.

\begin{thm}\label{thm:localgirthcorrcol}
    Every planar graph is local girth correspondence colourable.
\end{thm}

In fact, our theorem holds in the much more restrictive setting of \emph{weak degeneracy}. 
This is a parameter akin to degeneracy that is relevant to various graph colouring parameters. 
To explain, we require a definition.

\begin{definition}\label{def:delsave}
    Let $G$ be a graph and let $f:V(G) \rightarrow \mathbb{N}$ be a function.
    \begin{enumerate}
        \item\label{itm:del}
        Let $\del(G, f, v) := (G-v, f')$, where $f':V(G-v) \rightarrow \mathbb{N}$ is given by 
        \[ f'(u) = 
        \begin{cases*} 
        f(u)-1 &if $u\in N_G(v)$ \\
        f(u) &otherwise.
        \end{cases*} \]
        If $v\in V(G)$ then we say that $\del(G, f, v)$ is \emph{legal} if $f'(u) \geq 0$ for all $u \in V(G) \setminus \{v\}$. 
        \item\label{itm:delsave}
        We let $\delsave(G, f, v, w) := (G-v, f'')$, where $f'':V(G-v) \rightarrow \mathbb{N}$ is given by
        \[ f''(u) = 
        \begin{cases*}
            f(u)-1 & if $u\in N_G(v)\setminus \{w\}$ \\
            f(u) & otherwise.
        \end{cases*}
        \]
        If $v\in V(G)$ and $w\in N_G(v)$ then we say that $\delsave(G, f, v, w)$ is \emph{legal} if $f(v) > f(w)$ and $f''(u)\ge 0$ for all $u\in V(G)\setminus\{v\}$. 
    \end{enumerate}
\end{definition}

With this definition in mind, we can give an equivalent definition of degeneracy by way of the $\del$ operation: a graph $G$ is $k$-degenerate if, beginning from the constant function $f:V(G) \rightarrow \{k\}$, there exists a legal sequence of $\del$ operations that removes all the vertices of $G$. 
In the context of list and correspondence colouring  we think of the function $f$ as a measure of the number of available colours, though somewhat annoyingly these parameters are off-by-one. 
One can show that a $k$-degenerate graph is $(k+1)$-correspondence colourable with the following argument. 
Suppose that a graph is $k$-degenerate, and consider a $(k+1)$-correspondence assignment. 
If we assign colours one-by-one, every time we colour-and-delete a vertex we have to remove a corresponding colour from the list of each of its neighbours colours. 
Such a one-by-one colouring succeeds if whenever we delete a vertex, none of its neighbours end up with no colours on their lists. Thus, the vertices can be coloured in the order they are deleted in a legal fashion starting with the function $f$.

Weak degeneracy is a more flexible version of degeneracy that permits another operation, one that ``saves'' a colour in the above colouring analogy.

\begin{definition}\label{defs:weakdegen}
    Let $G$ be a graph and $f:V(G) \rightarrow \mathbb{N}$ be a function. We say $G$ is \emph{weakly $f$-degenerate} if there exists a sequence of legal $\del$ and $\delsave$ operations that removes every vertex of $G$.
\end{definition}

Weak degeneracy was introduced in 2023 by Bernshteyn and Lee~\cite{weakdegen}. Since then, the topic has received a remarkable amount of attention considering the short timeline. Most relevant to our discussion: Han, Wang, Wu, Zhou, and Zhu \cite{weakdegengirth5} proved that planar graphs of girth at least five are weakly 2-degenerate; since planar graphs of girth at least four are 3-degenerate, they are trivially weakly 3-degenerate; and Bernshteyn, Lee, and the second author proved that planar graphs are weakly 4-degenerate \cite{bernshteyn2024weak}.
Our main theorem is a unified strengthening of all of these results in the local girth setting. Before stating it, we require one final definition (which copes with the aforementioned off-by-one problem).

\begin{definition}
    The function $f$ is a \emph{local girth function} for a graph $G$ if $f:V(G) \rightarrow \mathbb{N}$ has $f(v) \ge \max\{7-g(v),2\}$ for every $v \in V(G)$.
\end{definition}

\begin{thm}\label{thm:main}
    If $G$ is a planar graph and $f$ is a local girth function for $G$, then $G$ is weakly $f$-degenerate.
\end{thm}

This directly implies Theorem \ref{thm:localgirthcorrcol}. In fact, as discussed in  \cite{weakdegen}, this weak degeneracy theorem implies an analogous theorem for correspondence‐\emph{painting}, an online version of the correspondence colouring introduced by Kim et al.~\cite{kim2020line}. 

The proof of Theorem \ref{thm:main} is inspired by that of Thomassen's proofs on the 5-list-colorability of planar graphs \cite{thomassen5LC} and the 3-list-colourability of planar graphs of girth at least five \cite{thomassen3LC,thomassen3LCnew}; as well as the proof of Postle and Smith-Roberge on local girth list colouring planar graphs \cite{PS22}; and the proof of a theorem of Dvo\v{r}\'{a}k, Lidick\`{y}, and Mohar \cite{dvovrak20175}, on extending the precolouring of a path of length two in a planar graph to a 5-list-colouring of the entire graph. Each of these proofs is established via proving a stronger, more restrictive statement: in each case, a plane embedding and list assignment $L$ for a graph $G$ are fixed. Vertices in the outer face boundary walk of $G$ have more restricted list sizes, and a short path on the outer face boundary is precoloured. The proofs then argue about extending the precolouring to an $L$-colouring of $G$. 
In some of these proofs, the stronger statement proved by induction is made more complex by the presence of \emph{exceptions}. 
That is, there is some excluded subgraph-type structural assumption that the graph and list assignment must satisfy in order for the conclusion to hold. This means that whenever one applies the induction hypothesis one must check that the required structural assumption holds. 
Structural conditions that appear in prior works~\cite{thomassen2007exponentially,PS22} include the absence of so-called \emph{generalized wheels}. Our proof involves analogous exceptions, though they are simpler to describe. In particular, we do not have any infinite families of exceptions.

The proofs of the list colouring theorems discussed above do not readily generalize to the weak degeneracy framework, and the proof of the local girth result of Postle and Smith-Roberge does not readily generalize to the correspondence colouring framework. As a result, though many steps of our arguments are greatly influenced by theirs, the analysis is more delicate as the tools available to us are far more restricted.  
There is some prior work on this type of proof in the weak degeneracy framework, however. 
Our proof is also inspired by the proof of the weak 4-degeneracy of planar graphs by Bernshteyn, Lee, and Smith-Roberge \cite{bernshteyn2024weak}; this too is established via a stronger inductive statement that is more restrictive (in analogous ways to those described above, in the weak degeneracy setting). 

In some of the results listed, the stronger inductive statement involves keeping track of an independent set of vertices in the outer face boundary walk of $G$, whose list assignments are still more restricted. For instance, in \cite{PS22} and \cite{thomassen3LC, thomassen3LCnew}, the analysis involves an independent set $A$ of vertices of girth at least five and list size at most two. In \cite{dvovrak20175} and \cite{bernshteyn2024weak} (itself inspired by \cite{dvovrak20175}), the analysis involves tracking an independent set $B$ of vertices of girth three and list size at most three (analogously restricted in the case of weak degeneracy) in the outer face boundary walk of $G$. 
To the best of our knowledge, our proof is the first that keeps track of analogous version of both $A$ \emph{and} $B$. Keeping track of these two independent sets of further restricted vertices complicates the analysis, in that when we perform any sort of reduction, there are more properties the resulting graph and list assignment must satisfy in order to ensure the inductive hypothesis holds. 
On the other hand, though our inductive statement is more complicated than that given in \cite{PS22} (showing planar graphs are local girth list colourable), there are fewer exceptions: in particular, two infinite families of exceptions in \cite{PS22} are each replaced with a single graph. 
The result overall is a shorter argument, and moreover one which is self-contained (and in particular which no longer relies on a theorem of Thomassen from \cite{thomassen2007exponentially}).  
Like the proof of \cite{PS22} (and unlike the proofs of the remainder of the results listed), some of our reductions involve deleting an arbitrarily long path in the outer face boundary walk of $G$, instead of merely a handful of vertices.

Though our main result is a common strengthening of several influential results on list colouring planar graphs, we do \emph{not} prove a strengthening of \cite[Theorem 4]{thomassen2007exponentially}, which implies that a planar graph $G$ with 5-list assignment $L$ admits at least $2^{v(G)/9}$ $L$-colourings. 
Similarly, we do not strengthen \cite[Theorem 6.5]{PS23} which states that a planar graph $G$ with local girth list assignment $L$ has at least $5^{v(G)/12}$ distinct $L$-colourings. 
More precisely, we do not prove any exponential lower bound on the number of $L$-colourings in these settings.

The literature on colourings of planar graphs provides examples of proofs that proceed in a manner similar to ours that do~\cite{thomassen2007exponentially} and do not~\cite{thomassen5LC,PS22} immediately yield an exponential lower bound on the number of colourings in question. 
We do not believe that an exponential lower bound can be extracted from our arguments without substantial extra work, in part due to the fact that we remove an arbitrarily long path in the outer face boundary walk of the graph as part of our main induction. 
One could possibly show exponentially many colourings by extending the methods of the second author and Postle~\cite{PS23} to the more general setting we consider. 
These methods depend on results similar to ours to establish the existence of a single colouring, and so we provide an important starting point for such analysis.

There is a natural analogue of exponentially many colourings in the weak degeneracy setting, though it is slightly delicate to define. 
\begin{definition}
    The \emph{availability} $a(\del(G, f, v))$ of the operation $\del(G, f, v)$ is $f(v)+1$. 
    The \emph{availability} $a(\delsave(G,f,v,w))$ of the operation $\delsave(G,f,v,w)$ is $f(v)-f(w)$.
    The \emph{average availability} $\overline a(\sigma)$ of a sequence of $n$ operations $\sigma$ is the geometric mean of the availabilities: $\overline a(\sigma) = \prod_{i=1}^n a(\sigma_i)^{1/n}$.
\end{definition}
\noindent
Note that the availability of an operation is at least one when the operation is legal sequence, and every legal sequence has average availability at least one. The following result is easy to prove by induction on the number of vertices of $G$.
\begin{prop}
    Let $G$ be a graph and let $f:V(G)\to \mathbb{N}$ be a function. 
    Suppose that the sequence of operations $\sigma$ certifies that $G$ is  weakly $f$-degenerate. Then for any correspondence assignment $(L,M)$ of $G$ with $|L(v)|\ge f(v)+1$, $G$ admits at least $\overline{a}(\sigma)^{v(G)}$ distinct $(L,M)$-colourings.
\end{prop}
\noindent
We end the introduction with a conjecture that, via the above result, would establish the existence of exponentially many local girth correspondence colourings of planar graphs.
\begin{conj}
    There exists a constant $c>1$ such that the following holds. 
    Given a planar graph $G$ and local girth function $f$ for $G$, there exists a legal sequence of operations $\sigma$ with average availability at least $c$ that removes every vertex of $G$.
\end{conj}

\section{Preliminaries}

Throughout, we use $v(G)$ and $e(G)$ to denote the number of vertices and edges, respectively, in a graph $G$. Given a graph $G$, path $P = v_1v_2 \dots v_t \subseteq G$, and vertex $v \in V(G)\setminus V(P)$ with $vv_1 \in E(G)$, we denote the path $vv_1v_2\dots v_t$ by $v+P$, and similarly define $P+v$ when $vv_t\in E(G)$. We use the convention that $\mathbb{N} = \{0, 1, 2, 3, \dots\}$.

We start with a refinement of the definitions of legality for the weak degeneracy operations that we find technically convenient. 
    \begin{definition}\label{def:techWD}
    Let $G$ be a graph and let $f:V(G) \rightarrow \mathbb{N}$ be a function.

    If $v\notin V(G)$ we note that $\del(G, f, v)$ is well-defined  but simply returns $(G, f)$ unchanged. In this case we define the operation to be trivially legal.
        
    Analogously, we extend the definition of legality for $\delsave$ as follows.
        \begin{itemize}
            \item If $v\notin V(G)$ then for any $w$, $\delsave(G,f,v,w)=(G,f)$ and is trivially legal. 
            \item If $v\in V(G)$ and $w\notin N_G(v)$ then $\delsave(G,f,v,w)=\del(G,f,v)$ and is legal if $\del(G,f,v)$ is legal.
        \end{itemize}
    
    When $G$ and $f$ are clear from context, we write $\del(v)$ and $\delsave(v,w)$ instead of $\del(G, f, v)$ and $\delsave(G, f, v, w)$. 
    This means that we can think of $\del(v)$ and $\delsave(v,w)$ operations as functions on a set $\mathcal U$ of graph-function pairs: $\del(v):\mathcal U\to\mathcal U$ is the map $(G, f) \mapsto \del(G, f, v)$ and $\delsave(v,w):\mathcal U\to\mathcal U$ is the map $(G, f) \mapsto \delsave(G, f, v, w)$.
    
    Using this perspective, given a starting graph-function pair $(G, f)\in\mathcal U$ and a sequence $\sigma = (\sigma_1, \dotsc, \sigma_k)$ of $\del$ and $\delsave$ operations, we say that the sequence is \emph{legal} if when we write $(G_0,f_0)=(G,f)$ and $(G_i, f_i) = \sigma_i(G_{i-1}, f_{i-1})$ for $1\le i\le k$, each $\sigma_i(G_{i-1}, f_{i-1})$ is a legal application of either $\del$ or $\delsave$.
\end{definition}

We consider it obvious that list colouring is monotone in the sense that if we have a graph $G$ and list assignment $L$ such that $G$ admits an $L$-colouring, then $G$ must admit an $L'$-colouring for any list assignment $L'$ with $L'(v)\supseteq L(v)$ for all $v\in V(G)$. 
The following lemma of Bernshteyn and Lee~\cite{weakdegen} establishes this in the weak degeneracy setting.

\begin{lemma}[{Bernshteyn and Lee~\cite[Lem.~2.1]{weakdegen}}]\label{lem:monotone}
    Let $G$ be a graph and let $f,f':V(G)\to\mathbb N$ be functions such that $f'\ge f$ pointwise. Then if $G$ is weakly $f$-degenerate, it is also weakly $f'$-degenerate.
\end{lemma}
The proof is elementary: one simply has to observe that the only issue is if an application of $\delsave$ becomes illegal due to the increased function value, and in this case one can afford to replace $\delsave$ with $\del$.

The rest of this section introduces notation we need to describe the structure of our inductive proof.

\begin{definition}
    Let $G$ be a plane graph. We denote by $\delta G$ the graph whose vertices and edges comprise the outer face boundary of $G$.
\end{definition}

\begin{definition}\label{def:accpath}
A path $P$ with $v(P)\le 4$ is \emph{acceptable} in a graph $G$ if 
\begin{enumerate}[label=\textup{(A\arabic*)}]
    \item $V(P)\subset V(G)$,
    \item $V(P)$ induces a path in $G$, and
    \item if $v(P) = 4$, then either 
    \begin{enumerate}[label=\textup{(A3\alph*)}]
        \item  both internal vertices of $P$ have girth at least $4$ (in $G$), or 
        \item  one internal vertex of $P$ has girth at least $5$ (in $G$).
    \end{enumerate} 
\end{enumerate}
A cycle $C$ is \emph{acceptable} in $G$ if $C$ contains a spanning acceptable path. 
\end{definition}

\begin{definition}\label{def:canvas}
A \emph{canvas} is a tuple $(G, P, A, B, f)$ where $G$ is a plane graph and 
\begin{enumerate}[label=\textup{(C\arabic*)}]
    \item\label{canv:acceptable} $P\subseteq \delta G$ is an acceptable path or cycle in $G$,
    \item\label{canv:A} $A \cap V(G) \subseteq V(\delta G)\setminus V(P)$ is an independent set of vertices of girth at least 5, 
    \item\label{canv:B} $B \cap V(G) \subseteq V(\delta G) \setminus V(P)$ is an independent set of vertices of girth 3,
    \item\label{canv:f} $f$ is a function taking values in $\mathbb{N}$ where the restriction of $f$ to $V(G)\setminus V(P)$ satisfies the following:
    \begin{enumerate}[label=\textup{(C4\alph*)}]
        \item\label{canv:fA} for all $v\in A \cap V(G)$, $f(v)=1$,
        \item\label{canv:fB} for all $v\in B \cap V(G)$, $f(v)=2$,
        \item\label{canv:fCn3} for all $v\in V(\delta G)\setminus (V(P) \cup A \cup B)$, if $g(v) \neq 3$ then $f(v) \geq 2$,
        \item\label{canv:fC3} for all $v\in V(\delta G)\setminus (V(P) \cup A \cup B)$, if $g(v) = 3$ then $f(v) \geq 3$,
        \item\label{canv:fint} for all $v\in V(G) \setminus V(\delta G)$, 
        \begin{itemize}
            \item if $g(v) = 3$ then $f(v)\ge 4$,
            \item if $g(v) = 4$ then $f(v)\ge 3$,
            \item if $g(v) \ge 5$ then $f(v)\ge 2$.
        \end{itemize}
    \end{enumerate}
\end{enumerate}
\end{definition}

We note the following.
\begin{obs}\label{obs:subcanvas}
    Removing vertices from a canvas keeps it a being a canvas and cannot turn it into an exceptional one.
\end{obs}

\begin{definition}\label{def:degencanvas}
    Given a canvas $K=(G,P,A,B,f)$, we write $f_K$ for the function $f_K:V(G-P)\to \mathbb Z$ given by $f_K(v) = f(v) - |N(v)\cap V(P)|$.
    
   We say that a canvas $K=(G,P,A,B,f)$ is \emph{weakly $f$-degenerate} if the graph $G-P$ is weakly $f_K$-degenerate.
\end{definition}

Note that we need to allow for negative integers in the image of $f_K$ above for technical reasons, once we define our stronger statement for induction we will see that in unexceptional cases the image of $f_K$ is a subset of $\mathbb{N}$ as expected.
The above definition connects with a common theme in earlier works on list colouring planar graphs. We think of the vertices in $P$ as precoloured and seek to extend the colouring to $G-P$ in a way that yields a valid colouring of $G$. The analogue of this in the weak degeneracy framework is that $P$ has been ``preremoved'' by $\del$ or $\delsave$ operations and we wish to continue the sequence of operations and remove all of $V(G-P)$ legally, including the effect of the function value being reduced by the prior removal of $P$.

\begin{thm}\label{thm:inductive}
    Let $G$ be a plane graph and let $P\subset \delta G$ be either an acceptable path $u_1u_2\dotsb u_k$ or an acceptable cycle $u_1u_2\dotsb u_ku_1$ in $G$. 
    Then every canvas $(G,P,A,B,f)$ is weakly $f$-degenerate unless one of the following exceptional cases holds:
    \begin{enumerate}[label={\textup{(X\arabic*)}}]
        \item\label{ex:B} $k=3$, $P$ is a path, and there is a vertex $v\in B$ adjacent to each vertex in $P$.
        \item\label{ex:A} $k=4$, $P$ is a path, and there is a vertex $v\in A\setminus V(P)$ such that $\{u_1, u_k\} \subset N(v)$.
        \item\label{ex:AB} $k=4$, $P$ is a path, and up to reversing the names of vertices on $P$ there is a vertex $v_1\in B$ such that $\{u_1, u_2\} \subset N(v_1)$, and there is a vertex $v_2\in A$ such that $\{u_4, v_1\}\subset N(v_2)$.
    \end{enumerate}
\end{thm}

If a canvas contains the structure described by \ref{ex:B}, \ref{ex:A}, or \ref{ex:AB} above, we call it an \emph{exceptional canvas of type \ref{ex:B}, \ref{ex:A}, or \ref{ex:AB}}, respectively. See Figure~\ref{fig:exceptions} for an illustration of exceptional canvases.  Note that if $(G,P,A,B,f)$ is an exceptional canvas, then $v(P) \geq 3$.

\begin{figure}[ht]\centering
    \subcaptionbox{Exception of type~\ref{ex:B}.\label{fig:exceptionB}}[0.3\textwidth]{\begin{tikzpicture}[scale=1]
    \node[precoloured] (u1) at (120:2cm){};
    \node[precoloured] (u2) at (90:2cm){};
    \node[precoloured] (u3) at (60:2cm){};
    \node[vertex,label={below:$v\in B$}] (v) at (0,0) {};
    \draw (v) -- (u1) -- (u2) -- (u3) -- (v) -- (u2);
\end{tikzpicture}

\begin{tikzpicture}[scale=1]
    \node[precoloured] (u1) at (135:2cm){};
    \node[precoloured] (u2) at (105:2cm){};
    \node[precoloured] (u3) at (75:2cm){};
    \node[precoloured] (u4) at (45:2cm){};

    \node[vertex,label={below:$v\in A$}] (v) at (0,0) {};
    \draw (v) -- (u1) -- (u2) -- (u3) -- (u4) -- (v);
\end{tikzpicture}

\begin{tikzpicture}[xscale=-1,yscale=1]
    \node[precoloured] (u4) at (135:2cm){};
    \node[precoloured] (u3) at (105:2cm){};
    \node[precoloured] (u2) at (75:2cm){};
    \node[precoloured] (u1) at (45:2cm){};

    \node[vertex,label={below:$v_1\in B$}] (v1) at (0.75,0) {};
    \node[vertex,label={below:$v_2\in A$}] (v2) at (-0.75,0) {};
    \draw (v1) -- (u1) -- (u2) -- (u3) -- (u4) -- (v2);
    \draw (u2) -- (v1) -- (v2);
\end{tikzpicture}}%
    \hspace*{0.02\textwidth}
    \subcaptionbox{Exception of type~\ref{ex:A}.\label{fig:exceptionA}}[0.3\textwidth]{\begin{tikzpicture}[scale=1]
    \node[precoloured] (u1) at (120:2cm){};
    \node[precoloured] (u2) at (90:2cm){};
    \node[precoloured] (u3) at (60:2cm){};
    \node[vertex,label={below:$v\in B$}] (v) at (0,0) {};
    \draw (v) -- (u1) -- (u2) -- (u3) -- (v) -- (u2);
\end{tikzpicture}

\begin{tikzpicture}[scale=1]
    \node[precoloured] (u1) at (135:2cm){};
    \node[precoloured] (u2) at (105:2cm){};
    \node[precoloured] (u3) at (75:2cm){};
    \node[precoloured] (u4) at (45:2cm){};

    \node[vertex,label={below:$v\in A$}] (v) at (0,0) {};
    \draw (v) -- (u1) -- (u2) -- (u3) -- (u4) -- (v);
\end{tikzpicture}

\begin{tikzpicture}[xscale=-1,yscale=1]
    \node[precoloured] (u4) at (135:2cm){};
    \node[precoloured] (u3) at (105:2cm){};
    \node[precoloured] (u2) at (75:2cm){};
    \node[precoloured] (u1) at (45:2cm){};

    \node[vertex,label={below:$v_1\in B$}] (v1) at (0.75,0) {};
    \node[vertex,label={below:$v_2\in A$}] (v2) at (-0.75,0) {};
    \draw (v1) -- (u1) -- (u2) -- (u3) -- (u4) -- (v2);
    \draw (u2) -- (v1) -- (v2);
\end{tikzpicture}}%
    \hspace*{0.02\textwidth}%
    \subcaptionbox{Exception of type~\ref{ex:AB}.\label{fig:exceptionAB}}[0.3\textwidth]{\begin{tikzpicture}[scale=1]
    \node[precoloured] (u1) at (120:2cm){};
    \node[precoloured] (u2) at (90:2cm){};
    \node[precoloured] (u3) at (60:2cm){};
    \node[vertex,label={below:$v\in B$}] (v) at (0,0) {};
    \draw (v) -- (u1) -- (u2) -- (u3) -- (v) -- (u2);
\end{tikzpicture}

\begin{tikzpicture}[scale=1]
    \node[precoloured] (u1) at (135:2cm){};
    \node[precoloured] (u2) at (105:2cm){};
    \node[precoloured] (u3) at (75:2cm){};
    \node[precoloured] (u4) at (45:2cm){};

    \node[vertex,label={below:$v\in A$}] (v) at (0,0) {};
    \draw (v) -- (u1) -- (u2) -- (u3) -- (u4) -- (v);
\end{tikzpicture}

\begin{tikzpicture}[xscale=-1,yscale=1]
    \node[precoloured] (u4) at (135:2cm){};
    \node[precoloured] (u3) at (105:2cm){};
    \node[precoloured] (u2) at (75:2cm){};
    \node[precoloured] (u1) at (45:2cm){};

    \node[vertex,label={below:$v_1\in B$}] (v1) at (0.75,0) {};
    \node[vertex,label={below:$v_2\in A$}] (v2) at (-0.75,0) {};
    \draw (v1) -- (u1) -- (u2) -- (u3) -- (u4) -- (v2);
    \draw (u2) -- (v1) -- (v2);
\end{tikzpicture}}%
    \captionsetup{width=0.9\textwidth}
    \caption{The exceptional cases of Theorem~\ref{thm:inductive}. Vertices of the path $P$ are squares.\label{fig:exceptions}}
\end{figure}
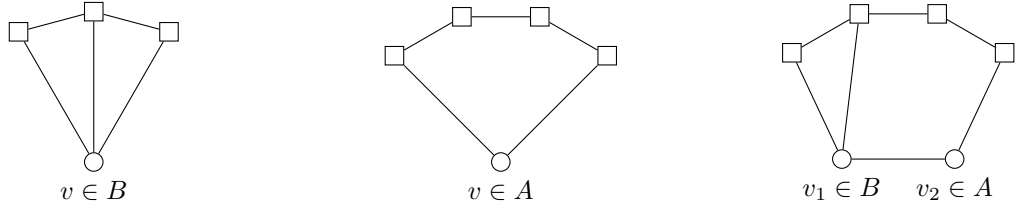

It is straightforward to check that with the prescribed function $f$, exceptional canvases $K$ are not weakly $f$-degenerate because the function $f_K$ in Definition \ref{def:degencanvas} assigns a negative value to some vertices. In fact, in the list-colouring setup there are list assignments $L$ to exceptional canvases with $|L(v)|=f(v)+1$ that do not admit a list colouring. 
We have the following observation for unexceptional canvases.

\begin{obs}\label{obs:canvasfunction}
    If $K=(G,P,A,B,f)$ is an unexceptional canvas then $f_K$ is nonnegative on $G-P$.
\end{obs}

The deduction of Theorem~\ref{thm:main} from Theorem~\ref{thm:inductive} is simple.

\begin{proof}[Proof of Theorem~\ref{thm:main}]
    Let $G$ be a plane graph and let $f$ be a local girth function for $G$. If $G$ has one vertex then it is $0$-degenerate and there is nothing more to prove. 
    In the case that $v(G)\ge 2$ we claim that if $P=uv$ is any two-vertex path in $\delta G$ then $G$ is weakly $f$-degenerate by a sequence of operations that begins $\del(u)$, $\del(v)$. 
    This is because $K=(G,P,\emptyset,\emptyset, f)$ is an unexceptional canvas and hence by Theorem~\ref{thm:inductive}, $K$ is weakly $f$-degenerate. Using Definition~\ref{def:degencanvas}, this means that the graph  $G-P$ is weakly $f_K$-degenerate for the function $f_K$ given by $f_K(w) = f(w) - |N(w)\cap P|$. But $f_K$ is the function one has starting with the pair $(G, f)$ and performing the sequence of operations $\del(u)$, $\del(v)$, which it is easy to see must be legal.
\end{proof}

\section{Properties of a minimum counterexample to Theorem \ref{thm:inductive}}\label{sec:properties}

A \emph{counterexample} is a canvas (see Definition~\ref{def:canvas}) for which Theorem \ref{thm:inductive} does not hold. A \emph{minimum counterexample} is a counterexample $K=(G,P,A,B,f)$ that is minimum with respect to $v(G)$, and subject to that, minimum with respect to $v(G-P)$, and subject to that minimum with respect to $\sum_{v\in V(G)}f(v)$. 
In what follows, we assume that Theorem \ref{thm:inductive} is false and let $K$ be a minimum counterexample. In this section, we establish some properties of $K$. 
Note that $K$ must be unexceptional as Theorem~\ref{thm:inductive} holds trivially for any exceptional canvases.  The following observation can be deduced easily from the definition of minimum counterexample and Lemma~\ref{lem:monotone}. 

\begin{obs}\label{obs:feq}
    Equality holds for the function $f$ in~\ref{canv:fCn3}--\ref{canv:fint}.
\end{obs}

\subsection{Chords and Short Cycles}
We will require the following definition.
\begin{definition}\label{def:seppath}
    We say that a path $H \subseteq G$ with endpoints in $V(\delta G)$ and no other vertices in $V(\delta G)$ \emph{separates $G$ into $G_1$ and $G_2$} if the following hold:
    \begin{itemize}
        \item $G_1$ and $G_2$ are connected subgraphs of $G$ that inherit the plane embedding of $G$,
        \item $G_1 \cap G_2 = H$ and $G_1 \cup G_2 = G$
        \item $V(G_1) \setminus V(G_2) \neq \emptyset$ and $V(G_2) \setminus V(G_1) \neq \emptyset$, and
        \item $H \subseteq \delta G_i$ for each $i \in \{1,2\}$. 
    \end{itemize}
    If the path $H$ separates $G$ into $G_1$ and $G_2$ and $e(H) = t$, we say that $H$ is a \emph{$t$-chord}. 
\end{definition}
Throughout, we adopt the convention that if $H$ separates $G$ into $G_1$ and $G_2$, then $|V(G_1) \cap V(P)| \geq |V(G_2) \cap V(P)|$.

\begin{definition}
    Let $H$ be a cycle with a fixed plane embedding. The \emph{interior} of $H$ is the unique non-infinite face bounded by $H$. 
\end{definition}

Our first structural results for a minimal counterexample show that 3- and 4-cycles in $G$ have no vertices in their interior. We also restrict the girth of vertices in any 5-cycles that do have vertices in their interior.

\begin{lemma}\label{lem:triangles}
    $G$ does not contain a cycle of length at most four with a vertex in its interior.
\end{lemma}
\begin{proof}
    Suppose for contradiction that $H$ is such a cycle. 
    Without loss of generality, $H$ is either a 3-cycle (and thus does not have a 1-chord) or a 4-cycle without a 1-chord in its interior. This is because a 4-cycle with a vertex and 1-chord in its interior contains a 3-cycle with vertex in its interior.
    Let $X$ be the set of vertices of $G$ embedded in the interior of $H$. 
    By Observation~\ref{obs:subcanvas}, $(G-X,P,A,B,f)$ is an unexceptional canvas (because $K$ is unexceptional), and by minimality it is weakly $f$-degenerate by a legal sequence of operations $\sigma$.
    Let $u$ be a vertex in $H$, and let $(G',f') = \del(G[X \cup V(H)],f,u)$. 
    We will define $P'$, $A'$ and $B'$ such that $K'= (G', P', A', B', f')$ is a canvas. 
    Let $P'=H-u$, let $A'$ be the set $\{v \in V(\delta G')\setminus V(P'): f'(v) \leq 1\}$, and let $B'=\emptyset$. 
    Note that this means $V(G')=X\cup V(P')$.
    We now show that $K'$ is a canvas. 
   By our choice of $H$, $v(P')\le 3$ and $V(P')$ induces a path in $G'$, and so $P'$ is acceptable in $G'$ and thus \ref{canv:acceptable} holds. 
    Since $B'=\emptyset$, both \ref{canv:B} and~\ref{canv:fB} hold.
    Note that $V(\delta G')\setminus V(P') \subseteq X\subseteq V(G-\delta G)$ and thus $A'$ only contains vertices of girth at least $5$. 
    This is because if $g(v)\le 4$ then $f(v)\ge 3$ by~\ref{canv:f} and hence $f'(v)\ge 2$.
    Also due to the fact that $V(\delta G')\setminus V(P') \subseteq V(G-\delta G)$, \ref{canv:f} gives that $f(v) \ge 2$ for all $v\in V(\delta G')\setminus V(P')$. 
    Then $A\cap A'=\emptyset$ and thus $f'(v) = 1$ for all $v \in A'$. 
    Moreover, $v\in A'$ only if $v\in N_G(u)$.  
    Thus, \ref{canv:fA} holds.  
    Since every vertex in $A'$ has girth at least five and is adjacent in $G$ to $u$, $A'$ is independent in $G'$ and so \ref{canv:A} holds. 
    \ref{canv:fint} holds because $K$ is a canvas. 
    Suppose now that $v \in V(\delta G') \setminus (V(P') \cup A' \cup B')$. If $g(v) \ne 3$, then $f(v) \leq 3$ and hence $f'(v) \geq 2$ and so \ref{canv:fCn3} holds. If $g(v) = 3$, then $f(v) \geq 4$ and so $f'(v) \geq 3$ and \ref{canv:fC3} holds. 
    Note that $K'$ must be unexceptional as $B'=\emptyset$ rules out the existence of exceptions of types~\ref{ex:B} and~\ref{ex:AB}, and $v(P')\le 3$ rules out type~\ref{ex:A}.
    
    By the minimality of $K$, the canvas $K'$ is weakly $f'$-degenerate by some legal sequence of operations. 
    That is, for the function $f'_{K'}$ given by $f'_{K'}(v) = f'(v)-|N(v)\cap P'|$, the graph $G'-P'$ is weakly $f'_{K'}$-degenerate by some legal sequence, say, $\sigma'$.

    But then $K$ is weakly $f$-degenerate by the sequence of operations formed by the concatenation $\sigma\sigma'$. 
    That is, with $f_K$ as in Definition~\ref{def:degencanvas}, we start from $(G-P, f_K)$ and perform the sequence of operations $\sigma$. By minimality, the sequence is legal starting with $(G-X-P, f_K)$ and since we verified that $K'=(G',P',A',B',f')$ is a canvas with $V(G'-P')=X$, it is also legal starting with $(G-P, f_K)$. 
    After performing the sequence $\sigma$ we have the graph-function pair $(G'-P', f'_{K'})$ because $G'-P'=G[X]$, and in performing $\sigma$ starting with the function $f_K$ we removed $V(H)$ in such a way the function value on any $v\in X$ falls by $|N(v)\cap V(H)|$. 
    But this is also true for $f'_{K'}$: starting from $f$ the function value falls by one on all neighbours of $u$ in forming $f'$ via $\del(G[X\cup V(H)], f, u)$ and forming $f'_{K'}$ from $f'$ accounts for the remaining decrease.
\end{proof}

The proof of Lemma~\ref{lem:triangles} is paradigmatic of most of our proofs. We take a (hypothetical) minimal counterexample, manually transform it into a smaller object using weak degeneracy operations, show that the smaller object is an unexceptional canvas, and obtain a contradiction by establishing that the concatenation of some sequences of operations is legal. 
In more complex examples, these latter two goals overlap somewhat: the fact that the smaller object is a canvas implies a significant number of the nonnegativity conditions for legality (Observation~\ref{obs:canvasfunction}). 


\begin{lemma}\label{lem:5cycles}
    If $H \subseteq G$ is a 5-cycle with a vertex in its interior, then $H$ contains neither a vertex of girth at least five (in $G$) nor two adjacent vertices of girth four (in $G$).
\end{lemma}
\begin{proof}
    Suppose not. Let $H = w_1w_2w_3w_4w_5w_1$, where either $g(w_3) \geq 5$, or $g(w_2) = g(w_3) = 4$.  Let $X$ be the set of vertices of $G$ embedded in the interior of $H$. By Observation~\ref{obs:subcanvas} and the minimality of $K$, $(G-X,P,A,B,f)$ is a canvas that is weakly $f$-degenerate. 

    Let $(G', f') = \del(G[X \cup V(H)], w_5)$. Let $P' = H-w_5$ and note that $v(P')=4$. 
    By Lemma \ref{lem:triangles}, $H$ does not have a chord in its interior and thus $V(P')$ induces a path in $G'$. 
    Let $A'$ be the set $\{v \in V(\delta G') \setminus V(H): f'(v) \leq 1\}$ and let $B' = \emptyset$.
    
    We claim that $K' = (G', P', A', B', f')$ is an unexceptional canvas. First, we show it is a canvas. 
    Since either $g(w_2) = g(w_3) = 4$ or $g(w_3) \geq 5$, $P'$ is an acceptable path in $G'$ and thus \ref{canv:acceptable} holds. 
    Since $B' = \emptyset$, both \ref{canv:B} and \ref{canv:fB} hold. 
    Note that for each vertex $v \in V(G'-P')$, we have that $f'(v) \geq f(v)-1$ and $f(v) \geq 2$. 
    Thus, every vertex $v$ in $A'$ has $f'(v) \geq 1$ and so \ref{canv:fA} holds. Moreover, if $v \in A'$, then $f(v) = 2$ and so $g(v) \geq 5$. 
    Since every vertex in $A'$ has girth at least five and is adjacent to $w_5$ in $G$, it follows that $A'$ is an independent set and so \ref{canv:A} holds. Property~\ref{canv:fint} holds because $K$ is a canvas. 
    Suppose now that $v \in V(\delta G') \setminus (V(P') \cup A' \cup B')$. If $g(v) \neq 3$, then $f(v) \leq 3$ and hence $f'(v) \geq 2$ and so \ref{canv:fCn3} holds. 
    If $g(v) = 3$, then $f(v) \geq 4$ and so $f'(v) \geq 3$ and \ref{canv:fC3} holds. 
    
    Finally, we claim that $K'$ is unexceptional. It is not an exceptional canvas of type~\ref{ex:B} since $v(P') = 4$. Moreover, no vertex of $A'$ is adjacent to an endpoint of $P'$ since every vertex $v \in A'$ is adjacent in $G$ to $w_5$ and has $g(v) \geq 5$, and so $K'$ is not an exceptional canvas of type~\ref{ex:A} or~\ref{ex:AB}.

    By the minimality of $K$, $K'$ is weakly $f'$-degenerate. But then $K$ is weakly $f$-degenerate by the concatenation of the sequences of operations that we have for $(G-X,P,A,B,f)$ and $K'$ exactly as in the proof of Lemma~\ref{lem:triangles}. This gives the desired contradiction.
    \end{proof}


\begin{lemma}\label{lem:nolowdegs}
$G$ does not contain a vertex $v \in V(G) \setminus V(P)$ with $f(v) \geq \deg(v)$.    
\end{lemma}
\begin{proof}
    Suppose not. 
    Since $K$ is a minimum counterexample, $(G-v,P,A,B,f)$ (which is an unexceptional canvas by Observation~\ref{obs:subcanvas}) is weakly $f$-degenerate. 
    That is, there is a legal sequence of operations $\sigma$ that reaches the empty graph starting with $(G-v-P, f_K)$ (where $f_K$ is as in Definition~\ref{def:degencanvas}). 
    Since $f(v) \geq \deg(v)$, the sequence $\sigma$ is legal when starting with $(G-P, f_K)$, and at the end we have the single vertex $v$ with a nonnegative function value. Thus, we can finally remove $v$ with a legal $\del(v)$ operation.
\end{proof}

\begin{lemma}\label{lem:2-conn}
    $G$ is 2-connected.
\end{lemma}
\begin{proof}
Suppose not. If $G$ is disconnected, then since $K$ is a minimum counterexample to Theorem \ref{thm:inductive} we obtain a contradiction by invoking Theorem \ref{thm:inductive} on each component of $G$ separately. Thus, we may assume that $G$ is connected and that $\delta G$ has a 0-chord (i.e.\ a cutvertex) $u$ that separates $G$ into two graphs $G_1$ and $G_2$. To be precise, the case $v(H)=1$ of Definition~\ref{def:seppath} gives that $G_1$ and $G_2$ are induced subgraphs of $G$ such that $V(G_1)\cap V(G_2)=\{u\}$ and $V(G_1)\cup V(G_2)=V(G)$.
Let $P_1$ be the restriction of $P$ to $G_1$.  Note that since $K$ is an unexceptional canvas, by Observation \ref{obs:subcanvas} so too is $(G_1, P_1, A, B, f)$. Since $K$ is a minimum counterexample, $(G_1, P_1, A, B, f)$ is weakly $f$-degenerate. 

If $V(P) \cap V(G_2) = \emptyset$, let $P_2 = u$. Otherwise, let $P_2$ be the restriction of $P$ to $G_2$. Note that $u\in V(P_2)$ in either case. 
We have that $(G_2, P_2, A, B, f)$ is an unexceptional canvas since by convention $|V(P_1) \cap V(P)| \geq |V(P_2) \cap V(P)|$ and thus $v(P_2) \leq 2$. 
Again, by the minimality of $K$ we have that $(G_2, P_2, A, B, f)$ is weakly $f$-degenerate. This is a contradiction, as then $G$ is weakly $f$-degenerate.
\end{proof}

\noindent
Lemma \ref{lem:2-conn} implies that $\delta G$ is a cycle.

\begin{lemma}\label{lem:1chord}
    $\delta G$ does not have a 1-chord.
\end{lemma}
\begin{proof}
Suppose not, and let $uv$ be a 1-chord that separates $G$ into $G_1$ and $G_2$. Note that $\delta G \neq P$, since if $P$ is a cycle, it is an induced cycle by the definition of acceptable (Definition~\ref{def:accpath}). Thus we may assume that $P$ is an acceptable path. Write $P_1 = P \cap G_1$ and $P_2 = P \cap G_2$.   
Since $K$ is unexceptional, by Observation \ref{obs:subcanvas} $K_1=(G_1, P_1, A, B, f)$ is an unexceptional canvas.

If neither $u$ nor $v$ is an internal vertex of $P$, then since by convention $v(P_1) \geq v(P_2)$, we have that $P_1 = P$. By the minimality of $K$, $K_1$ is weakly $f$-degenerate. Turning to $G_2$, note that $K_2=(G_2, uv, A, B, f)$ is a canvas (this follows easily from the fact that $uv$ is acceptable and $K$ is a canvas), and that it is unexceptional since $|\{u,v\}| = 2$. By the minimality of $K$, $K_2$ is weakly $f$-degenerate.
But then $K$ is weakly $f$-degenerate\textemdash a contradiction.
Thus, $G$ does not contain a 1-chord where neither endpoint is an internal vertex of $P$.

Hence, we may assume without loss of generality that $u$ is an internal vertex of $P$, and therefore that $v(P) \geq 3$. Subject to the convention that $v(P_1) \geq v(P_2)$, we may assume that the 1-chord $uv$ is chosen to minimize $v(G_2)$. This ensures that $\delta G_2$ does not have a 1-chord. Since $u$ is an internal vertex of $P$, the above conventions give that $v(P_1)\in\{2,3\}$ and $v(P_2) = 2$.

Suppose that $\{v\}\cup V(P_2)$ induces a path in $G_2$. The same argument as above applies and reaches a contradiction unless the canvas $K_2'' = (G_2, v+P_2, A, B, f)$ is exceptional. But then it must be that $(G_2, v+P_2, A, B, f)$ is an exceptional canvas of type~\ref{ex:B}, since $|\{v\}\cup V(P_2)|=3$. But then $\delta G_2$ has a 1-chord, contradicting our initial choice of $1$-chord $uv$. 
Thus, we may assume that $\{v\}\cup V(P_2)$ does not induce a path in $G_2$, which means that $v$ is adjacent to both vertices of $P_2$. 
Let $w \ne u$ be the second vertex of $P_2$. 
By our choice of 1-chord $uv$, we have that $v(\delta G_2) = 3$. If not, then $vw$ would be a 1-chord of $G$ with a smaller $v(G_2)$.
Given this structure, we break into four cases pictured in Figure~\ref{fig:1chord}.

\textbf{Case 1.} Suppose that $v$ is adjacent to at least two vertices in $P_1$. In particular, this implies that $v+P_1$ is not an induced path. Since $v$ is also adjacent to $u$ and $w$ the girth of $u$ is three, and then the fact that $P$ is an acceptable path gives $v(P) = 3$. 
Let $P = yuw$. Since $K$ is unexceptional, $v \not \in B$. 
Since we already showed that $\delta G$ does not contain a 1-chord where neither endpoint is an internal vertex of $P$, any remaining 1-chords include $u$.  
Thus $\delta G = yuwvy$. 
This is depicted in Figure~\ref{fig:1chord_case1}. By Lemma~\ref{lem:triangles}, we have that the triangles $yuvy$ and $uwvu$ do not have vertices in their interiors, and thus $V(G) = V(\delta G)$. Since $v \notin B$ and $g(v) = 3$, it follows that $f(v) \geq 3$. Since $\deg(v)=3$ this contradicts Lemma~\ref{lem:nolowdegs}.

Thus, in the remaining cases $u$ is the only neighbour of $v$ in $P_1$, and so $v + P_1$ is an induced path in $G_1$. Let $P_1'=v+P_1$.

If $K_1'=(G_1,P_1',A,B,f)$ is an unexceptional canvas, then by the minimality of $K$, $K_1'$ is weakly $f$-degenerate. But then $K$ is weakly $f$-degenerate, a contradiction. Thus, either $K_1'$ is not a canvas, or it is an exceptional canvas. First we show $K_1'$ is a canvas. Note by Observation \ref{obs:subcanvas}, it suffices to show $P_1'$ is acceptable in $G_1$.  By the fact that $P$ is acceptable in $G$, either $P_1'$ has three vertices, or $P_1'$ has four vertices.   If $P_1'$ has four vertices, then since the internal vertices in $P_1'$ are the same as those in $P$ and $P$ is acceptable in $G$, it follows that $P_1'$ is acceptable in $G_1$. Thus, $K_1'$ is a canvas, and so it is exceptional. We break into cases depending on the length of $P_1'$. Recall that in the remaining cases, $u$ is the only neighbour of $v$ in $P_1$.

\textbf{Case 2.} $v(P_1')=3$. Let $P = yuw$. Then $K_1'$ is an exceptional canvas of type~\ref{ex:B}, and so there exists a vertex $v' \in V(\delta G)$ with $v' \in B$ and $N(v') \supseteq \{v, y, u\}$. 
Since we already argued that $\delta G$ does not contain a 1-chord whose endpoints do not include $u$, we have $\delta G = yuwvv'y$. This is depicted in Figure~\ref{fig:1chord_case2}.
By Lemma~\ref{lem:triangles}, the triangles $yuv'y$, $v'uvv'$, and $vuwv$ do not have vertices in their interior. 
Thus, $V(G) = V(\delta G) = \{u,v,v',w,y\}$. Since $B$ is an independent set in $G$ and $g(v) = 3$, we have that $f(v) = 3$. Since $\deg(v) =3$, this contradicts Lemma \ref{lem:nolowdegs}.

\textbf{Case 3.} $v(P_1')=4$. Let $P = xyuw$ and note that $g(y) \geq 5$ since $P$ is acceptable and $g(u)=3$. 
Then $K_1'$ is an exceptional canvas of type~\ref{ex:A} (\textbf{Case 3a}) or~\ref{ex:AB} (\textbf{Case 3b}); see Figures~\ref{fig:1chord_case3a} and~\ref{fig:1chord_case3b}.
In either case, by our previous analysis $\delta G$ does not contain a 1-chord whose endpoints do not include $u$ or $y$. Since $K$ is unexceptional and $B$ is an independent set, in either case we have that $f(v) = 3$. By Lemmas \ref{lem:triangles} and \ref{lem:5cycles}, $V(G) = V(\delta G_1) \cup \{w\}$, and thus $\deg(v)=3$ which contradicts Lemma~\ref{lem:nolowdegs}.
\end{proof}

\begin{figure}[ht]\centering
    \subcaptionbox{Case 1: $v$ has at least two neighbours in $P_1$.\label{fig:1chord_case1}}[0.35\textwidth]{\begin{tikzpicture}[scale=0.75]

    \node[precoloured,label={above:$y$}] (y) at (-1.5,0) {};
    \node[precoloured,label={above:$u$}] (u) at (0,0){};
    \node[precoloured,label={above:$w$}] (w) at (1.5,0){};
    \node[vertex,label={below:$v\notin B$}] (v) at (0,-2) {};

    \draw (y) -- (u) -- (w);
    \draw (v) -- (y);
    \draw[thickedge] (v) -- (u);
	\draw (v) -- (w);
    

    \draw [decorate,decoration={brace,amplitude=5pt,raise=4ex}] (u.west) -- (w.east) node[midway,above=5ex]{$P_2$};

    \draw [decorate,decoration={brace,amplitude=5pt,raise=6ex}] (y.west) -- (u.east) node[midway,above=7ex]{$P_1$};    

\end{tikzpicture}

\begin{tikzpicture}[scale=0.75]

    \node[precoloured,label={above:$y$}] (y) at (-1.5,0) {};
    \node[precoloured,label={above:$u$}] (u) at (0,0){};
    \node[precoloured,label={above:$w$}] (w) at (1.5,0){};
    \node[vertex,label={below:$v$}] (v) at (0,-2) {};

    \node[vertex,label={left:$v'\in B$}] (v') at (-1.41,-1.41) {};

    \draw (y) -- (u) -- (w);
    \draw[thickedge] (v) -- (u);
	\draw (v) -- (w);
	
	\draw (v') -- (y);
	\draw (v') -- (u);
	\draw (v') -- (v);

    \draw [decorate,decoration={brace,amplitude=5pt,raise=4ex}] (u.west) -- (w.east) node[midway,above=5ex]{$P_2$};

    \draw [decorate,decoration={brace,amplitude=5pt,raise=6ex}] (y.west) -- (u.east) node[midway,above=7ex]{$P_1$};    

\end{tikzpicture}

\begin{tikzpicture}[scale=0.75]

    \node[precoloured,label={above:$x$}] (x) at (-3,0) {};
    \node[precoloured,label={above:$y$}] (y) at (-1.5,0) {};
    \node[precoloured,label={above:$u$}] (u) at (0,0){};
    \node[precoloured,label={above:$w$}] (w) at (1.5,0){};
    \node[vertex,label={below:$v$}] (v) at (0,-2.5) {};
    
    \node[vertex,label={left:$v'\in A$}] (v') at (-2,-2.5) {};

    \draw (x) -- (y) -- (u) -- (w);
    \draw[thickedge] (v) -- (u);
    \draw (v) -- (w);
    \draw (x) -- (v') -- (v);
    
    \draw [decorate,decoration={brace,amplitude=5pt,raise=4ex}]
    (u.west) -- (w.east) node[midway,above=5ex]{$P_2$};
    
    \draw [decorate,decoration={brace,amplitude=5pt,raise=6ex}]
    (x.west) -- (u.east) node[midway,above=7ex]{$P_1$};    

\end{tikzpicture}

\begin{tikzpicture}[scale=0.75]

    \node[precoloured,label={above:$x$}] (x) at (-3,0) {};
    \node[precoloured,label={above:$y$}] (y) at (-1.5,0) {};
    \node[precoloured,label={above:$u$}] (u) at (0,0){};
    \node[precoloured,label={above:$w$}] (w) at (1.5,0){};
    \node[vertex,label={below:$v$}] (v) at (0,-2.5) {};
    
    \node[vertex,label={left:$v''\in A$}] (v'') at (-3,-1) {};
    \node[vertex,label={left:$v'\in B$}] (v') at (-1.5,-2.5) {};

    \draw (x) -- (y) -- (u) -- (w);
    \draw[thickedge] (v) -- (u);
    \draw (v) -- (w);
    
    \draw (v) -- (v') -- (u);
    \draw (v') -- (v'') -- (x);
    
    \draw [decorate,decoration={brace,amplitude=5pt,raise=4ex}]
    (u.west) -- (w.east) node[midway,above=5ex]{$P_2$};
    
    \draw [decorate,decoration={brace,amplitude=5pt,raise=6ex}]
    (x.west) -- (u.east) node[midway,above=7ex]{$P_1$};    

\end{tikzpicture}}%
    \hspace*{0.1\textwidth}
    \subcaptionbox{Case 2: $v$ has one neighbour in $P_1$ and $v(P_1')=3$.\label{fig:1chord_case2}}[0.35\textwidth]{\begin{tikzpicture}[scale=0.75]

    \node[precoloured,label={above:$y$}] (y) at (-1.5,0) {};
    \node[precoloured,label={above:$u$}] (u) at (0,0){};
    \node[precoloured,label={above:$w$}] (w) at (1.5,0){};
    \node[vertex,label={below:$v\notin B$}] (v) at (0,-2) {};

    \draw (y) -- (u) -- (w);
    \draw (v) -- (y);
    \draw[thickedge] (v) -- (u);
	\draw (v) -- (w);
    

    \draw [decorate,decoration={brace,amplitude=5pt,raise=4ex}] (u.west) -- (w.east) node[midway,above=5ex]{$P_2$};

    \draw [decorate,decoration={brace,amplitude=5pt,raise=6ex}] (y.west) -- (u.east) node[midway,above=7ex]{$P_1$};    

\end{tikzpicture}

\begin{tikzpicture}[scale=0.75]

    \node[precoloured,label={above:$y$}] (y) at (-1.5,0) {};
    \node[precoloured,label={above:$u$}] (u) at (0,0){};
    \node[precoloured,label={above:$w$}] (w) at (1.5,0){};
    \node[vertex,label={below:$v$}] (v) at (0,-2) {};

    \node[vertex,label={left:$v'\in B$}] (v') at (-1.41,-1.41) {};

    \draw (y) -- (u) -- (w);
    \draw[thickedge] (v) -- (u);
	\draw (v) -- (w);
	
	\draw (v') -- (y);
	\draw (v') -- (u);
	\draw (v') -- (v);

    \draw [decorate,decoration={brace,amplitude=5pt,raise=4ex}] (u.west) -- (w.east) node[midway,above=5ex]{$P_2$};

    \draw [decorate,decoration={brace,amplitude=5pt,raise=6ex}] (y.west) -- (u.east) node[midway,above=7ex]{$P_1$};    

\end{tikzpicture}

\begin{tikzpicture}[scale=0.75]

    \node[precoloured,label={above:$x$}] (x) at (-3,0) {};
    \node[precoloured,label={above:$y$}] (y) at (-1.5,0) {};
    \node[precoloured,label={above:$u$}] (u) at (0,0){};
    \node[precoloured,label={above:$w$}] (w) at (1.5,0){};
    \node[vertex,label={below:$v$}] (v) at (0,-2.5) {};
    
    \node[vertex,label={left:$v'\in A$}] (v') at (-2,-2.5) {};

    \draw (x) -- (y) -- (u) -- (w);
    \draw[thickedge] (v) -- (u);
    \draw (v) -- (w);
    \draw (x) -- (v') -- (v);
    
    \draw [decorate,decoration={brace,amplitude=5pt,raise=4ex}]
    (u.west) -- (w.east) node[midway,above=5ex]{$P_2$};
    
    \draw [decorate,decoration={brace,amplitude=5pt,raise=6ex}]
    (x.west) -- (u.east) node[midway,above=7ex]{$P_1$};    

\end{tikzpicture}

\begin{tikzpicture}[scale=0.75]

    \node[precoloured,label={above:$x$}] (x) at (-3,0) {};
    \node[precoloured,label={above:$y$}] (y) at (-1.5,0) {};
    \node[precoloured,label={above:$u$}] (u) at (0,0){};
    \node[precoloured,label={above:$w$}] (w) at (1.5,0){};
    \node[vertex,label={below:$v$}] (v) at (0,-2.5) {};
    
    \node[vertex,label={left:$v''\in A$}] (v'') at (-3,-1) {};
    \node[vertex,label={left:$v'\in B$}] (v') at (-1.5,-2.5) {};

    \draw (x) -- (y) -- (u) -- (w);
    \draw[thickedge] (v) -- (u);
    \draw (v) -- (w);
    
    \draw (v) -- (v') -- (u);
    \draw (v') -- (v'') -- (x);
    
    \draw [decorate,decoration={brace,amplitude=5pt,raise=4ex}]
    (u.west) -- (w.east) node[midway,above=5ex]{$P_2$};
    
    \draw [decorate,decoration={brace,amplitude=5pt,raise=6ex}]
    (x.west) -- (u.east) node[midway,above=7ex]{$P_1$};    

\end{tikzpicture}}%
    \\\vspace*{0.05\textwidth}
    \subcaptionbox{Case 3a: $v$ has one neighbour in $P_1$ and $v(P_1')=4$.\label{fig:1chord_case3a}}[0.35\textwidth]{\begin{tikzpicture}[scale=0.75]

    \node[precoloured,label={above:$y$}] (y) at (-1.5,0) {};
    \node[precoloured,label={above:$u$}] (u) at (0,0){};
    \node[precoloured,label={above:$w$}] (w) at (1.5,0){};
    \node[vertex,label={below:$v\notin B$}] (v) at (0,-2) {};

    \draw (y) -- (u) -- (w);
    \draw (v) -- (y);
    \draw[thickedge] (v) -- (u);
	\draw (v) -- (w);
    

    \draw [decorate,decoration={brace,amplitude=5pt,raise=4ex}] (u.west) -- (w.east) node[midway,above=5ex]{$P_2$};

    \draw [decorate,decoration={brace,amplitude=5pt,raise=6ex}] (y.west) -- (u.east) node[midway,above=7ex]{$P_1$};    

\end{tikzpicture}

\begin{tikzpicture}[scale=0.75]

    \node[precoloured,label={above:$y$}] (y) at (-1.5,0) {};
    \node[precoloured,label={above:$u$}] (u) at (0,0){};
    \node[precoloured,label={above:$w$}] (w) at (1.5,0){};
    \node[vertex,label={below:$v$}] (v) at (0,-2) {};

    \node[vertex,label={left:$v'\in B$}] (v') at (-1.41,-1.41) {};

    \draw (y) -- (u) -- (w);
    \draw[thickedge] (v) -- (u);
	\draw (v) -- (w);
	
	\draw (v') -- (y);
	\draw (v') -- (u);
	\draw (v') -- (v);

    \draw [decorate,decoration={brace,amplitude=5pt,raise=4ex}] (u.west) -- (w.east) node[midway,above=5ex]{$P_2$};

    \draw [decorate,decoration={brace,amplitude=5pt,raise=6ex}] (y.west) -- (u.east) node[midway,above=7ex]{$P_1$};    

\end{tikzpicture}

\begin{tikzpicture}[scale=0.75]

    \node[precoloured,label={above:$x$}] (x) at (-3,0) {};
    \node[precoloured,label={above:$y$}] (y) at (-1.5,0) {};
    \node[precoloured,label={above:$u$}] (u) at (0,0){};
    \node[precoloured,label={above:$w$}] (w) at (1.5,0){};
    \node[vertex,label={below:$v$}] (v) at (0,-2.5) {};
    
    \node[vertex,label={left:$v'\in A$}] (v') at (-2,-2.5) {};

    \draw (x) -- (y) -- (u) -- (w);
    \draw[thickedge] (v) -- (u);
    \draw (v) -- (w);
    \draw (x) -- (v') -- (v);
    
    \draw [decorate,decoration={brace,amplitude=5pt,raise=4ex}]
    (u.west) -- (w.east) node[midway,above=5ex]{$P_2$};
    
    \draw [decorate,decoration={brace,amplitude=5pt,raise=6ex}]
    (x.west) -- (u.east) node[midway,above=7ex]{$P_1$};    

\end{tikzpicture}

\begin{tikzpicture}[scale=0.75]

    \node[precoloured,label={above:$x$}] (x) at (-3,0) {};
    \node[precoloured,label={above:$y$}] (y) at (-1.5,0) {};
    \node[precoloured,label={above:$u$}] (u) at (0,0){};
    \node[precoloured,label={above:$w$}] (w) at (1.5,0){};
    \node[vertex,label={below:$v$}] (v) at (0,-2.5) {};
    
    \node[vertex,label={left:$v''\in A$}] (v'') at (-3,-1) {};
    \node[vertex,label={left:$v'\in B$}] (v') at (-1.5,-2.5) {};

    \draw (x) -- (y) -- (u) -- (w);
    \draw[thickedge] (v) -- (u);
    \draw (v) -- (w);
    
    \draw (v) -- (v') -- (u);
    \draw (v') -- (v'') -- (x);
    
    \draw [decorate,decoration={brace,amplitude=5pt,raise=4ex}]
    (u.west) -- (w.east) node[midway,above=5ex]{$P_2$};
    
    \draw [decorate,decoration={brace,amplitude=5pt,raise=6ex}]
    (x.west) -- (u.east) node[midway,above=7ex]{$P_1$};    

\end{tikzpicture}}%
    \hspace*{0.1\textwidth}%
    \subcaptionbox{Case 3b: $v$ has one neighbour in $P_1$ and $v(P_1')=4$.\label{fig:1chord_case3b}}[0.35\textwidth]{\begin{tikzpicture}[scale=0.75]

    \node[precoloured,label={above:$y$}] (y) at (-1.5,0) {};
    \node[precoloured,label={above:$u$}] (u) at (0,0){};
    \node[precoloured,label={above:$w$}] (w) at (1.5,0){};
    \node[vertex,label={below:$v\notin B$}] (v) at (0,-2) {};

    \draw (y) -- (u) -- (w);
    \draw (v) -- (y);
    \draw[thickedge] (v) -- (u);
	\draw (v) -- (w);
    

    \draw [decorate,decoration={brace,amplitude=5pt,raise=4ex}] (u.west) -- (w.east) node[midway,above=5ex]{$P_2$};

    \draw [decorate,decoration={brace,amplitude=5pt,raise=6ex}] (y.west) -- (u.east) node[midway,above=7ex]{$P_1$};    

\end{tikzpicture}

\begin{tikzpicture}[scale=0.75]

    \node[precoloured,label={above:$y$}] (y) at (-1.5,0) {};
    \node[precoloured,label={above:$u$}] (u) at (0,0){};
    \node[precoloured,label={above:$w$}] (w) at (1.5,0){};
    \node[vertex,label={below:$v$}] (v) at (0,-2) {};

    \node[vertex,label={left:$v'\in B$}] (v') at (-1.41,-1.41) {};

    \draw (y) -- (u) -- (w);
    \draw[thickedge] (v) -- (u);
	\draw (v) -- (w);
	
	\draw (v') -- (y);
	\draw (v') -- (u);
	\draw (v') -- (v);

    \draw [decorate,decoration={brace,amplitude=5pt,raise=4ex}] (u.west) -- (w.east) node[midway,above=5ex]{$P_2$};

    \draw [decorate,decoration={brace,amplitude=5pt,raise=6ex}] (y.west) -- (u.east) node[midway,above=7ex]{$P_1$};    

\end{tikzpicture}

\begin{tikzpicture}[scale=0.75]

    \node[precoloured,label={above:$x$}] (x) at (-3,0) {};
    \node[precoloured,label={above:$y$}] (y) at (-1.5,0) {};
    \node[precoloured,label={above:$u$}] (u) at (0,0){};
    \node[precoloured,label={above:$w$}] (w) at (1.5,0){};
    \node[vertex,label={below:$v$}] (v) at (0,-2.5) {};
    
    \node[vertex,label={left:$v'\in A$}] (v') at (-2,-2.5) {};

    \draw (x) -- (y) -- (u) -- (w);
    \draw[thickedge] (v) -- (u);
    \draw (v) -- (w);
    \draw (x) -- (v') -- (v);
    
    \draw [decorate,decoration={brace,amplitude=5pt,raise=4ex}]
    (u.west) -- (w.east) node[midway,above=5ex]{$P_2$};
    
    \draw [decorate,decoration={brace,amplitude=5pt,raise=6ex}]
    (x.west) -- (u.east) node[midway,above=7ex]{$P_1$};    

\end{tikzpicture}

\begin{tikzpicture}[scale=0.75]

    \node[precoloured,label={above:$x$}] (x) at (-3,0) {};
    \node[precoloured,label={above:$y$}] (y) at (-1.5,0) {};
    \node[precoloured,label={above:$u$}] (u) at (0,0){};
    \node[precoloured,label={above:$w$}] (w) at (1.5,0){};
    \node[vertex,label={below:$v$}] (v) at (0,-2.5) {};
    
    \node[vertex,label={left:$v''\in A$}] (v'') at (-3,-1) {};
    \node[vertex,label={left:$v'\in B$}] (v') at (-1.5,-2.5) {};

    \draw (x) -- (y) -- (u) -- (w);
    \draw[thickedge] (v) -- (u);
    \draw (v) -- (w);
    
    \draw (v) -- (v') -- (u);
    \draw (v') -- (v'') -- (x);
    
    \draw [decorate,decoration={brace,amplitude=5pt,raise=4ex}]
    (u.west) -- (w.east) node[midway,above=5ex]{$P_2$};
    
    \draw [decorate,decoration={brace,amplitude=5pt,raise=6ex}]
    (x.west) -- (u.east) node[midway,above=7ex]{$P_1$};    

\end{tikzpicture}}%
    \captionsetup{width=0.9\textwidth}
    \caption{Cases for Lemma~\ref{lem:1chord} in which we argue that the 1-chord $uv$ (the thick edge) cannot be in a minimal counterexample.\label{fig:1chord}}
\end{figure}

\begin{lemma}\label{lem:Pispath}
    $P$ is an induced path, and $V(G) \setminus V(\delta G)$ is nonempty.
\end{lemma}
\begin{proof}
By the definition of acceptable (Definition~\ref{def:accpath}), either $P$ is an induced path, or $P$ is a cycle. Since $v(P) \leq 4$, if $P$ is a cycle, then by Lemma \ref{lem:triangles}, $G-P$ is empty and so $K$ is trivially weakly $f$-degenerate, a contradiction.
Thus, $P$ is an induced path. Suppose next that $V(G) \setminus V(\delta G)$ is empty. 
Let $P' = \delta G - P$ and for $v\in V(P')$ let $f_K(v)=f(v)-|N(v)\cap P|$ be the function in Definition~\ref{def:degencanvas}. By Lemma \ref{lem:1chord}, $V(P')$ induces a path in $G$, and only the endpoints of $P'$ are adjacent to vertices in $P$.

If $v(P') = 1$, say $P' = u$, then $u\notin A$ since $K$ is unexceptional, and so $f_K(u) \geq 0$. Then $G-P$ is weakly $f_K$-degenerate, contradicting that $K$ is not weakly $f$-degenerate. 
Thus, we may assume that $v(P') \geq 2$. 
Let $u$ and $w$ be the endpoints of $P'$. Since $\delta G$ is chordless, $f_K(v) = f(v) \geq 1$ for all $v \in V(P') \setminus \{u,w\}$. If $f_K(u) = f_K(w) = 0$, then both $u$ and $w$ are in $A$ and so $u$ and $w$ are non-adjacent.  

By first removing the independent set $\{u,w\} \cap A$ (in any order) via $\del$ and then removing the remaining vertices of $P'$ via $\del$ in the order in which they appear along $P'$, we obtain that $P'=G-P$ is weakly $f_K$-degenerate. 
This contradicts the assumption that $K$ is not weakly $f$-degenerate. 
\end{proof}

\begin{lemma}\label{lem:Bmax}
    There does not exist a path $xyz \subseteq \delta G-P$ with $g(x) = g(y) = g(z) = 3$ and $\{x,y,z\} \cap B = \emptyset$.
\end{lemma}
\begin{proof}
    Suppose that there does exist such a path. Then by \ref{canv:fC3} we have $f(y)\ge 3$. Let $f'(y) := 2$, and $f'(v) := f(v)$ for all $v \in V(G) \setminus \{y\}$. 
    Since $B$ is independent and $\delta G$ has no 1-chord (Lemma~\ref{lem:1chord}), $B\cup\{y\}$ is independent and thus $K'=(G, P, A, B \cup \{y\}, f')$ is a canvas. 
    By Lemma \ref{lem:1chord}, there is no edge from $y$ to $V(P)$, and so $K'$ is not exceptional. As $K$ was chosen to minimize $\sum_{v \in V(G)} f(v)$, it follows that $K'$ is not a counterexample to Theorem \ref{thm:inductive} and thus $K'$ is weakly $f'$-degenerate. But then by Lemma~\ref{lem:monotone} the canvas $K$ is weakly $f$-degenerate, a contradiction.
\end{proof} 
In particular, Lemma \ref{lem:Bmax} implies that $\delta G$ does not contain three consecutive vertices $v$ with $f(v) = 4$. Moreover, a nearly identical argument shows the following.
We omit the proof.

\begin{lemma}\label{lem:Amax}
    There does not exist a path $xyz \subseteq \delta G-P$ with $g(v) \geq 5$ for all $v \in \{x,y,z\}$ and $\{x,y,z\} \cap A = \emptyset$.
\end{lemma}

Though we cannot entirely rule out the existence of 2-chords, the lemma below establishes some of their neighbouring structure. 
We show examples in Figure~\ref{fig:2chords}.

\begin{lemma}\label{lem:2chords}
If $xyz$ is a 2-chord of $\delta G$ that separates $G$ into $G_1$ and $G_2$, then either one of $x$ and $z$ is an internal vertex of $P$, or $(G_2, xyz, A, B, f)$ is an exceptional canvas of type~\ref{ex:B}.
In the latter case, note that there is a vertex $v\in B$ adjacent to each vertex in the 2-chord, meaning that $g(v)=3$ and $f(v)=2$.
\end{lemma}
\begin{proof}
    Recall that we assume that $|V(G_1)\cap V(P)|\ge |V(G_2)\cap V(P)|$.
    Suppose for contradiction that the lemma does not hold. By Observation~\ref{obs:subcanvas} and the minimality of $K$, $(G_1, P, A, B, f)$ is weakly $f$-degenerate. 
    Note that $\{x,y,z\}$ induces a path in $G$ since $\delta G$ has no 1-chord by Lemma~\ref{lem:1chord}. Thus, $xyz$ is an acceptable path and so $K_2 = (G_2, xyz, A, B, f)$ is a canvas. We have that $K_2$ is unexceptional by assumption, and so since $K$ is a minimum counterexample, $K_2$ is weakly $f$-degenerate. But then $G$ is weakly $f$-degenerate, a contradiction.
\end{proof}

\begin{figure}[ht]\centering
    \subcaptionbox{The 2-chord contains an internal vertex of $P$.\label{fig:2chorda}}[0.4\textwidth]{\begin{tikzpicture}[scale=1]
    \node[precoloured] (u1) at (0,0) {};
    \node[precoloured,label={above:$x$}] (x) at (1,0) {};
    \node[precoloured] (u3) at (2,0) {};
    \node[vertex,label={left:$y$}] (y) at (1,-1) {};
    \node[vertex,label={below:$z$}] (z) at (1,-2) {};

    \coordinate (d1) at (-0.5, -0.5);
    \coordinate (d2) at (2.5, -0.5);
    \coordinate (d3) at (0, -2);
    \coordinate (d4) at (2, -2);
    \coordinate (d5) at (-0.5, -1.5);
    \coordinate (d6) at (2.5, -1.5);
    \node (d7) at (-0.5, -1) {\rvdots};
    \node (d8) at (2.5, -1) {\rvdots};
    
    \draw (u1) -- (x) -- (u3);
    \draw[thickedge] (x) -- (y) -- (z);
    \draw (d3) -- (z) -- (d4);
    \begin{scope}[on background layer]
        \draw (d1) to[bend left=45] (u1.center);
        \draw (d2) to[bend right=45] (u3.center);
    \end{scope}
    \draw (d3) to[bend left=45] (d5);
    \draw (d4) to[bend right=45] (d6); 
    
\end{tikzpicture}

\begin{tikzpicture}[scale=1]
    \node[precoloured] (u1) at (0,0) {};
    \node (u2) at (1,0) {\ldots};
    \node[precoloured,label={above:$x$}] (x) at (2,0) {};
    \node[vertex,label={left:$y$}] (y) at (2,-1) {};
    \node[vertex,label={below:$z$}] (z) at (2,-2) {};

    \node[vertex,label={right:$v\in B$}] (v) at (3, -1) {};

    \node (d1) at (1,-2) {\ldots};
    \node (d2) at (-0.5,-1) {\rvdots};
    \coordinate (d3) at (-0.5, -0.5);

    \draw (u1) -- (u2) -- (x);
    \draw[thickedge] (x) -- (y) -- (z);
    \draw (v) -- (x);
    \draw (v) -- (y);
    \draw (v) -- (z);
    \draw (d1) -- (z);
    \begin{scope}[on background layer]
        \draw (u1.center) to[bend right=45] (d3.center);
    \end{scope}

\end{tikzpicture}}%
    \hspace*{0.1\textwidth}
    \subcaptionbox{An exceptional canvas of type~\ref{ex:B} lies on one side of the 2-chord.\label{fig:2chordb}}[0.4\textwidth]{\begin{tikzpicture}[scale=1]
    \node[precoloured] (u1) at (0,0) {};
    \node[precoloured,label={above:$x$}] (x) at (1,0) {};
    \node[precoloured] (u3) at (2,0) {};
    \node[vertex,label={left:$y$}] (y) at (1,-1) {};
    \node[vertex,label={below:$z$}] (z) at (1,-2) {};

    \coordinate (d1) at (-0.5, -0.5);
    \coordinate (d2) at (2.5, -0.5);
    \coordinate (d3) at (0, -2);
    \coordinate (d4) at (2, -2);
    \coordinate (d5) at (-0.5, -1.5);
    \coordinate (d6) at (2.5, -1.5);
    \node (d7) at (-0.5, -1) {\rvdots};
    \node (d8) at (2.5, -1) {\rvdots};
    
    \draw (u1) -- (x) -- (u3);
    \draw[thickedge] (x) -- (y) -- (z);
    \draw (d3) -- (z) -- (d4);
    \begin{scope}[on background layer]
        \draw (d1) to[bend left=45] (u1.center);
        \draw (d2) to[bend right=45] (u3.center);
    \end{scope}
    \draw (d3) to[bend left=45] (d5);
    \draw (d4) to[bend right=45] (d6); 
    
\end{tikzpicture}

\begin{tikzpicture}[scale=1]
    \node[precoloured] (u1) at (0,0) {};
    \node (u2) at (1,0) {\ldots};
    \node[precoloured,label={above:$x$}] (x) at (2,0) {};
    \node[vertex,label={left:$y$}] (y) at (2,-1) {};
    \node[vertex,label={below:$z$}] (z) at (2,-2) {};

    \node[vertex,label={right:$v\in B$}] (v) at (3, -1) {};

    \node (d1) at (1,-2) {\ldots};
    \node (d2) at (-0.5,-1) {\rvdots};
    \coordinate (d3) at (-0.5, -0.5);

    \draw (u1) -- (u2) -- (x);
    \draw[thickedge] (x) -- (y) -- (z);
    \draw (v) -- (x);
    \draw (v) -- (y);
    \draw (v) -- (z);
    \draw (d1) -- (z);
    \begin{scope}[on background layer]
        \draw (u1.center) to[bend right=45] (d3.center);
    \end{scope}

\end{tikzpicture}}%
    \captionsetup{width=0.9\textwidth}
    \caption{2-chords $xyz$ of the type that are not ruled out by Lemma~\ref{lem:2chords}.\label{fig:2chords}}
\end{figure}

A nearly identical proof gives the following result. 
We show examples in Figure~\ref{fig:3chords}.

\begin{lemma}\label{lem:3chords}
    If $xyzw$ is a 3-chord where $g(y) \geq 5$ and neither $x$ nor $w$ is an internal vertex of $P$, then $xyzw$ separates $G$ into $G_1$ and $G_2$ where $(G_2, xyzw, A, B, f)$ is an exceptional canvas of type~\ref{ex:A} or~\ref{ex:AB}. 
    If in addition we have $g(z)\ge 5$, then $(G_2, xyzw, A, B, f)$ is an exceptional canvas of type~\ref{ex:A}.
\end{lemma}


\begin{figure}[ht]\centering
    \subcaptionbox{A 3-chord yielding an exceptional canvas of type~\ref{ex:A}.}[0.4\textwidth]{\begin{tikzpicture}[scale=1]
    \node (u2) at (1,0) {\ldots};
    \node[vertex,label={above:$x$}] (x) at (2,0) {};
    \node[vertex,label={left:$y$}] (y) at (2,-1) {};
    \node[vertex,label={left:$z$}] (z) at (2,-2) {};
    \node[vertex,label={below:$w$}] (w) at (2,-3) {};

    \node[vertex,label={right:$v\in A$}] (v) at (3, -1.5) {};

    \node (d1) at (1,-3) {\ldots};
    \coordinate (d3) at (-0.5, -0.5);

    \draw (u2) -- (x);
    \draw[thickedge] (x) -- (y) -- (z) -- (w);
    \draw (v) -- (x);
    \draw (v) -- (w);
    \draw (d1) -- (w); 

\end{tikzpicture}

\begin{tikzpicture}[scale=1]
    \node (u2) at (1,0) {\ldots};
    \node[vertex,label={above:$x$}] (x) at (2,0) {};
    \node[vertex,label={left:$y$}] (y) at (2,-1) {};
    \node[vertex,label={left:$z$}] (z) at (2,-2) {};
    \node[vertex,label={below:$w$}] (w) at (2,-3) {};

    \node[vertex,label={right:$u\in A$}] (u) at (3, -1) {};
    \node[vertex,label={right:$v\in B$}] (v) at (3, -2) {};

    \node (d1) at (1,-3) {\ldots};
    \coordinate (d3) at (-0.5, -0.5);

    \draw (u2) -- (x);
    \draw[thickedge] (x) -- (y) -- (z) -- (w);
    \draw (v) -- (u) -- (x);
    \draw (v) -- (w);
    \draw (d1) -- (w); 
    \draw (v) -- (z);

\end{tikzpicture}}%
    \hspace*{0.1\textwidth}
    \subcaptionbox{A 3-chord yielding an exceptional canvas of type~\ref{ex:AB}.}[0.4\textwidth]{\begin{tikzpicture}[scale=1]
    \node (u2) at (1,0) {\ldots};
    \node[vertex,label={above:$x$}] (x) at (2,0) {};
    \node[vertex,label={left:$y$}] (y) at (2,-1) {};
    \node[vertex,label={left:$z$}] (z) at (2,-2) {};
    \node[vertex,label={below:$w$}] (w) at (2,-3) {};

    \node[vertex,label={right:$v\in A$}] (v) at (3, -1.5) {};

    \node (d1) at (1,-3) {\ldots};
    \coordinate (d3) at (-0.5, -0.5);

    \draw (u2) -- (x);
    \draw[thickedge] (x) -- (y) -- (z) -- (w);
    \draw (v) -- (x);
    \draw (v) -- (w);
    \draw (d1) -- (w); 

\end{tikzpicture}

\begin{tikzpicture}[scale=1]
    \node (u2) at (1,0) {\ldots};
    \node[vertex,label={above:$x$}] (x) at (2,0) {};
    \node[vertex,label={left:$y$}] (y) at (2,-1) {};
    \node[vertex,label={left:$z$}] (z) at (2,-2) {};
    \node[vertex,label={below:$w$}] (w) at (2,-3) {};

    \node[vertex,label={right:$u\in A$}] (u) at (3, -1) {};
    \node[vertex,label={right:$v\in B$}] (v) at (3, -2) {};

    \node (d1) at (1,-3) {\ldots};
    \coordinate (d3) at (-0.5, -0.5);

    \draw (u2) -- (x);
    \draw[thickedge] (x) -- (y) -- (z) -- (w);
    \draw (v) -- (u) -- (x);
    \draw (v) -- (w);
    \draw (d1) -- (w); 
    \draw (v) -- (z);

\end{tikzpicture}}%
    \captionsetup{width=0.9\textwidth}
    \caption{3-chords $xyzw$ of the type we cannot rule out in Lemma~\ref{lem:3chords}: $g(y)\ge 5$, neither $x$ nor $w$ is an internal vertex of $P$, and one side of the chord is an exception of type~\ref{ex:A} or~\ref{ex:AB}.\label{fig:3chords}}
\end{figure}

As a consequence of Lemmas~\ref{lem:triangles}, \ref{lem:1chord} and \ref{lem:Pispath}, we have the following fact in a minimum counterexample.
\begin{cor}
       $v(P) \geq 3$.
\end{cor}
\begin{proof}
    Suppose for contradiction that $v(P)\le 2$.
    By Lemma \ref{lem:Pispath}, $V(P)$ induces a path in $G$ and so $\delta G$ contains a vertex $v$ immediately following $P$. 
    Since $K$ is not exceptional, either $f(v) \geq 2$, or $f(v) = 1$ and $v$ is adjacent to only one vertex in $P$. 
    In any case, $f_K(v) = f(v) - |N(v)\cap V(P)| \geq 0$.  
    We claim that $K' = (G,P+v, A, B, f)$ is a canvas. Since $K'$ is obtained from $K$ by growing the acceptable path, the only property we need to check is that $P+v$ is indeed acceptable in $G$.  Since $v(P) \leq 2$, it follows that $P+v$ has at most three vertices, and so it is trivially an acceptable path or induces an acceptable cycle. Next, since $P+v$ has at most three vertices and $\delta G$ has no 1-chord by Lemma \ref{lem:1chord}, $K'$ is unexceptional.
    By the minimality of $K$, $K'$ is weakly $f$-degenerate which means that $G-P-v$ is weakly $f'$-degenerate for the function $f':V(G-P-v)\to\mathbb{N}$ with $f'(u)=f(u)-|N(u)\cap V(P+v)|$. 
    But this $f'$ is the function one obtains by calling $\del(G-P,f_K,v)$. 
    This shows that $K$ is weakly $f$-degenerate, a contradiction. 
\end{proof}

\subsection{Removing an arbitrarily long path in \texorpdfstring{$\delta G$}{the boundary of G}}\label{subsec:removeR}

The next step of our proof follows a similar structure to the analogous proof for list colouring found in \cite{PS22}: given a minimum counterexample $K$ we remove a path $R$ near $P$ and argue for contradiction that a smaller canvas $\tilde K$ obtained from $K$ via this operation satisfies the hypotheses of Theorem \ref{thm:inductive}.
Suppose that $\tilde f$ is the function associated with the canvas $\tilde K$.
As vertices $v$ in $\delta G$ already have relatively small values $f(v)$, we aim to define $R$ in such a way that no vertex in $\delta G$ sees $\tilde f(v)$ dip too far below $f(v)$. 
The precise definition of $R$ thus depends on the values of $f$ on vertices in $\delta G$. 
One factor complicates our analysis beyond that in \cite{PS22}.
In the list colouring setup it is tempting to choose colours for vertices in $R$ in a coordinated fashion that is not obviously possible to mimic in the weak degeneracy setup. 
We must therefore carefully simulate something analogous to a coordinated choice of colour along $R$ with the weak degeneracy operations.

As we will show in this section, there are essentially two ways in which this general approach can fail, both of which come down to $\tilde K = (\tilde G, P, \tilde A, \tilde B, \tilde f)$ failing to be a canvas: either $\tilde B$ will fail to be an independent set, or $\tilde A$ will fail to be an independent set.
Fortuitously, we are able to show that these two situations are mutually exclusive. We will handle the former case in Section \ref{sec:tildeB} and the latter in Section \ref{sec:tildeA}, thus completing the proof of Theorem \ref{thm:inductive}.

We start by establishing notation for this part of the proof. 
There is a subtle technicality to handle coming from the way we define a canvas being weakly degenerate. 
We start with the canvas $K=(G,P,A,B,f)$ which (by definition) is weakly $f$-degenerate if and only if the graph $G-P$ is weakly $f_K$-degenerate where $f_K(v)=f(v)-|N(v)\cap P|$ for each $v \in V(G-P)$. 
When we define $R$ and give a way of removing $R$ with weak degeneracy operations, we intend for these operations to be legal when starting from the pair $(G-P,f_K)$. 
Suppose that the resulting function is $f':V(G)\setminus (V(P)\cup V(R))\to \mathbb{N}$. 
Then we define $\tilde f(v)$ by $\tilde f(v)=f'(v)+|N(v)\cap P|$ and show that for suitable sets $\tilde A$ and $\tilde B$, $\tilde K = (G-R, P, \tilde A, \tilde B, \tilde f)$ is an unexceptional canvas.
By minimality, $\tilde K$ is not a counterexample to Theorem~\ref{thm:inductive} and hence $\tilde K$ is weakly $\tilde f$-degenerate. 
But this means that the graph $G-R-P$ is weakly $f'$-degenerate because $f'(v) = \tilde f(v) - |N(v)\cap P|$ by definition. 
This shows, therefore, that $K$ is weakly $f$-degenerate in the sense that starting from $G-P$ and the function $f_K$, by a sequence of operations that starts with the removal of $R$ and then continues with a sequence whose existence follows from the fact that $\tilde K$ is not a counterexample, we can reach an empty graph. 
This is the desired contradiction.

The subtle technicality is the subtraction of $|N(v)\cap V(P)|$ from $f(v)$ in the definition of $f_K$, and the addition of this term back to get the function $\tilde f$ from $f'$. One consequence of this technicality is that the validity of the sequence of operations used to remove $R$ must be established starting with the graph $G-P$ and the function $f_K$. 
We carefully define $f'$ and $\tilde f$ after identifying $R$ and the sequence of operations used to remove it below.

Let $P = u_k u_{k-1} \dots u_1$, and let $P'$ and $P''$ be the two subpaths of $\delta G$ in the unique decomposition $\delta G = PP'P''u_k$, where for $P'' = v_1 \dots v_t$, we have either
\begin{itemize}
    \item $P'$ is empty and $f(v_1) \geq 2$, or
    \item $P'$ contains a single vertex $v_0$ with $v_0 \in A$ (and so $f(v_0) = 1$).
\end{itemize}
In both cases, define $v_{t+1} := u_k$. 
We always have that $P''$ is non-empty. 
In the first case this follows from the fact $P$ is an induced path (by Lemma \ref{lem:Pispath}), and hence $\delta G$ contains a vertex outside $P$. 
In the second case, since $v_0 \in A$, we have that $g(v_0) \geq 5$ and thus $v(P) \geq 4$. As $K$ is unexceptional, $v_0$ has a neighbour outside $u_1$ and $u_k$, and since $\delta G$ is a cycle, the result follows. Moreover, $f(v_1) \geq 2$ in both cases since $A$ is an independent set.

\begin{figure}[ht]\centering
    \begin{tikzpicture}[scale=1]

    \node[precoloured,label={above:$u_k$},label={below:$v_{t+1}$}] (uk) at (-1,0) {};
    \node (ud) at (0,0){$\ldots$};
    \node[precoloured,label={above:$u_1$}] (u1) at (1,0){};
    \node[vertex,paren,label={above:$v_0$}] (v0) at (2,0){};
    \node[vertex,label={above:$v_1$}] (v1) at (3,0){};
    \node (vd) at (4,0) {$\ldots$};
    \node[vertex,label={above:$v_t$}] (vt) at (5,0){};

    \draw (uk) -- (ud) -- (u1) -- (v0) -- (v1) -- (vd) -- (vt);
    \begin{scope}[on background layer]
    \draw (vt.center) to[in=0,out=0,looseness=2] (5,-2) -- (-1, -2) to[in=180,out=180,looseness=2] (uk.center);
    \end{scope}
    \draw [decorate,decoration={brace,amplitude=5pt,raise=6ex}] (uk.west) -- (u1.east) node[midway,above=7ex]{$P$};
    \draw [decorate,decoration={brace,amplitude=5pt,raise=6ex}] (1.6,0) -- (2.4,0) node[midway,above=7ex]{$P'$};    
    \draw [decorate,decoration={brace,amplitude=5pt,raise=6ex}] (v1.west) -- (vt.east) node[midway,above=7ex]{$P''$};    
  
\end{tikzpicture}%
    \captionsetup{width=0.9\textwidth}
    \caption{The decomposition of $\delta G$ into $PP'P''u_k$. If $v_0$ exists then $v_0\in A$, and we have $f(v_1)\ge 2$ in any case. The labels $u_k$ and $v_{t+1}$ refer to the same vertex for convenience.\label{fig:Ps}}
\end{figure}
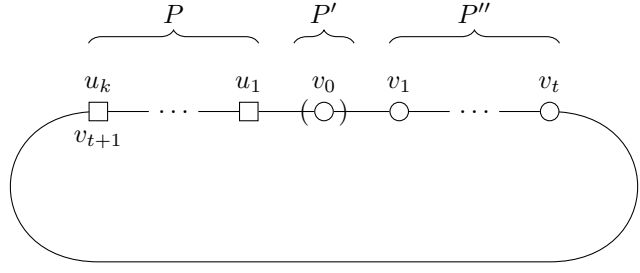

Depending on the structure of $P''$, we define and remove a path $R$ with carefully specified weak degeneracy operations as follows.
Note that we consider the cases mutually exclusive. For brevity, we do not state mutually exclusive conditions and we remove $R$ according to the lexicographically minimal applicable case. 
We will write $j$ for the largest index such that $v_j \in V(R)$.  In some cases we will refer to subpaths $Q_2$ and $Q_3$ of $P''$ defined as follows: $Q_2$ is the maximal subpath of $P''$ immediately following $v_2$ with the property that $f(v) =2$ for all $v \in V(Q_2)$, and $Q_3$ is the maximal subpath of $P''$ immediately following $Q_2$ with the property that $f(v) = 3$ for all $v \in V(Q_3)$. 
\begin{enumerate}[label=(R\arabic*)]
    \item\label{case:vP''<=2} If $v(P'') \le 2$ then let $R = P''$.
    \item\label{case:f33} If $f(v_1) = f(v_2) = 3$, then let $R = v_1v_2$. 
    \item\label{case:f3233}  If $f(v_1) = 3$ and $f(v_2) \leq 2$ and $v(Q_3) \geq 2$, let $R = v_2Q_2$. 
    \item\label{case:f323} If $f(v_1) = 3$ and $f(v_2) \leq 2$ and $v(Q_3) \leq 1$, let $R = v_2Q_2Q_3$.
    \item\label{case:f23}  If $f(v_1) = 2$, $f(v_2) = 3$, and $f(v_3) \leq 2$, let $R = v_1v_2$.
    \item\label{case:f233} If $f(v_1) = 2$ and $f(v_2) = f(v_3) = 3$, let $R = v_1$.
    \item\label{case:f2233}  If $f(v_1) = 2$, $f(v_2) \leq 2$ and $v(Q_3) \geq 2$, let $R = v_1v_2Q_2$.
    \item\label{case:f223}  Otherwise, it must be that $f(v_1) = 2$,  $f(v_2) \leq 2$ and $v(Q_3) \leq 1$. Let $R = v_1v_2Q_2Q_3$.
\end{enumerate}

Given the definition of $R$, in each case we perform the following sequences of weak degeneracy operations to remove $R$. 
We will prove that the sequence used to remove $R$ is legal in $G-P$ when starting with the function $f_K$. 
Because $\delta G$ has no 1-chord (by Lemma \ref{lem:1chord}), the only vertices in $\delta G-P$ that are adjacent to a vertex in $P$ are $v_0$ (only if $P'$ is nonempty), $v_1$ (only if $P'$ is empty) and the last vertex of $P''$. Hence for every internal vertex $v$ of $R$, we have that $f_K(v) = f(v) \geq 1$. 

\begin{enumerate}[label=(D\arabic*)]
    \item\label{handle:vP''<=2} If $v(P'') \le 2$ then $R = P'' \in\{v_1, v_1v_2\}$.
    \begin{itemize} 
        \item If $v(P'')=1$, remove $R$ via $\delsave(v_1, v_0)$.  This is legal since either $v_0$ does not exist or $f_K(v_0) = 0$ and $f_K(v_1) \geq 1$.
        \item If $v(P'')=2$, remove $R$ via $\del(v_2), \delsave(v_1,v_0)$. This is legal because $v_1\notin A$ and hence $f_k(v_1) = f(v_1) \geq 2$, and if $P'$ is non-empty then $f_K(v_0) = 0$.
    \end{itemize}
    \item\label{handle:f33} If $f(v_1) = f(v_2) = 3$, then $R = v_1v_2$. 
    Remove $R$ via $\delsave(v_2,v_3), \delsave(v_1, v_0)$. Recall that by Lemma \ref{lem:Bmax}, if $v_3 \neq u_k$ then $f_K(v_3) < 3$ and hence $\delsave(v_2,v_3)$ is legal. Either $v_0$ does not exist or $f_K(v_0) = 0$, and hence $\delsave(v_1, v_0)$ is also legal. 

    \item\label{handle:f3233}  If $f(v_1) = 3$ and $f(v_2) \leq 2$ and $v(Q_3) \geq 2$ then $R = v_2Q_2$. 
    Remove $R$ via $\del(v_j)$, $\del(v_{j-1})$, $\dotsc$, $\del(v_2)$. 
    This is always legal since $f_K(v) \geq 1$ for all $v \in V(R)$.
    \item\label{handle:f323} If $f(v_1) = 3$ and $f(v_2) \leq 2$ and $v(Q_3) \leq 1$ then $R = v_2Q_2Q_3$.
    Remove $R$ via first performing $\delsave(v_j, v_{j+1})$, and then removing the remaining vertices in $R$ via $\del(v_i)$ in decreasing order of index $i$. 
    By the definition of $Q_3$, we have that if $v_{j+1} \neq u_k$, then $f_K(v_{j+1}) < 3$ and hence $\delsave(v_j, v_{j+1})$ is legal; and if $v_{j+1}=u_k$ then $\delsave(v_{j}, v_{j+1}) = \del(v_j)$. 
    The remaining operations are also legal since $f_K(v) \geq 1$ for all $v \in R$.
    \item\label{handle:f23} If $f(v_1) = 2$, $f(v_2) = 3$, and $f(v_3) \leq 2$, then $R = v_1v_2$.
    Remove $R$ via $\delsave(v_2,v_3)$, $\delsave(v_1,v_0)$. 
    Note that $\delsave(v_2,v_3)$ is legal since $f_K(v_2) > f_K(v_3)$ if $v_3 \neq u_k$, and $\delsave(v_2,v_3) = \del(v_2)$ otherwise. Moreover, $\delsave(v_1,v_0)$ is legal since either $v_0$ does not exist or $f_K(v_0) = 0$ and $f_K(v_1) = f(v_1) = 2$.
    \item\label{handle:f233} If $f(v_1) = 2$ and $f(v_2) = f(v_3) = 3$ then $R = v_1$. Remove $R$ via $\delsave(v_1,v_0)$. This is legal since either $v_0$ does not exist or $f_K(v_0) = 0$.
    \item\label{handle:f2233}  If $f(v_1) = 2$, $f(v_2) \leq 2$ and $v(Q_3) \geq 2$, let $R = v_1v_2Q_2$. Remove $R$ via $\del(v_j)$, $\del(v_{j-1})$, $\dotsc$, $\delsave(v_1, v_0)$.  The delete operations are legal since $f_K(v) \geq 2$ for all $v \in V(R) \setminus \{v_1,v_2\}$, and $f_K(v_2) \geq 1$. The operation $\delsave(v_1, v_0)$ is legal since $f(v_1) = 2$ and either $v_0$ does not exist or $f_K(v_0) = 0$. 
    \item\label{handle:f223}  Otherwise, it must be that $f(v_1) = 2$,  $f(v_2) \leq 2$ and $v(Q_3) \leq 1$. Let $R = v_1v_2Q_2Q_3$. Remove $R$ via $\delsave(v_j,v_{j+1})$, $\del(v_{j-1}), \dots, \del(v_2), \delsave(v_1,v_0)$. By the definition of $R$, if $v_{j+1} \neq u_k$, then $f(v_{j+1}) > f(v_j)$ and hence $\delsave(v_j,v_{j+1})$ is legal. The remaining operations are legal for the same reason as \ref{handle:f2233}.
\end{enumerate}

\noindent
The arguments given above yield the following important lemma.

\begin{lemma}\label{lem:delRlegal}
    The sequence of operations used to remove $R$ is legal when starting with the pair $(G-P, f_K)$.
\end{lemma}

Recall that $A$ is an independent set. Hence in Case \ref{handle:vP''<=2}, if $P'$ is nonempty then $v_1 \not \in A$ and so $f(v_1) \geq 2$. If $P'$ is empty, then by definition $f(v_1) \geq 2$. 
Analysing the cases above, we find the following.

\begin{obs}\label{obs:whereistheAvert}
 If $R$ contains a vertex $v$ with $f(v) = 1$, then $v = v_2$.
\end{obs}

\noindent
The following observation will also be useful.

\begin{obs}\label{obs:whereisthelist4vert}
    If $R$ is defined according to \ref{case:f3233}, \ref{case:f323}, \ref{case:f2233}, or \ref{case:f223}, then $R$ contains at most one vertex $v$ with $f(v) = 3$, and if such a vertex exists then it is the vertex $v_j$. 
\end{obs}

Let the pair $(G-P-R, f')$ be obtained from $(G-P,f_K)$ by removing $R$ with the prescribed weak degeneracy operations. Define $\tilde f(v)$ by $\tilde f(v)=f'(v)+|N(v)\cap P|$, and let $\tilde G = G-R$.
Let $\tilde A$ be the set of vertices $v$ in $V(\delta \tilde G) \setminus V(P)$ with $\tilde f(v) \leq 1$.
Let $\tilde B$ be the set of vertices $v$ in $V(\delta \tilde G) \setminus V(P)$ with $\tilde f(v) = 2$ and $g_{\tilde G}(v) = 3$. 
Note that if $\tilde K := (\tilde G, P, \tilde A, \tilde B, \tilde f)$ is an unexceptional canvas, then by the minimality of $K$ we have that $\tilde K$ is weakly $\tilde f$ degenerate. 
As argued above, following from Lemma~\ref{lem:delRlegal} in this case $K$ is $f$-degenerate\textemdash a contradiction. 
Thus, either $\tilde K$ is not a canvas, or it is an exceptional canvas. 
Note that Properties~\ref{canv:acceptable} and~\ref{canv:fint} hold for $\tilde K$, since $K$ is a canvas. 
Moreover, Properties~\ref{canv:fCn3} and~\ref{canv:fC3} hold by definition of $\tilde A$ and $\tilde B$.

\clearpage
\begin{lemma}\label{lem:dGtildeAprops} The following properties hold for vertices in $V(\delta G)$. 
\begin{enumerate}[label=\textup{(\roman*)}]
    \item\label{itm:dGtildeAgf} If $v \in V(\delta \tilde G) \cap V(\delta G)$ and $v \in \tilde A$, then $g(v) \geq 5$ and $\tilde f(v) = 1$. 
    \item\label{itm:dGtildeABindep} $\tilde A \cap V(\delta G)$ and $\tilde B \cap V(\delta G)$ are independent sets in $\tilde G$. 
\end{enumerate}
\end{lemma}
\begin{proof}
By the definition of $\tilde A$, to prove~\ref{itm:dGtildeAgf} we must consider the vertices $v\in V(\delta\tilde G)\cap V(\delta G)$ with $\tilde f(v)\le 1$ and show that $g(v)\ge 5$ and $\tilde f(v)=1$. 

For any vertex $v$, if $\tilde f(v)\le 1$ this is either because $f(v)=1$ and $\tilde f(v)=1$, or because $v$ is adjacent to at least one vertex of $R$ whose removal causes $\tilde f(v) < f(v)$. 
In the first case, the fact that $K$ is a canvas means that $g(v) \geq 5$ as required. 
Since $\delta G$ has no 1-chord by Lemma \ref{lem:1chord}, the only vertices $v\in \delta G$ for which the second case applies are the vertices immediately preceding or immediately following $R$ along $\delta G$.
Note further that if $v\in \delta G$ is immediately preceding or following $R$ then the first case applies if the neighbour $v'$ of $v$ in $R$ was removed with a $\delsave(v',v)$ operation. 

Consider the vertex $v$ which immediately precedes $R$ (if it exists). Its neighbour $v'$ in $R$ is removed with $\delsave(v',v)$ in all cases except \ref{handle:f3233} and \ref{handle:f323}. 
In these cases, $v = v_1$ and $f(v) = 3$. It follows that $v \not \in B$ and $\g(v) = 3$. Since $\tilde f(v) = f(v)-1 = 2$, we have that $v \notin \tilde A$ and hence there is nothing to prove for $v$ to establish~\ref{itm:dGtildeAgf}. 

Next, consider the vertex $v=v_{j+1}$ immediately following $R$. Note that if $v = u_k$, there is nothing to show since $u_k \in V(P)$ and therefore $u_k \not \in \tilde{A}$. 
Similarly, in all cases except \ref{handle:f3233}, \ref{handle:f233}, and \ref{handle:f2233} we have $\tilde f(v)=f(v)$ (and so the simpler analysis above applies). 
In these cases, $f(v) = 3$ and so $\tilde f(v_{j+2}) = 3$. 
After removing $R$ we have that $\tilde f(v) = 2$, and hence $v\notin \tilde A$ and there is nothing to prove for $v$ to establish~\ref{itm:dGtildeAgf}.

We now turn to proving~\ref{itm:dGtildeABindep} with a similar method. 
It suffices to verify that $\delta G - R$ does not contain a pair of consecutive vertices in $\tilde A$ or $\tilde B$, and by the analysis above such a pair must be located immediately preceding $R$ or immediately following $R$. 
Moreover, since $A$ and $B$ are independent, an edge in $\tilde A \cap V(\delta G)$ or $\tilde B \cap V(\delta G)$ can exist only if a vertex in $\delta G$ adjacent to $R$ is in $\tilde A\setminus A$ or $\tilde B\setminus B$ respectively, and the next vertex along lies in $A$ or $B$ respectively.

Following the proof of~\ref{itm:dGtildeAgf}, consider again the vertex $v$ which immediately precedes $R$ (if it exists). Its neighbour $v'$ in $R$ is removed with $\delsave(v',v)$ in all cases except \ref{handle:f3233} and \ref{handle:f323}, and in these cases it cannot be that $v\in \tilde A$. In these cases, $v = v_1$. Note that $v_0$ is either in $A$ or non-existent, and so neither $\tilde A$ nor $\tilde B$ induce an edge immediately preceding $R$.

Similarly, consider the vertex $v=v_{j+1}$ immediately following $R$. Again, if $v = u_k$, there is nothing to show. 
In all cases except \ref{handle:f3233}, \ref{handle:f233}, and \ref{handle:f2233} we have $\tilde f(v)=f(v)$ and so $v \not \in (\tilde{A} \cup \tilde{B}) \setminus (A \cup B)$.
In the problematic cases, $f(v_{j+1}) = 3$ (and so $g(v_{j+1}) = 3$) and there exists a vertex $v_{j+2}$ with $f(v_{j+2}) = 3$ (and hence also $\tilde f(v_{j+2}) = 3$). 
After removing $R$ we have that $\tilde f(v_{j+1}) = 2$ and hence $v_{j+1}\in \tilde B$. Since $\tilde f(v_{j+2})=3$, we have that $v_{j+2}\notin \tilde B$ and so $\tilde B$ is independent in this case as desired.
\end{proof}

We will also need to verify that the vertices in $\tilde A \setminus V(\delta G)$ have girth at least five and do not see $\tilde f$ dip too far below $f$. To do so, we use the following useful consequence of Lemma \ref{lem:2chords}.

\begin{cor}\label{cor:notmanyRnbrs}
If $v$ is a vertex in $V(G-\delta G)$ with $g(v) \geq 4$, then $v$ has at most one neighbour in $R$.
\end{cor}
\begin{proof}
    Suppose not, and let $u$ and $w$ be neighbours of $v$ in $R$. Note $u$ and $w$ are not adjacent, since $g(v) \geq 4$. It also follows from the fact that $g(v) \geq 4$ that $uvw$ is a 2-chord that violates Lemma \ref{lem:2chords}.
\end{proof}

As a consequence of this, we obtain the following.

\begin{cor}\label{cor:4s5sdidnotlosetoomuch}
If $v \in \tilde A$, then $g(v) \geq 5$ and $\tilde f(v) = 1$.
\end{cor}
\begin{proof}
    If $v \in \tilde{A}$ and $v \in V(\delta G)$, this follows from Lemma \ref{lem:dGtildeAprops}\ref{itm:dGtildeAgf}. If $v \in \tilde{A} \setminus V(\delta G)$ and $g(v) \geq 4$, then since $\tilde{f}(v) \leq 1$, by Corollary \ref{cor:notmanyRnbrs} we have that $g(v) \geq 5$ as desired. Thus we may assume $g(v) = 3$; in this case, since $f(v) = 4$, we have that $v$ is adjacent to at least three vertices $v_q,v_r,v_s$ in $R$. 
    Then $v(R) \geq 3$ and so $R$ was defined according to either \ref{case:f3233}, \ref{case:f323}, \ref{case:f2233}, or \ref{case:f223}. 
    By Lemma \ref{lem:2chords}, we have without loss of generality that $r = q+1$, that $s = r+1$, and that $v_r \in B$. 
    Since $B$ is an independent set, $v_s \not \in B$. Since $g(v_s) = 3$, it follows that $f(v_s) = 4$. This contradicts Observation~\ref{obs:whereisthelist4vert}.
\end{proof}

Lemma~\ref{lem:dGtildeAprops} and Corollary~\ref{cor:4s5sdidnotlosetoomuch} give the girth condition relevant to $\tilde A$ for property~\ref{canv:A} and give properties~\ref{canv:fA} and~\ref{canv:fB} for $\tilde K$.

Thus, the function $\tilde f$ satisfies the relevant properties in the definition of canvas with respect to $\tilde K$.
At this point, the only parts of the definition of canvas we have not verified for $\tilde K$ are the independence conditions in~\ref{canv:A} and~\ref{canv:B}. 
Nevertheless, since we are working with a hypothetical minimal counterexample, $\tilde K$ must fail to be a canvas.

\begin{cor}
    $\tilde K$ is not a canvas.
\end{cor}
\begin{proof}
    Suppose for a contradiction that $\tilde K$ is a canvas. We will argue that it is not exceptional, and thus contradict our assumption that $K$ is a minimal counterexample to Theorem~\ref{thm:inductive}.
    
    Since $\delta G$ has no 1-chord by Lemma \ref{lem:1chord} and $K$ is not exceptional, no vertex in $\tilde A \cap V(\delta G)$ has two neighbours in $P$. Moreover, no vertex in $\tilde A \setminus V(\delta G)$ neighbours both endpoints of $P$, as otherwise there is a 2-chord violating Lemma \ref{lem:2chords}.
    Thus, $\tilde K$ is not an exceptional canvas of type~\ref{ex:A}. 
    
    Similarly, we claim that $\tilde K$ is not an exceptional canvas of type~\ref{ex:AB}: to see this, suppose not. Let $b \in \tilde B$ be the vertex adjacent to two vertices of $P$, and let $a \in \tilde A$ be the vertex adjacent to an endpoint of $P$ and to $b$. 
    If $a \notin V(\delta G)$, then since $a \in \tilde A$ it follows that $a$ is adjacent to a vertex $x$ in $R$. But then $x,a$ and the endpoint of $P$ adjacent to $a$ form a 2-chord violating Lemma \ref{lem:2chords}. Hence $a \in V(\delta G)$.  Since $\delta G$ has no 1-chord by Lemma \ref{lem:1chord}, $b \not \in \delta G$, and hence since $b \in \tilde B$, we have that $b$ has two neighbours in $R$. But then $a$, $b$, and a neighbour of $b$ in $R$ form a 2-chord violating Lemma~\ref{lem:2chords} (since $g(a) \geq 5$).
    
    Suppose finally that $\tilde K$ is an exceptional canvas of type~\ref{ex:B}. 
    Then $v(P) = 3$, and $\tilde B \setminus B$ contains a vertex $b$ adjacent to $u_1, u_2$, and $u_3$. Since $b$ is adjacent to $u_1$ and $u_3$, Lemma \ref{lem:2chords} implies that $u_1bu_3$ separates $G$ into $G_1$ and $G_2$, where $(G_2, u_1bu_3, A, B, f)$ is an exceptional canvas of type~\ref{ex:B}. But then $v(\delta G) = 4$, and $b$ is adjacent to every vertex in $\delta G$.  
    By Lemma \ref{lem:triangles}, $v(G) = 5$ and so $\deg(b) = 4$. Since $g(b) = 3$ and $b \not \in \delta G$, we have that $f(b) = 4$. This contradicts Lemma \ref{lem:nolowdegs}.
\end{proof}

We have now proved that $\tilde K$ is not a canvas, defined $\tilde B$ such that it contains vertices of girth three, verified that all vertices in $\tilde A$ have girth at least five (Corollary \ref{cor:4s5sdidnotlosetoomuch}), and proved that every canvas property holds except the independence statements in~\ref{canv:A} and~\ref{canv:B}, it must be that $\tilde A$ or $\tilde B$ is not an independent set in $\tilde G$. 
Indeed, the statements of the upcoming Lemmas~\ref{lem:BBstructure} and \ref{lem:AAstructure} show that exactly one of these occurs.

\begin{lemma}\label{lem:BBstructure}
    If $\tilde B$ is not an independent set in $\tilde G$ then the following hold.
    \begin{enumerate}[label=\textup{(\roman*)}]
        \item\label{itm:BBw} $G-\delta G$ contains a vertex $w$ that is adjacent to $v_1$, $v_2$, and $v_3$,

        \item\label{itm:BBunique} $\tilde B$ induces exactly one edge $v_1w$ in $\tilde G$.
        \item\label{itm:BBv2}
        $N(v_2) = \{v_1,v_3,w\}$, and
        \item\label{itm:BBdefofR}
        $R=v_2v_3$ and is defined according to~\ref{case:f323} with $v(Q_2)=0$ and $Q_3=v_3$.
        %
    \end{enumerate}
\end{lemma}

\begin{proof}[Proof of Lemma~\ref{lem:BBstructure}]
Suppose that $\tilde B$ is not an independent set in $\tilde G$ and let $ww'$ be an edge induced by $\tilde B$ in $\tilde G$. 
We break into cases according to the size of $\{w,w'\}\cap V(\delta G)$.

\textbf{Case 1.} Suppose that $w$ and $w'$ are both in $\tilde B \setminus V(\delta G)$. We will show that this does not happen as it leads to a contradiction. 
By property~\ref{canv:fint} and Observation~\ref{obs:feq}, we have that $f(w) = 4$. 
Since no operation used to remove $R$ is of the form $\delsave(v,w)$ with $w\notin\delta G$, we have $\tilde f(w) = f(w) - |N(w)\cap R|$. 
Thus, $w$ must be adjacent (in $G$) to two vertices, say $v_q$, $v_r$, in $R$. 
Similarly, $w'$ is adjacent to two vertices $v_s$, $v_t$ in $R$. 
Without loss of generality, we may assume that $q < r$ and  $s < t$, and since $G$ is planar we may also assume that $r\leq s$. 
Note then that $v(R) \geq 3$, and so $R$ is defined according to either \ref{case:f3233}, \ref{case:f323}, \ref{case:f2233}, or \ref{case:f223}. Using Observation \ref{obs:whereisthelist4vert}, it follows that $w$ and $w'$ are each adjacent to \emph{exactly} two vertices in $R$, as otherwise there is a 2-chord violating Lemma \ref{lem:2chords} (as one of the endpoints of this chord would have function value 2). But then again by Lemma \ref{lem:2chords}, we have that $r = q+1$ and $t = s+1$, and so $g(v_q)=g(v_r) = g(v_s) = g(v_t) = 3$. 
By the definition of $R$, both $v_2$ and $v_3$ are in $V(R)$.

Since $B$ is an independent set in $G$, at least one vertex $v \in \{v_s, v_t\}$ has $f(v) = 3$. 
Recall that we define $j$ such that $v_j$ is last vertex of $R$.
By Observation \ref{obs:whereisthelist4vert} again, since $R$ is defined according to either \ref{case:f3233}, \ref{case:f323}, \ref{case:f2233}, or \ref{case:f223}, $R$ has at most one vertex $v$ with $f(v) = 3$, and if such a vertex exists, it is the vertex $v_j$. Hence $v_t = v_j$, $f(v_t) = 3$, and $3 \not \in \{f(v_q), f(v_r), f(v_s)\}$. 
But then $v_q$ and $v_r$ are adjacent vertices of girth three with $f(v_q) = f(v_r) = 2$, contradicting the fact that $B$ is an independent set. 

\textbf{Case 2.} Suppose that $w$ and $w'$ are both in $\tilde B \cap V(\delta G)$. 
Again, we will show that this does not happen as it leads to a contradiction. 
$R$ is not defined according to case~\ref{case:vP''<=2} because then $V(\delta G-P)\cap V(\delta\tilde G)\subset\{v_0\}$ and contains no vertices of girth three.
Since $B$ is independent, without loss of generality we may assume that $w \in (\tilde{B}\setminus B) \cap V(\delta G)$ and thus $f(w)=3$ by~\ref{canv:fC3} and Observation~\ref{obs:feq}.  
Since $\delta G$ has no 1-chord by Lemma \ref{lem:1chord}, every vertex $v\in \delta G$ has its function value drop by at most 1 when we remove $R$, and hence a vertex $v\in \tilde B\cap V(\delta G)$ either lies in $B$ or has $f(v)=3$ and is adjacent to an endpoint of $R$.
Since $g(v_0)\ge 5$ (if it exists), we have that $v_0 \not \in \{w, w'\}$ and so either $w=v_1$ or $w=v_{j+1}$. 

If $w = v_1$, then $v_1 \notin V(R)$ and so since $\delta G$ has no 1-chord (Lemma~\ref{lem:1chord}) we have that $w'=v_2$ (and hence $v_2 \notin V(R)$). Inspecting the cases reveals that whenever $v_1\notin V(R)$ we have $v_2\in V(R)$, a contradiction.

Suppose now $w = v_{j+1}$. Note that in all cases except~\ref{case:vP''<=2} (which we already ruled out), either $Q_3$ is the terminal segment of $R$ or $R$ contains none of $Q_3$. Since is $Q_3$ defined as the maximal subpath of $P''$ immediately following $Q_2$ on which $f$ is $3$, since $w=v_{j+1}$ then $R$ is defined to contain none of $Q_3$.  Hence $R$ is defined according to one of cases~\ref{case:f33}, \ref{case:f3233}, \ref{case:f23}, \ref{case:f233} and~\ref{case:f2233}. 
Again, since $\delta G$ has no 1-chord by Lemma~\ref{lem:1chord}, we must have $w'=v_{j+2}$ and this means $f(v_{j+2})=2$ and $v(Q_3)=1$.
This leaves only case~\ref{case:f33}. But then since $f(w) = 3$, we have that $\delta G-P$ contains a path $v_1v_2w$ on three vertices of girth three on which $f=3$, violating Lemma~\ref{lem:Bmax}.

\textbf{Case 3.} Since the two cases above cannot occur, we conclude that if $\tilde B$ is not independent in $\tilde G$ then there exists a vertex $w \in \tilde B \setminus V(\delta G)$ and vertex $w'$ in $\tilde B \cap V(\delta G)$ such that $ww' \in E(\tilde G)$. 
This rules out being in case~\ref{case:vP''<=2} because then $V(\delta G-P)\cap V(\delta\tilde G)\subset\{v_0\}$ and contains no vertices of girth three.

Since $g(w)=3$ and $w \in \tilde B \setminus V(\delta G)$, it follows from~\ref{canv:fint} for the canvas $K$ that $w$ is adjacent to at least two vertices $v_s$ and $v_t$ in $R$, where without loss of generality $s < t$. 
That $v(R)\ge 2$ rules out case~\ref{case:f233}. 
Similar to Case 2, a vertex $w'\in \tilde B\cap V(\delta G)$ occurs either when $w'\in B$, or when $f(w')=3$ and $w'$ is adjacent to an endpoint of $R$. 

We first rule out the case $w'\in B$. 
If this occurs, then inspecting the cases we see that $w' = v_r$ for some $r > t$. 
But then the 2-chord $v_sww'$ violates Lemma~\ref{lem:2chords} because its endpoints are not internal vertices of $P$ and $w'\in B$.
Then it must be that $w'\notin B$ which means $f(w')=3$ and $w'$ is adjacent to an endpoint of $R$. 

We now rule out the case $w'=v_{j+1}$. 
As in Case 2, if $v_{j+1}\in \tilde B$ with $f(v_{j+1})=3$ then we must be in a case where $R$ contains none of $Q_3$. This leaves cases~\ref{case:f33}, \ref{case:f3233}, \ref{case:f23}, and~\ref{case:f2233}. 
We cannot be in case~\ref{case:f33} as then $R=v_1v_2$ and we have that $s=1$, $t=2$, and so $v_1v_2v_3$ is a path of three girth three vertices in $\delta G$ on which $f$ is $3$, violating Lemma~\ref{lem:Bmax}.
We cannot be in case~\ref{case:f23} as then $j=2$ and in case~\ref{case:f23} we have $f(v_{j+1})\le 2$. 
This leaves cases~\ref{case:f3233} and~\ref{case:f2233}, but in both of these cases $R$ contains no vertices on which $f$ is three and so the 2-chord $v_sww'$ contradicts Lemma~\ref{lem:2chords} because $f(v_s) = 2$ and $B$ is an independent set.

It follows that $w'$ is adjacent to the endpoint of $R$ that is not $v_j$; and since $g(v_0) \geq 5$ (if it exists), inspecting the cases shows $w' = v_1$. 
Given the cases ruled out above, we are now in a position to establish the essential structure described by the lemma.

Since $v_1 \notin B$, we have that $f(v_1)=3$. 
Since moreover $v_1\notin R$, these facts rule out every case except~\ref{case:f3233} and~\ref{case:f323}.
Recall that we have $w\in \tilde B\setminus V(\delta G)$ which we know is adjacent to $w'=v_1$ and to $v_s$ and $v_t$ with $1<s<t$. 
If $t \geq 4$ then the 2-chord $v_1wv_t$ contradicts Lemma \ref{lem:2chords}, and hence $t = 3$. This gives $s=2$, establishes~\ref{itm:BBw}, and gives $g(v_1)=g(v_2)=g(v_3)=3$. 

If $v_2 \notin B$, then the 2-chord $v_1wv_3$ contradicts Lemma~\ref{lem:2chords}, and hence $v_2\in B$ with $f(v_2)=2$. 
The fact that $B$ is independent then forces $f(v_1)=f(v_3)=3$ and hence $v(Q_2)=\emptyset$ and $v(Q_3)\ge1$.
But this rules out case~\ref{case:f3233} as we have cannot have both $v(R)\ge 2$ and $v(Q_2)=0$ in this case.
Since the only remaining case we can be in is~\ref{case:f323}, this establishes~\ref{itm:BBdefofR}. 
By the planarity of $G$, there is at most one vertex in $\tilde B \setminus V(\delta G)$ adjacent to all of $v_1$, $v_2$, and $v_3$, and thus the vertex $w$ as described in the lemma is unique. 
This means that $\tilde B$ induces exactly one edge, as claimed by~\ref{itm:BBunique}. 
Finally, item~\ref{itm:BBv2} follows from the planarity of $G$ and Lemma~\ref{lem:triangles}: the interiors of the triangles $wv_1v_2w$ and $wv_2v_3w$ are empty.
\end{proof} 

\begin{lemma}\label{lem:AAstructure}
    If $\tilde A$ is not an independent set in $\tilde G$, then the following hold:
    \begin{enumerate}[label=\textup{(\roman*)}]
        \item $G$ contains a $3$-chord $v_1w_1w_2v_3$ such that $g(w_1) = g(w_2) = 5$,\label{itm:AA3chord}
        \item\label{itm:AAunique} $\tilde A$ induces exactly one edge $w_1w_2$ in $\tilde G$,
        \item\label{itm:onlynbrsw1w2} the only neighbour of $w_1$ in $\delta 
        G-P$ is $v_1$, and the only neighbour of $w_2$ in $\delta G-P$ is $v_3$, and\item\label{itm:AAdefofR}
        $R$ is defined according to \ref{case:f2233} or \ref{case:f223}.
    \end{enumerate}
\end{lemma}
\begin{proof}
    Suppose that $\tilde A$ is not an independent set, and let $uu'$ be an edge induced by $\tilde A$. 
    Since $A$ is independent, we may assume that $u\in \tilde A\setminus A$.
    Note that every such $u$ is adjacent in $G$ to a vertex $v$ in $R$. Moreover, inspecting the cases~\ref{handle:vP''<=2}--\ref{handle:f223} shows every vertex in $\tilde A \setminus A$ (and in particular, $u$) is not in $V(\delta G)$.

    If $u$ is adjacent in $\tilde G$ to a vertex $u'$ in $\tilde A \cap A$, it follows that $u'uv$ is a 2-chord that violates Lemma \ref{lem:2chords}, a contradiction. 
    Thus, $u' \in \tilde A \setminus A$. 
    Since $u$ and $u'$ each have neighbours in $R$, there exist vertices $v_s$ and $v_t$ in $R$ where $uv_s \in E(G)$ and $u'v_t \in E(G)$. Since $g(u) \geq 5$, the vertices $v_s$ and $v_t$ are distinct and non-adjacent. Hence $v_suu'v_t$ is a 3-chord, and so \ref{itm:AA3chord} holds. 
    
    By Lemma~\ref{lem:3chords} (the case where both internal vertices of the 3-chord have girth at least five), the 3-chord $v_suu'v_t$ separates $G$ into $G_1$ and $G_2$, where $G_2$ has the structure of an exceptional canvas of type~\ref{ex:A}. 
    By Observation~\ref{obs:whereistheAvert}, we conclude that $v_2 \in A$, and without loss of generality  $v_s = v_1$ and $v_t = v_3$.
    Moreover, since $R$ contains $v_1$ and $v_3$, we have that $R$ was defined according to \ref{case:f2233} or \ref{case:f223}, and thus~\ref{itm:AAdefofR} holds. Statement~\ref{itm:onlynbrsw1w2} is an easy consequence of Lemma \ref{lem:2chords} and the fact that $g(w_1) = g(w_2) = 5$.

    Suppose now that $\tilde A$ induces two distinct edges, say $uu'$ and $ww'$ (note that it may be the case that $\{u,u'\}\cap \{w,w'\}\ne \emptyset$). 
    As argued above, there must now exist distinct 3-chords $v_1uu'v_3$ and $v_1ww'v_3$. 
    As $G$ is planar, without loss of generality $u'$ lies in the bounded face of the cycle $v_1ww'v_3v_2v_1$. This contradicts Lemma \ref{lem:5cycles}, and hence~\ref{itm:AAunique} holds.
\end{proof}

Note that if $\tilde B$ is not independent, then by Lemma \ref{lem:BBstructure}\ref{itm:BBdefofR} we have that $R$ was defined according to \ref{case:f323}. 
If $\tilde A$ is not independent, then by Lemma \ref{lem:AAstructure}\ref{itm:AAdefofR} we have that $R$ was defined according to \ref{case:f2233} or \ref{case:f223}. 
As the cases are mutually exclusive, exactly one of $\tilde A$ and $\tilde B$ is not independent. 
Combined with our earlier analysis, we have the following observation.

\begin{obs}\label{obs:almostcanvas}
    $\tilde{K}$ satisfies all canvas properties except exactly one of the following: either $\tilde{B}$ is not independent, or $\tilde{A}$ is not independent. 
\end{obs}

In the following sections, we complete the proof of Theorem~\ref{thm:inductive} by tackling each case separately.
In Section \ref{sec:tildeB} we handle the case where $\tilde B$ induces an edge, and in Section \ref{sec:tildeA} we consider the case where $\tilde A$ induces an edge.

\section{Handling the edge in \texorpdfstring{$\tilde G[\tilde B]$}{tildeG[tildeB]}}\label{sec:tildeB}

In this section we assume that $\tilde{B}$ is not independent, and therefore the necessary structure established by Lemma~\ref{lem:BBstructure} exists. 
The arguments in this section are inspired by those in \cite{dvovrak20175} and \cite{bernshteyn2024weak} (itself inspired by \cite{dvovrak20175}). The relevant proof in \cite{dvovrak20175} involves a coordinated colouring choice, which in general has no analogue in the weak degeneracy setup; our reductions are thus more akin to those in \cite{bernshteyn2024weak}. Still, the analysis requires more care as the canvases in \cite{bernshteyn2024weak} do not involve the set $A$, for instance, so their weak degeneracy function is substantially simpler. 
It is also not valid to assume that $G$ is a near-triangulation in our setup.

The following lemma shows that when $\tilde B$ is not an independent set, $K$ is not a counterexample to Theorem \ref{thm:inductive}, thus reaching the desired contradiction.

\begin{lemma}\label{lem:BBhandle}
    If $\tilde{B}$ is not independent then $K$ is not a counterexample to Theorem \ref{thm:inductive}.
\end{lemma}

\begin{proof}
By Lemma \ref{lem:BBstructure}, if $\tilde B$ is not independent then $\tilde{G}[\tilde{B}]$ contains exactly one edge $v_1w$, where the vertex $w$ is also adjacent to $v_2$ and $v_3$, and we are in case~\ref{case:f323} with $R=v_2v_3$. 
The lemma also gives that $|N(v_2)|=\{v_1,v_3,w\}$.

Let $\sigma$ be the sequence 
\[ \sigma = \delsave(v_3,v_4), \del(v_0), \delsave(v_1,v_2). \]
Let $f_K(v) = f(v)-|N(v) \cap V(P)|$, and let $(\overline G, \overline f)$ be the graph-function pair obtained by performing the sequence of operations $\sigma$ starting from $(G-P, f_K)$.
Then $v_2\in \overline G$ with $\deg_{\overline G}(v_2)=1$ by Lemma~\ref{lem:BBstructure}\ref{itm:BBv2}, and $\overline f(v_2)=1$. 
Let $\hat G := G - \{v_0,v_1,v_2,v_3\}$ and let $\hat f(v) = \overline f(v) + |N(v)\cap P|$ be a function with domain $V(\hat G)\cup\{v_2\}$. 
Let $\hat{A}$ be the set of vertices $v \in  V(\delta \hat{G}) \setminus V(P)$ with $\hat{f}(v) \leq 1$, and let $\hat{B}$ be the set of vertices $v \in V(\delta \hat{G}) \setminus V(P)$ with $\hat{f}(v) = 2$ and $g_{\hat G}(v) = 3$. 

\begin{figure}[ht]\centering
    \begin{tikzpicture}[scale=1]

    \node[tinylabel] (flab) at (1,1) {$f$:};
    
    \node (ud) at (0,0){$\ldots$};
    \node[precoloured,label={above:$u_1$}] (u1) at (1,0){};
    
    \node[vertex,paren,label={above:$v_0$},fill=black] (v0) at (2,0){};
    \node[tinylabel] (fv0) at (2,1) {$1$};

    \node[vertex,label={above:$v_1$},fill=black] (v1) at (3,0){};
    \node[tinylabel] (fv1) at (3,1) {$3$};
    
    \node[vertex,label={above:$v_2$}] (v2) at (4,0){};
    \node[tinylabel] (fv2) at (4,1) {$2$};
    
    \node[vertex,label={above:$v_3$},fill=black] (v3) at (5,0){};
    \node[tinylabel] (fv3) at (5,1) {$3$};

    \node[vertex,label={above:$v_4$}] (v4) at (6,0){};
    \node[tinylabel] (fv3) at (6,1) {$\le 2$};
    
    \node (vd) at (7,0) {$\ldots$};

    \node[vertex,label={below:$w$}] (w) at (4,-1){};
    \node[tinylabel] (fv1) at (4,-2) {$f$: $4$};

    \draw (ud) -- (u1) -- (v0) -- (v1) -- (v2) -- (v3) -- (v4) -- (vd);
    \draw (v1) -- (w) -- (v2);
    \draw (v3) -- (w);

    \draw[dashed] (3.5,1.5) -- (3.5,0) arc (180:270:0.5) -- (5,-0.5) arc (-90:0:0.5) -- (5.5,1.5);
    \node (R) at (4.0,1.5) {$R$};

  
\end{tikzpicture}
    \captionsetup{width=0.9\textwidth}
    \caption{Necessary graph structure in the case that $\tilde B$ is not independent. 
    The black vertices are removed by the sequence of operations $\sigma$ and $\hat G = G - \{v_0,v_1,v_2,v_3\}$.
    Note that it may be the case that $v_4 = u_k$ (and hence $v_4\in V(P)$).
    \label{fig:BB}}
\end{figure}
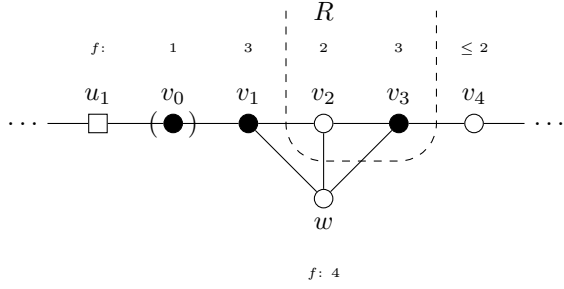

We will argue that $\hat{K}:=(\hat{G}, P, \hat{A}, \hat{B}, \hat{f})$ is an unexceptional canvas.
Then by minimality, $\hat K$ is $\hat f$-degenerate and hence there exists a legal sequence of operations $\tau$ which, starting with $(\hat G-P, \hat f_{\hat K})$ reaches the empty graph. 
We claim that the sequence $\tau$ is legal starting with the graph $\overline G$ (which we think of as $\hat G -P +v_2$) and function $\hat f$ (which we carefully defined such that $\hat f(v_2)=1$). 
This is because $\tau$ is legal starting with $(\overline G-v_2,\hat f)$, and because $v_2$ has degree $1$ in $\overline G$. Since $\hat f(v_2)=1$, after $\tau$ is performed the remaining graph-function pair consists of the isolated vertex $v_2$ with function value zero. 
If we also show that $\sigma$ is a legal sequence when starting with $(G-P, f_K)$, then the sequence of operations $\sigma\tau\del(v_2)$ is legal when starting with $(G-P, f_K)$ and shows that $K$ is $f$-degenerate---a contradiction.

\begin{claim}\label{claim:Khatiscanvas}
    $\hat{K}$ is a canvas and $\hat B\setminus B=\{w\}$.
\end{claim}

\begin{cproof}[Proof of claim]     
First, note that~\ref{canv:acceptable} and~\ref{canv:fint} hold by Observation~\ref{obs:subcanvas} since $K$ is a canvas with the same path $P$ as $\hat K$, and the interior of $\hat G$ is a subset of the interior of $G$.
Next, we claim that \ref{canv:fCn3} and \ref{canv:fC3} hold. 
This follows from the fact that for all vertices $v\in V(\delta G) \cap V(\hat{G})$ we have $\hat f(v)=f(v)$ which we now establish. 
This clearly holds for any such $v$ with no neighbour that is removed by $\sigma$, and that is all such $v$ except $v_4$. 
Since $\delta G$ has no 1-chord by Lemma~\ref{lem:1chord}, the only neighbour of $v_4$ removed by $\sigma$ is $v_3$, which is removed with $\delsave(v_3,v_4)$. Hence, $\hat f(v_4)=f(v_4)$ as required. 
Then since $K$ is a canvas, \ref{canv:fCn3} and \ref{canv:fC3} hold for $\hat K$.

Recall that $\hat{B}$ is defined as containing only vertices $v$ of girth $g_{\tilde G(v)}=3$ for which $\hat f(v)=2$, and so~\ref{canv:fB} holds by definition. 
Next, we argue~\ref{canv:B} holds: it remains to check that $\hat B$ is an independent set in $\hat G$.
Since $B$ is independent, every edge induced by $\hat B$ contains a vertex in $\hat B\setminus B$. 
Note that since $\hat f$ is equal to $f$ on $V(\hat{G})\cap V(\delta G)$, every vertex in $\hat B \setminus B$ is in $V(\hat G)\setminus V(\delta G)$. 
If there exists a vertex $w' \neq w$ in $\hat{B} \setminus B$, then it must be that $w'$ is adjacent to a pair of vertices from $\{v_0,v_1,v_3\}$ as these are the vertices removed by $\sigma$ in forming $\hat G$. Note that while $v_2\notin V(\hat G)$ it was not removed by $\sigma$ in the construction of $\hat f$.

The pair is not $\{v_0,v_1\}$ as (if $v_0$ exists) $g(v_0)\ge 5$. 
But the pair is also not $\{v_0,v_3\}$ or $\{v_1,v_3\}$, because since $w' \not \in N(v_2)$, this would yield a 2-chord contradicting Lemma~\ref{lem:2chords}. See Figure~\ref{fig:BhatminusB}.

Hence $\hat{B} \setminus B \subset \{w\}$, and it is easy to check that $w\in \hat B$ so this is equality. 
Moreover, $w$ is not adjacent to a vertex $w'$ in $\hat{B}\cap B$, as otherwise $v_1ww'$ is also a 2-chord violating Lemma~\ref{lem:2chords} (since we know that $w' \neq v_3$ in this case). See Figure~\ref{fig:BhatcapB}.
It follows that $\hat{B}$ is independent and so~\ref{canv:B} holds.

We next show that~\ref{canv:fA} holds. 
As above, $\hat f$ agrees with $f$ on $V(\hat{G})\cap V(\delta G)$, so every vertex $v \in \hat{A}\setminus A$ is in $V(\hat G) \setminus V(\delta G)$. 
Let $v$ be a vertex in $\hat{A} \setminus A$. 
Note that since $v_0$, if it exists, has girth at least five, no vertex is adjacent to every vertex in $\{v_0,v_1,v_3\}$. It follows that every vertex $u$ in $\hat{G}$ has $\hat{f}(u) \geq f(u) - 2$.
Hence since $\hat{f}(v) \leq 1$, we have that $g(v) \geq 4$. 
Moreover, $v$ is adjacent to at most one of $v_1$ and $v_3$, as otherwise since $g(v) \geq 4$, we have that $v_1vv_3$ is a 2-chord violating Lemma \ref{lem:2chords} (cf.\ the vertex labelled $w'''$ in Figure~\ref{fig:BhatminusB}).
Then $\hat{f}(v) \geq f(v)-1$, and since $\hat{f}(v) \leq 1$ by the definition of $\hat{A}$, we have that $\hat{f}(v) = f(v)-1 = 1$ and hence $f(v) = 2$. 
Thus $g(v) \geq 5$ since $K$ is a canvas. Since $v \in \hat{A}\setminus A$ was arbitrary, it follows that $\hat{A}$ contains only vertices $v$ of girth at least five and with $\hat{f}(v) = 1$. 
Thus, \ref{canv:fA} holds.

Finally, we show that the independence required for~\ref{canv:A} holds. 
To see this, suppose for contradiction that $\hat{A}$ contains vertices $v$ and $v'$ with $vv' \in E(\hat{G})$. Since $A$ is independent, at least one of $v$ and $v'$ is in $\hat{A}\setminus A$.  
Without loss of generality, suppose that $v \in \hat{A}\setminus A$. 
Then there exists $i \in \{0,1,3\}$ such that $vv_i \in E(G)$. 
Since $g(v) \geq 5$, if $v'\in \delta(G)$ we have that $v_ivv'$ is a 2-chord violating Lemma~\ref{lem:2chords}. 
Thus, we may assume that both $v$ and $v'$ are in $\hat{A} \setminus V(\delta G)$, and hence there exists $i' \in \{0,1,3\}$ such that $v'v_{i'} \in E(G)$. 
We have $i \neq i'$ since $g(v) \geq 5$. 
But then given the existence of $w$, the 3-chord $v_ivv'v_{i'}$ violates Lemma~\ref{lem:3chords} because $g(v_2)=3$. See Figure~\ref{fig:AhatAhatedge}.
\end{cproof}

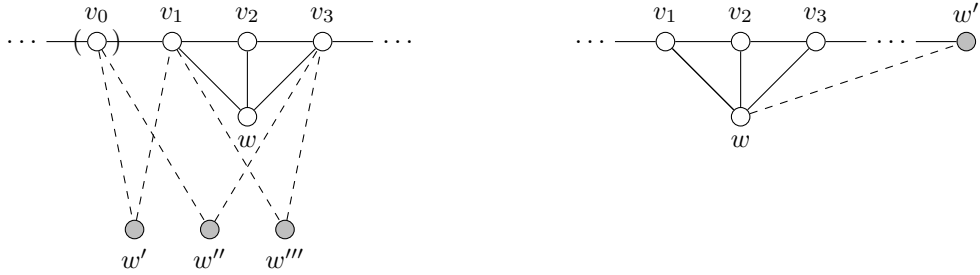
\begin{figure}[ht]\centering
    \subcaptionbox{Vertices $w'$, $w''$, $w'''$, none of which can exist in $\hat B\setminus B$ due to girth the girth of $v_0$ ($w'$), or 2-chords violating Lemma~\ref{lem:2chords} ($w''$ and $w'''$).\label{fig:BhatminusB}}[0.4\textwidth]{\begin{tikzpicture}[scale=1]

    \coordinate (NE) at (6.5,1);
    \coordinate (SW) at (0.5,-3.5);
    \useasboundingbox (SW) rectangle (NE);

    \node (ud) at (1,0){$\ldots$};
    \node[vertex,paren,label={above:$v_0$}] (v0) at (2,0){};
    \node[vertex,label={above:$v_1$}] (v1) at (3,0){};
    \node[vertex,label={above:$v_2$}] (v2) at (4,0){};
    \node[vertex,label={above:$v_3$}] (v3) at (5,0){};
    \node (vd) at (6,0) {$\ldots$};

    \node[vertex,label={below:$w$}] (w) at (4,-1){};

    \draw (ud) -- (v0) -- (v1) -- (v2) -- (v3) -- (vd);
    \draw (v1) -- (w) -- (v2);
    \draw (v3) -- (w);

    \node[nvertex, label={below:$w'$}] (w') at (2.5,-2.5){};
    \node[nvertex, label={below:$w''$}] (w'') at (3.5,-2.5){};
    \node[nvertex, label={below:$w'''$}] (w''') at (4.5,-2.5){};

    \draw[dashed] (v0) -- (w') -- (v1);
    \draw[dashed] (v0) -- (w'') -- (v3);
    \draw[dashed] (v1) -- (w''') -- (v3);
    
\end{tikzpicture}

\begin{tikzpicture}[scale=1]

    \coordinate (NE) at (7.5,1);
    \coordinate (SW) at (1.5,-3.5);
    \useasboundingbox (SW) rectangle (NE);
    
    \node (ud) at (2,0){$\ldots$};
    \node[vertex,label={above:$v_1$}] (v1) at (3,0){};
    \node[vertex,label={above:$v_2$}] (v2) at (4,0){};
    \node[vertex,label={above:$v_3$}] (v3) at (5,0){};
    \node (vd) at (6,0) {$\ldots$};

    \node[vertex,label={below:$w$}] (w) at (4,-1){};

    \draw (ud) -- (v1) -- (v2) -- (v3) -- (vd);
    \draw (v1) -- (w) -- (v2);
    \draw (v3) -- (w);

    \node[nvertex, label={above:$w'$}] (w') at (7,0){};
    \draw (vd) -- (w');
    \draw (v1) -- (w);
    \draw[dashed] (w) -- (w');
\end{tikzpicture}

\begin{tikzpicture}[scale=1]

    \coordinate (NE) at (6.5,1);
    \coordinate (SW) at (1.5,-3);
    \useasboundingbox (SW) rectangle (NE);

    \node (ud) at (2,0){$\ldots$};
    \node[vertex,label={above:$v_1$}] (v1) at (3,0){};
    \node[vertex,label={above:$v_2$}] (v2) at (4,0){};
    \node[vertex,label={above:$v_3$}] (v3) at (5,0){};
    \node (vd) at (6,0) {$\ldots$};

    \node[vertex,label={below:$w$}] (w) at (4,-1){};

    \draw (ud) -- (v1) -- (v2) -- (v3) -- (vd);
    \draw (v1) -- (w) -- (v2);
    \draw (v3) -- (w);

    \node[vertex, label={below:$v\phantom{'}$}] (v) at (3,-2){};
    \node[vertex, label={below:$v'$}] (v') at (5,-2){};

    \draw (v1) -- (v);
    \draw (v') -- (v3);
    \draw[dashed] (v) -- (v');
\end{tikzpicture}}%
    \hspace*{0.05\textwidth}
    \subcaptionbox{The vertex $w$ cannot have a neighbour $w'\in \hat B\cap B$ as then $v_1ww'$ violates Lemma~\ref{lem:2chords}.\label{fig:BhatcapB}}[0.4\textwidth]{\begin{tikzpicture}[scale=1]

    \coordinate (NE) at (6.5,1);
    \coordinate (SW) at (0.5,-3.5);
    \useasboundingbox (SW) rectangle (NE);

    \node (ud) at (1,0){$\ldots$};
    \node[vertex,paren,label={above:$v_0$}] (v0) at (2,0){};
    \node[vertex,label={above:$v_1$}] (v1) at (3,0){};
    \node[vertex,label={above:$v_2$}] (v2) at (4,0){};
    \node[vertex,label={above:$v_3$}] (v3) at (5,0){};
    \node (vd) at (6,0) {$\ldots$};

    \node[vertex,label={below:$w$}] (w) at (4,-1){};

    \draw (ud) -- (v0) -- (v1) -- (v2) -- (v3) -- (vd);
    \draw (v1) -- (w) -- (v2);
    \draw (v3) -- (w);

    \node[nvertex, label={below:$w'$}] (w') at (2.5,-2.5){};
    \node[nvertex, label={below:$w''$}] (w'') at (3.5,-2.5){};
    \node[nvertex, label={below:$w'''$}] (w''') at (4.5,-2.5){};

    \draw[dashed] (v0) -- (w') -- (v1);
    \draw[dashed] (v0) -- (w'') -- (v3);
    \draw[dashed] (v1) -- (w''') -- (v3);
    
\end{tikzpicture}

\begin{tikzpicture}[scale=1]

    \coordinate (NE) at (7.5,1);
    \coordinate (SW) at (1.5,-3.5);
    \useasboundingbox (SW) rectangle (NE);
    
    \node (ud) at (2,0){$\ldots$};
    \node[vertex,label={above:$v_1$}] (v1) at (3,0){};
    \node[vertex,label={above:$v_2$}] (v2) at (4,0){};
    \node[vertex,label={above:$v_3$}] (v3) at (5,0){};
    \node (vd) at (6,0) {$\ldots$};

    \node[vertex,label={below:$w$}] (w) at (4,-1){};

    \draw (ud) -- (v1) -- (v2) -- (v3) -- (vd);
    \draw (v1) -- (w) -- (v2);
    \draw (v3) -- (w);

    \node[nvertex, label={above:$w'$}] (w') at (7,0){};
    \draw (vd) -- (w');
    \draw (v1) -- (w);
    \draw[dashed] (w) -- (w');
\end{tikzpicture}

\begin{tikzpicture}[scale=1]

    \coordinate (NE) at (6.5,1);
    \coordinate (SW) at (1.5,-3);
    \useasboundingbox (SW) rectangle (NE);

    \node (ud) at (2,0){$\ldots$};
    \node[vertex,label={above:$v_1$}] (v1) at (3,0){};
    \node[vertex,label={above:$v_2$}] (v2) at (4,0){};
    \node[vertex,label={above:$v_3$}] (v3) at (5,0){};
    \node (vd) at (6,0) {$\ldots$};

    \node[vertex,label={below:$w$}] (w) at (4,-1){};

    \draw (ud) -- (v1) -- (v2) -- (v3) -- (vd);
    \draw (v1) -- (w) -- (v2);
    \draw (v3) -- (w);

    \node[vertex, label={below:$v\phantom{'}$}] (v) at (3,-2){};
    \node[vertex, label={below:$v'$}] (v') at (5,-2){};

    \draw (v1) -- (v);
    \draw (v') -- (v3);
    \draw[dashed] (v) -- (v');
\end{tikzpicture}}%
    \captionsetup{width=0.9\textwidth}
    \caption{Structure relevant to the proof of Claim~\ref{claim:Khatiscanvas}.\label{fig:BBdetails}}
\end{figure}

\begin{figure}[ht]\centering
    \begin{tikzpicture}[scale=1]

    \coordinate (NE) at (6.5,1);
    \coordinate (SW) at (0.5,-3.5);
    \useasboundingbox (SW) rectangle (NE);

    \node (ud) at (1,0){$\ldots$};
    \node[vertex,paren,label={above:$v_0$}] (v0) at (2,0){};
    \node[vertex,label={above:$v_1$}] (v1) at (3,0){};
    \node[vertex,label={above:$v_2$}] (v2) at (4,0){};
    \node[vertex,label={above:$v_3$}] (v3) at (5,0){};
    \node (vd) at (6,0) {$\ldots$};

    \node[vertex,label={below:$w$}] (w) at (4,-1){};

    \draw (ud) -- (v0) -- (v1) -- (v2) -- (v3) -- (vd);
    \draw (v1) -- (w) -- (v2);
    \draw (v3) -- (w);

    \node[nvertex, label={below:$w'$}] (w') at (2.5,-2.5){};
    \node[nvertex, label={below:$w''$}] (w'') at (3.5,-2.5){};
    \node[nvertex, label={below:$w'''$}] (w''') at (4.5,-2.5){};

    \draw[dashed] (v0) -- (w') -- (v1);
    \draw[dashed] (v0) -- (w'') -- (v3);
    \draw[dashed] (v1) -- (w''') -- (v3);
    
\end{tikzpicture}

\begin{tikzpicture}[scale=1]

    \coordinate (NE) at (7.5,1);
    \coordinate (SW) at (1.5,-3.5);
    \useasboundingbox (SW) rectangle (NE);
    
    \node (ud) at (2,0){$\ldots$};
    \node[vertex,label={above:$v_1$}] (v1) at (3,0){};
    \node[vertex,label={above:$v_2$}] (v2) at (4,0){};
    \node[vertex,label={above:$v_3$}] (v3) at (5,0){};
    \node (vd) at (6,0) {$\ldots$};

    \node[vertex,label={below:$w$}] (w) at (4,-1){};

    \draw (ud) -- (v1) -- (v2) -- (v3) -- (vd);
    \draw (v1) -- (w) -- (v2);
    \draw (v3) -- (w);

    \node[nvertex, label={above:$w'$}] (w') at (7,0){};
    \draw (vd) -- (w');
    \draw (v1) -- (w);
    \draw[dashed] (w) -- (w');
\end{tikzpicture}

\begin{tikzpicture}[scale=1]

    \coordinate (NE) at (6.5,1);
    \coordinate (SW) at (1.5,-3);
    \useasboundingbox (SW) rectangle (NE);

    \node (ud) at (2,0){$\ldots$};
    \node[vertex,label={above:$v_1$}] (v1) at (3,0){};
    \node[vertex,label={above:$v_2$}] (v2) at (4,0){};
    \node[vertex,label={above:$v_3$}] (v3) at (5,0){};
    \node (vd) at (6,0) {$\ldots$};

    \node[vertex,label={below:$w$}] (w) at (4,-1){};

    \draw (ud) -- (v1) -- (v2) -- (v3) -- (vd);
    \draw (v1) -- (w) -- (v2);
    \draw (v3) -- (w);

    \node[vertex, label={below:$v\phantom{'}$}] (v) at (3,-2){};
    \node[vertex, label={below:$v'$}] (v') at (5,-2){};

    \draw (v1) -- (v);
    \draw (v') -- (v3);
    \draw[dashed] (v) -- (v');
\end{tikzpicture}%
    \captionsetup{width=0.9\textwidth}
    \caption{In the proof of Claim~\ref{claim:Khatiscanvas}, an edge $vv'$ that cannot be induced by $\hat A$ as then the 3-chord $v_1vv'v_3$ violates Lemma~\ref{lem:3chords}.\label{fig:AhatAhatedge}}
\end{figure}
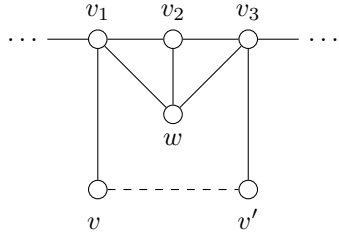
 
Thus, $\hat{K}$ is indeed a canvas. We argue that it is unexceptional. 
If $\hat{K}$ is an exceptional canvas of type~\ref{ex:B}, then since $K$ is unexceptional, since $\tilde{B}\setminus B = \{w\}$ by Claim \ref{claim:Khatiscanvas}, and since $K$ and $\hat K$ share the same path $P$, we have that $N(w)$ contains three vertices of $P$. 
But this leads to a 2-chord contradicting Lemma~\ref{lem:2chords} (namely $u_1wu_3$). 
Similarly, $\hat{K}$ is not an exceptional canvas of type~\ref{ex:A} as otherwise there is a 2-chord violating Lemma~\ref{lem:2chords}. 
Suppose finally that $\hat{K}$ is an exceptional canvas of type~\ref{ex:AB}. Then there exist vertices $u' \in \hat{B}$ and $u'' \in \hat{A}$ such that $u'u'' \in E(G)$, and each of $u'$ and $u''$ is adjacent to an endpoint of $P$. 
Since $K$ is not exceptional, either $u' \in \hat{B} \setminus B$ or $u'' \in \hat{A}\setminus A$. 
If $u' \in \hat{B}\setminus B$, then by Claim \ref{claim:Khatiscanvas} we have that $u' = w$. 
Let $x$ be the endpoint of $P$ adjacent to $w$. 

Since $\hat{K}$ is an unexceptional canvas and $K$ is a minimum counterexample, as outlined above, provided we show that the sequence $\sigma$ that we defined to remove $v_3$, $v_0$ and $v_1$ from $G-P$ is legal we can conclude that $K$ is weakly $f$-degenerate.

\clearpage
\begin{claim}
    The sequence of operations $\sigma$ is legal when starting with $(G-P, f_K)$.
\end{claim}
\begin{cproof}
    Since we verified that $\hat K$ is a canvas in Claim~\ref{claim:Khatiscanvas}, we know that no vertex in $V(\hat G)-P$ has its function value fall below zero through the execution of $\sigma$. 
    We also showed that after executing $\sigma$ the resulting function has value $1$ on $v_2$. 
    Thus, it suffices to check that
    \begin{itemize}
        \item when we call $\delsave(v_3,v_4)$ and $v_4\notin P$ we have $f_K(v_3) > f_K(v_4) \ge 0$,
        \item calling $\del(v_0)$ does not cause the function to fall below zero on $v_1$, and 
        \item when we call $\delsave(v_1, v_2)$ the ``save'' is legal because the relevant function is strictly larger on $v_1$ than on $v_2$.
    \end{itemize}
    The first item holds because $R$ is defined according to  case~\ref{case:f323} by Lemma~\ref{lem:BBstructure}\ref{itm:BBdefofR} and so $f_K(v_4)\le 2$ when $v_4\notin V(P)$.
    Given what we know about $\hat f$, the second item is implied by the third. 
    The third item holds because $f_K(v_1)=3$ and $f_K(v_2)=2$, and the combined effect of the first two operations in $\sigma$ is that each falls by exactly one.
\end{cproof}

\noindent
This completes the proof of Lemma~\ref{lem:BBhandle}.
\end{proof}

\section{Handling the edge in \texorpdfstring{$\tilde G[\tilde A]$}{tildeG[tildeA]}}\label{sec:tildeA}

In this section, we complete the proof of Theorem \ref{thm:inductive} by handling the case where $\tilde G[\tilde A]$ contains an edge. Recall that $\tilde K = (\tilde G, P, \tilde A, \tilde B, \tilde f)$ is obtained from $K$ by removing $R$ via the procedure described in Subsection \ref{subsec:removeR}. Since $\tilde{A}$ is not independent, the necessary structure established by Lemma~\ref{lem:AAstructure} exists. 
In particular, by Lemma~\ref{lem:AAstructure}\ref{itm:AA3chord}, $G$ contains a 3-chord $v_1w_1w_2v_3$ such that $g(w_1) = g(w_2) = 5$ and $v_2 \in A$; and by Lemma~\ref{lem:AAstructure}\ref{itm:AAdefofR}, $R$ is defined according to \ref{case:f2233} or \ref{case:f223}. 
See Figure~\ref{fig:AA}.

\begin{figure}[ht]\centering
    \begin{tikzpicture}[scale=1]

    \node (ud) at (0,0){$\ldots$};
    \node[precoloured,label={above:$u_1$}] (u1) at (1,0){};
    \node[vertex,paren,label={above:$v_0$}] (v0) at (2,0){};
    \node[vertex,label={above:$v_1$}] (v1) at (3,0){};
    \node[vertex,label={above:$v_2$}] (v2) at (4,0){};
    \node[vertex,label={above:$v_3$}] (v3) at (5,0){};
    \node (vd) at (6,0) {$\ldots$};

    \node[vertex,label={below:$w_1$}] (w1) at (3.5,-1){};
    \node[vertex,label={below:$w_2$}] (w2) at (4.5,-1){};
    \draw (v1) -- (w1) -- (w2) -- (v3);
    
    \draw (ud) -- (u1) -- (v0) -- (v1) -- (v2) -- (v3) -- (vd);

    \draw[dashed] (2.5,1) -- (2.5,0) arc (180:270:0.5) -- (5.5,-0.5);
    \node (R) at (3.5,1) {$R$};

  
\end{tikzpicture}
    \captionsetup{width=0.9\textwidth}
    \caption{Necessary graph structure in the case that $\tilde A$ is not independent.\label{fig:AA}} 
\end{figure}
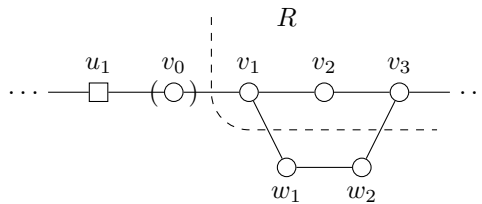

The general approach to handle this case will be to remove $w_1$ and $w_2$ \emph{as well as} $R$. The following claim is used to argue that it is possible to remove $V(R) \cup \{w_1, w_2\}$ from $G-P$ via a sequence of legal operations.

\begin{claim}\label{claim:w1w2notadjP}
    At most one of $w_1$ and $w_2$ has a neighbour in $P$, and $|N(\{w_1,w_2\}) \cap V(P)| \leq 1$. 
\end{claim}
\begin{cproof}
Suppose not. First suppose that each of $w_1$ and $w_2$ has a neighbour in $P$. Then since $g(w_1) = g(w_2) = 5$ (and $G$ is a plane graph), we have that $P'$ is nonempty (i.e. $v_0$ exists), that $v(P) = 4$, that $w_1$ is adjacent to $u_2$, and $w_2$ is adjacent to $u_4$. But then since $g(w_2) = 5$, we have that $u_4w_2v_3$ is a 2-chord violating Lemma \ref{lem:2chords}, a contradiction. Thus we may assume that there exists a vertex $w \in \{w_1,w_2\}$ with two neighbours in $P$. Since $g(w) = 5$, we have that $v(P)= 4$, and that $w$ is adjacent to both $u_1$ and $u_4$. But then $u_4wu_1$ is again a 2-chord violating Lemma \ref{lem:2chords}. 
\end{cproof}

Let $\ell\in\{1,2\}$ be such that $w_\ell$ has a neighbour in $P$ if such a choice exists; otherwise, let $\ell \in \{1,2\}$ be arbitrary. 
Recall that $f_K(v) = f(v) - |N(v) \cap V(P)|$ for all $v \in V(G)\setminus V(P)$. If $R$ is defined according to \ref{case:f223}, let
\[ 
\sigma = (\delsave(v_j, v_{j+1}), \del(v_{j-1}), \dots, \del(v_2), \delsave(v_1, v_0), \del(w_\ell), \del(w_{3-\ell}))
\]
and if $R$ is defined according to \ref{case:f2233} let 
\[ 
\sigma = (\del(v_j), \del(v_{j-1}), \dots, \del(v_2), \delsave(v_1, v_0), \del(w_\ell), \del(w_{3-\ell})). 
\]
Let $(\bar{G},\bar{f})$ be the graph and function obtained by performing the sequence of operations $\sigma$ starting from from $(G-P,f_K)$. 
Note that this removes $V(R)\cup \{w_1,w_2\}$. 
Let $\hat{G} = G-R-\{w_1,w_2\}$, and let $\hat{f}(v) = \bar{f}(v) + |N(v) \cap V(P)|$ for all $v \in V(\bar{G})$. Let $\hat{A}$ be the set $\{v \in V(\hat{G}) \setminus V(P) : \hat{f}(v) \leq 1\}$, and let $\hat{B}$ be the set $\{v \in V(\hat{G}) \setminus V(P): \hat{f}(v) = 2 \textnormal{ and } g(v) = 3\}$.
Let $\hat{K}: = (\hat{G}, P, \hat{A}, \hat{B}, \hat{f})$. Note that $\hat{B} \subseteq V(\delta \hat{G})$ and $\hat{A} \subseteq V(\delta \hat{G})$.
We will argue that the application of $\sigma$ starting from $(G-P, f_K)$ is legal through an analysis of the structure of $\hat K$.

Note that in both cases~\ref{case:f223} and~\ref{case:f2233}, an initial segment of $\sigma$ is precisely the sequence used to remove $R$ in the relevant cases~\ref{handle:f2233} and \ref{handle:f223}, and hence the initial segment of $\sigma$ which removes $R$ is legal by Lemma~\ref{lem:delRlegal}. 
To argue that the final two operations in $\sigma$ (which remove $w_1$ and $w_2$) are legal, we first show the following.

\begin{claim}\label{claim:delGfine}
    If $R$ is defined according to \ref{case:f223}, then  for each $v \in V(\delta G-P) \cap V(\hat{G})$,  $\hat{f}(v) = f(v)$. 
    If $R$ is defined according to \ref{case:f2233}, then for each $v \in V(\delta G-P) \cap V(\hat{G})$ except $v_{j+1}$,  $\hat{f}(v) = f(v)$. 
\end{claim}
\begin{cproof}
    Let $v \in V(\delta G-P) \cap V(\hat{G})$. It suffices to account for when and how neighbours of $v$ are removed by $\sigma$.
    Since $\delta G$ is chordless by  Lemma~\ref{lem:1chord}, and $v$ is not adjacent to $w_1$ or $w_2$ by Lemma~\ref{lem:AAstructure}\ref{itm:onlynbrsw1w2}, $\hat{f}(v) = f(v)$ unless $v\in \{v_0,v_{j+1}\}$. Moreover, only the removal of $v_1$ and $v_j$ can affect $\hat f(v)$ for such a $v$.
    If $v_0$ exists, then $\hat{f}(v_0) = f(v_0)$ since $v_1$ is removed via $\delsave(v_1,v_0)$ in $\sigma$.
    If $v = v_{j+1}$ and $R$ is defined according to \ref{case:f223}, then $v_j$ is removed via $\delsave(v_j,v_{j+1})$, and so similarly $\hat{f}(v_{j+1}) = f(v_{j+1})$.
\end{cproof}

We now show that $V(\hat{G})$ does not contain a vertex adjacent to one of $\{w_1, w_2\}$ as well as to a vertex in $R$. This will be useful in showing both that $w_\ell$ is removed via a legal application of $\del$, and in establishing several of the canvas properties for $\hat{K}$ (namely, that $\hat{f}(v)$ is sufficiently large for all vertices $v \in V(\hat{G}) \setminus V(P)$).

\begin{claim}\label{claim:notadjtow1w2andR}
    $V(\hat G)$ does not contain a vertex adjacent in $G$ to one of $\{w_1,w_2\}$ as well as to a vertex in $R$.
\end{claim}
\begin{cproof}
    Suppose not, and let $v$ be a counterexample. Let $w$ be a vertex in $\{w_1,w_2\}$ adjacent to $v$, and let $v' \in V(R)$ be the vertex in $R$ adjacent to $v$. Let $w'$ be the neighbour of $w$ in $V(R)$. (Recall that $w$ is adjacent to either $v_1$ or $v_3$ since $G$ has the structure established in Lemma~\ref{lem:AAstructure}, and $w' \in \{v_1,v_3\}$ is the only neighbour of $w$ in $R$, as otherwise there is a 2-chord violating Lemma~\ref{lem:2chords}.) If $v \in V(\delta G) \cap  V(\hat{G})$, then since $g(w) \geq 5$, we have that $v_{j+1}ww'$ is a 2-chord violating Lemma~\ref{lem:2chords}, a contradiction. Hence $v \in V(G) \setminus V(\delta G)$. But Observation~\ref{obs:whereistheAvert} means that $v_2$ is the only vertex in $A \cap V(R)$, and so $w'wvv'$ is a 3-chord violating Lemma~\ref{lem:3chords}, again a contradiction.
\end{cproof}

We now argue $w_\ell$ is removed legally. The argument that the removal of $w_{3-\ell}$ is also legal will come after.

\clearpage
\begin{claim}
The removal of $w_\ell$ in $\sigma$ is legal.
\end{claim}
\begin{cproof}
First, we note that $f_K(w_{3-\ell}) = f(w_{3-\ell}) = 2$  by Claim~\ref{claim:w1w2notadjP} and our choice of $\ell$. Moreover, $w_{3-\ell}$ does not have three neighbours in $V(R)\cup \{w_\ell\}$  by Lemma \ref{lem:2chords}. Thus, removing $V(R) \cup \{w_\ell\}$ does not cause the function value of $w_{3-\ell}$ to become negative.

To complete the prove the claim, it suffices to show that immediately before the execution of the operation $\del(w_\ell)$ in $\sigma$, no neighbour of $w_\ell$ in $\hat G$ has function value at most zero. 
This is equivalent to each $v\in V(\hat G)$ having no more than $f_K(v)$ neighbours in $R\cup\{w_\ell\}$. Recall that by Lemma~\ref{lem:delRlegal} and the definition of $\sigma$, all but the last two operations in $\sigma$ form a legal sequence of operations that removes $R$ starting from $(G-P, f_K)$, and hence each $v\in V(\hat{G})$ has no more than $f_K(v)$ neighbours in $V(R)$. 

By Claim \ref{claim:notadjtow1w2andR}, no neighbour of $w_\ell$ has a neighbour in $R$, and so it suffices to verify that no vertex $v$ in $\hat{G}$ with $f_K(v) = 0$ neighbours $w_\ell$. 
To that end, let $v$ be a vertex in $N(w_\ell)\cap V(\hat{G})$ with $f_K(v) = 0$. 
Since $f_K(v) = 0$, $v$ is adjacent to at least one vertex of $P$. Note that if $v$ is in $V(\delta G)$, then letting $v' \in \{v_1,v_3\}$ be the neighbour of $w_\ell$ in $R$, we have that $vw_\ell v'$ is a 2-chord violating Lemma \ref{lem:2chords}. Hence $v \not \in V(\delta G)$. First suppose $g(v) = 3$ (and so $f(v) = 4$). Then since $f_K(v) = 0$, $v$ has four neighbours in $P$. Thus $v(P) = 4$, and all vertices in $P$ have girth three, contradicting that $P$ is acceptable. Next suppose $g(v) = 4$ (and so $f(v) = 3$). Since $f_K(v) = 0$,  $v$ has at least three neighbours in $P$. But this is impossible, since $g(v) = 4$ and $v(P) \leq 4$. Finally, suppose $g(v) \geq 5$ (and so $f(v) = 2$). Then since $f_K(v) = 0$, $v$ has two neighbours in $P$. Since $V(P) \leq 4$ and $g(v) \geq 5$, it follows that $v(P) = 4$ and $v$ neighbours both endpoints of $P$. But then $u_1vu_4$ is a 2-chord violating Lemma \ref{lem:2chords}.

\end{cproof}

Eventually, we will show that $\hat{K}= (\hat{G}, P, \hat{A}, \hat{B}, \hat{f})$ fails to be a canvas. 
In particular, we will show it fails to be a canvas precisely because $\hat{A}$ is not an independent set, which will reveal the necessary structure to complete proof. To do this, we first argue that \emph{if} $\hat{K}$ is a canvas, then it is an unexceptional canvas (which will lead to a contradiction, as then we will show $K$ is not a counterexample). 

\begin{claim}\label{claim:ifcanvthenunex}
    If $\hat{K}$ is a canvas, then it is unexceptional.
\end{claim}
\begin{cproof}
Suppose for a contradiction that $\hat{K}$ is an exceptional canvas. 
If it is an exceptional canvas of type \ref{ex:B}, then since $P$ is acceptable we have $v(P)=3$ and there exists a vertex $u \in \hat{B}$ adjacent to $u_1$, $u_2$, and $u_3$. 
Since $K$ is unexceptional and $\delta G$ is chordless by Lemma \ref{lem:1chord}, we have that $u \not \in V(\delta G)$. 
Thus $f(u) = 4$. Since $u \in \hat{B}$, we have that $\hat{f}(u) = 2$ and hence that $u$ is adjacent in $G$ to two vertices in $V(R)\cup\{w_1,w_2\}$. 
By Claim~\ref{claim:notadjtow1w2andR} and the fact that $g(w_1) = g(w_2) = 5$ (Lemma~\ref{lem:AAstructure}\ref{itm:AA3chord}), $u$ is adjacent to two vertices $v_i, v_{i'}$ in $V(R)$, where without loss of generality $i > i'$. Thus $u_1uv_i$ is a 2-chord as described in Lemma~\ref{lem:2chords}. 
This is a contradiction, since $g(v_2) = 5$ (and if $v_0$ exists, $g(v_0)\geq 5$), and hence $i \geq 3$. 

Next suppose $\hat{K}$ is an exceptional canvas of type \ref{ex:A}. Then $v(P) = 4$, and there exists a vertex $u \in \hat{A}$ adjacent to both $u_1$ and $u_4$. Since $K$ is unexceptional and $\delta G$ is chordless, $u \not \in V(\delta G)$. But then since $g(u) \geq 5$, $u_1uu_4$ is a 2-chord violating Lemma \ref{lem:2chords}.

Finally, suppose $\hat{K}$ is an exceptional canvas of type \ref{ex:AB}. Then $v(P) = 4$ and there exists a vertex $u \in \hat{A}$ adjacent to a second vertex $v \in \hat{B}$ where $u$ and $v$ are each adjacent to distinct endpoints of $v(P)$. Note that either $v \not \in B$ or $u \not \in A$ since $K$ is unexceptional. If both $u \not \in A$ and $v \not \in B$, then using Claim \ref{claim:delGfine}, one of the following holds:
\begin{itemize}
    \item $R$ is defined according to \ref{case:f223}, or $R$ was defined according to \ref{case:f2233} and $v \neq v_{j+1}$. In either case both $u$ and $v$ are in $V(G) \setminus V(\delta G)$. But then there is a 3-chord on the vertices in the set $\{u_1, u_4, v, u\}$ which, given the existence of $v_2$, violates Lemma~\ref{lem:3chords}. See Figure~\ref{fig:f223ornotj1}.
    \item  $R$ is defined according to \ref{case:f2233} and $v = v_{j+1}$. In this case $u,v_{j+1}$, and the endpoint of $P$ adjacent to $u$ form a 2-chord which, given $g(u) \geq 5$, violates Lemma \ref{lem:2chords}. See Figure~\ref{fig:f2233j1}.
\end{itemize}

\begin{figure}[ht]\centering
    \subcaptionbox{Given the existence of $v_2$, the emphasized 3-chord violates Lemma~\ref{lem:3chords}.\label{fig:f223ornotj1}}[0.45\textwidth]{\begin{tikzpicture}[scale=1]
    
    \foreach \i in {1,2,3}{
    	\pgfmathtruncatemacro{\j}{\i+2}
    	\node[vertex,label={above:$v_\i$}] (v\i) at (\j,0){};
    }
    \node (vd1) at (6,0) {$\ldots$};
	\node[vertex,label={above:$v_i\in R$}] (vi) at (7,0){};
    \node (vd0) at (2,0) {$\ldots$};
    \node (vdi) at (8,0) {$\ldots$};
    \draw (vd0) -- (v1);
    \draw (vdi) -- (vi);
	
    \node[vertex,label={below:$w_1$}] (w1) at (3.5,-1){};
    \node[vertex,label={below:$w_2$}] (w2) at (4.5,-1){};
    \draw (w1) -- (w2) -- (v3);
    
    \foreach \i in {1,2} {
    	\pgfmathtruncatemacro{\j}{\i+1}
    	\draw (v\i) -- (v\j);
	}
	\draw (v3) -- (vd1);
	
	\node[vertex,label={below:$w_3$}] (w3) at (4.5, -2) {};
	\node[vertex,label={below:$w_4$}] (w4) at (6.5, -2) {};
	
	\draw[thickedge] (v1) --  (w1) -- (w3) -- (w4) -- (vi);
    \draw (vi) -- (vd1);
  
\end{tikzpicture}

\begin{tikzpicture}[scale=1]

    \foreach \i in {1,2,3}{
    	\pgfmathtruncatemacro{\j}{\i+2}
    	\node[vertex,label={above:$v_\i$}] (v\i) at (\j,0){};
    }
    \node (vd1) at (6,0) {$\ldots$};
	\node[vertex,label={above:$v_i\in R$}] (vi) at (7,0){};
    \node (vd0) at (2,0) {$\ldots$};
    \node (vdi) at (8,0) {$\ldots$};
    \draw (vd0) -- (v1);
    \draw (vdi) -- (vi);
    
    \node[vertex,label={below:$w_1$}] (w1) at (3.5,-1){};
    \node[vertex,label={below:$w_2$}] (w2) at (4.5,-1){};
    \draw (w2) -- (v3);
    
    \foreach \i in {1,2} {
    	\pgfmathtruncatemacro{\j}{\i+1}
    	\draw (v\i) -- (v\j);
	}
	\draw (v3) -- (vd1);
	
	\node[vertex,label={below:$w_3$}] (w3) at (5.5, -2) {};
	\node[vertex,label={below:$w_4$}] (w4) at (6.5, -2) {};
	
	\draw[thickedge] (v1) -- (w1) --  (w2) -- (w3) -- (w4) -- (vi);
    \draw (vi) -- (vd1);
  
\end{tikzpicture}

\begin{tikzpicture}[scale=1]

    \foreach \i in {1,2,3}{
    	\pgfmathtruncatemacro{\j}{\i+2}
    	\node[vertex,label={above:$v_\i$}] (v\i) at (\j,0){};
    }
    \node (vd1) at (2,0) {$\ldots$};
    \node (vd3) at (6,0) {$\ldots$};
    \draw (vd1) -- (v1);
    \draw (vd3) -- (v3);
    
    \node[vertex,label={below:$w_1$}] (w1) at (3.5,-1){};
    \node[vertex,label={below:$w_2$}] (w2) at (4.5,-1){};
    \draw (v1) -- (w1) -- (w2) -- (v3);
    
    \foreach \i in {1,2} {
    	\pgfmathtruncatemacro{\j}{\i+1}
    	\draw (v\i) -- (v\j);
	}
	
	\node[vertex,label={right:$w_3$}] (w3) at (5, -2) {};
	\node[vertex,label={right:$w_4$}] (w4) at (5.5, -1) {};
	
	\draw (w1) -- (w3) -- (w4) -- (v3);
	
	\node[elabel] (e1) at (4,-0.5) {$\emptyset$};
	\node[elabel] (e2) at (5,-1) {$\emptyset$};

\end{tikzpicture}

\begin{tikzpicture}[scale=1]

    \foreach \i in {1,2,3}{
    	\pgfmathtruncatemacro{\j}{\i+2}
    	\node[vertex,label={above:$v_\i$}] (v\i) at (\j,0){};
    }
    \node (vd1) at (2,0) {$\ldots$};
    \node (vd3) at (6,0) {$\ldots$};
    \draw (vd1) -- (v1);
    \draw (vd3) -- (v3);
	
    \node[vertex,label={below:$w_1$}] (w1) at (3.5,-1){};
    \node[vertex,label={below:$w_2$}] (w2) at (4.5,-1){};
    \draw (v1) -- (w1) -- (w2) -- (v3);
    
    \foreach \i in {1,2} {
    	\pgfmathtruncatemacro{\j}{\i+1}
    	\draw (v\i) -- (v\j);
	}
	
	\node[vertex,label={left:$w_3$}] (w3) at (2, -2) {};
	\node[vertex,label={left:$w_4$}] (w4) at (1.5, -1) {};
	
	\draw (w1) -- (w3) -- (w4) -- (v1);

\end{tikzpicture}

\begin{tikzpicture}[scale=1]

    \foreach \i in {1,2,3}{
    	\pgfmathtruncatemacro{\j}{\i+2}
    	\node[vertex,label={above:$v_\i$}] (v\i) at (\j,0){};
    }
	\node[vertex,label={above:$v_4$}] (v4) at (6.5,0){};
    \node (vd1) at (2,0) {$\ldots$};
    \node (vd4) at (7.5,0) {$\ldots$};
    \draw (vd1) -- (v1);
    \draw (vd4) -- (v4);
    
    \node[vertex,label={below:$w_1$}] (w1) at (3.5,-1){};
    \node[vertex,label={below:$w_2$}] (w2) at (4.5,-1){};
    \draw (v1) -- (w1) -- (w2) -- (v3);
    
    \foreach \i in {1,2,3} {
    	\pgfmathtruncatemacro{\j}{\i+1}
    	\draw (v\i) -- (v\j);
	}
	
	\node[vertex,label={below:$w_3$}] (w3) at (5.5, -2) {};
	\node[vertex,label={below:$w_4$}] (w4) at (6.5, -2) {};
	
	\draw (w2) -- (w3) -- (w4) -- (v4);
	

    \draw[dashed] (2.5,0.5) -- (2.5,0) arc (180:270:0.5) --(4,-0.5) arc (-90:0:0.5) -- (4.5,0.5);
\end{tikzpicture}

\begin{tikzpicture}[scale=1]
    \node[precoloured,label={above:$u_1$}] (u1) at (1,0){};
    \begin{scope}[xshift=1cm]
    \begin{scope}[shift={(135:-2cm)},rotate=90, xscale=-1]
    \node[precoloured] (u2) at (105:2cm){};
    \node[precoloured] (u3) at (75:2cm){};
    \end{scope}
    \end{scope}
    \node[precoloured] (u4) at ($ (u3) + (0,-1) $){};

    \begin{scope}[xshift=5cm, xscale=-1]
    \begin{scope}[shift={(135:-2cm)},rotate=90, xscale=-1]  
    \node[vertex,label={right:$v_j$}] (vj) at (105:2cm) {};
    \node[vertex,label={right:$v_{j+1}\in \hat B$}] (vj1) at (75:2cm){};
    \end{scope}
    \end{scope}
    \node[rotate=90] (vd2) at ($ (vj1) + (0,-1) $) {$\ldots$};
    \draw (vj1) -- (vd2);
    
    \node[vertex,label={left:$u\in \hat A$}] (vA) at (3,-2){};
    \node[vertex,paren,label={above:$v_0$}] (v0) at (2,0){};
    \node[vertex,label={above:$v_1$}] (v1) at (3,0){};
    \node[vertex,label={above:$v_2$}] (v2) at (4,0){};
    \node[vertex,label={above:$v_3$}] (v3) at (5,0){};
    
    \node[vertex] (w1) at (3.5,-1){};
    \node[vertex] (w2) at (4.5,-1){};
    
    
    \draw (v1) -- (w1) -- (w2) -- (v3);
    \draw (u1) -- (u2) -- (u3) -- (u4);
    \draw[thickedge] (u1) -- (vA) -- (vj1);
    \draw[bend right] (u3) to (vj1);
    \draw[bend right] (u4) to (vj1);
    \draw (u1) -- (v0) -- (v1) -- (v2) -- (v3);
    \draw (v3) to node[midway,sloped, fill=white] {$\ldots$} (vj) -- (vj1);

   \coordinate (BB1) at (current bounding box.north east);
   \coordinate (BB2) at (current bounding box.south west);
\end{tikzpicture}

\begin{tikzpicture}[scale=1]
    \useasboundingbox ($ (BB1) - (1,0) $) rectangle (BB2);

    \node[precoloured,label={above:$u_1$}] (u1) at (1,0){};
    \begin{scope}[xshift=1cm]
    \begin{scope}[shift={(135:-2cm)},rotate=90, xscale=-1]
    \node[precoloured] (u2) at (105:2cm){};
    \node[precoloured] (u3) at (75:2cm){};
    \node[precoloured] (u4) at (45:2cm){};

    \node[vertex,label={below left:$v\in \hat B$}] (vB) at (-0.75,0) {};
    \node[vertex,label={right:$u\in \hat A$}] (vA) at (0.75,0) {};
    \draw (vB) -- (u1) -- (u2) -- (u3) -- (u4) -- (vA);
    \draw (u2) -- (vB) -- (vA);
    \draw[thickedge] (u1) -- (vB) -- (vA) -- (u4);
    \end{scope}
    \end{scope}

    
    \node[vertex,paren,label={above:$v_0$}] (v0) at (2,0){};
    \node[vertex,label={above:$v_1$}] (v1) at (3,0){};
    \node[vertex,label={above:$v_2$}] (v2) at (4,0){};
    \node[vertex,label={above:$v_3$}] (v3) at (5,0){};
    \node (vd) at (6,0) {$\ldots$};

    \node[vertex] (w1) at (3.5,-1){};
    \node[vertex] (w2) at (4.5,-1){};
    \draw (v1) -- (w1) -- (w2) -- (v3);
    
    \draw (u1) -- (v0) -- (v1) -- (v2) -- (v3) -- (vd);
\end{tikzpicture}}%
    \hspace*{0.05\textwidth}
    \subcaptionbox{Since $g(u)\ge 5$, the emphasized 2-chord violates Lemma~\ref{lem:2chords}.\label{fig:f2233j1}}[0.45\textwidth]{\begin{tikzpicture}[scale=1]
    
    \foreach \i in {1,2,3}{
    	\pgfmathtruncatemacro{\j}{\i+2}
    	\node[vertex,label={above:$v_\i$}] (v\i) at (\j,0){};
    }
    \node (vd1) at (6,0) {$\ldots$};
	\node[vertex,label={above:$v_i\in R$}] (vi) at (7,0){};
    \node (vd0) at (2,0) {$\ldots$};
    \node (vdi) at (8,0) {$\ldots$};
    \draw (vd0) -- (v1);
    \draw (vdi) -- (vi);
	
    \node[vertex,label={below:$w_1$}] (w1) at (3.5,-1){};
    \node[vertex,label={below:$w_2$}] (w2) at (4.5,-1){};
    \draw (w1) -- (w2) -- (v3);
    
    \foreach \i in {1,2} {
    	\pgfmathtruncatemacro{\j}{\i+1}
    	\draw (v\i) -- (v\j);
	}
	\draw (v3) -- (vd1);
	
	\node[vertex,label={below:$w_3$}] (w3) at (4.5, -2) {};
	\node[vertex,label={below:$w_4$}] (w4) at (6.5, -2) {};
	
	\draw[thickedge] (v1) --  (w1) -- (w3) -- (w4) -- (vi);
    \draw (vi) -- (vd1);
  
\end{tikzpicture}

\begin{tikzpicture}[scale=1]

    \foreach \i in {1,2,3}{
    	\pgfmathtruncatemacro{\j}{\i+2}
    	\node[vertex,label={above:$v_\i$}] (v\i) at (\j,0){};
    }
    \node (vd1) at (6,0) {$\ldots$};
	\node[vertex,label={above:$v_i\in R$}] (vi) at (7,0){};
    \node (vd0) at (2,0) {$\ldots$};
    \node (vdi) at (8,0) {$\ldots$};
    \draw (vd0) -- (v1);
    \draw (vdi) -- (vi);
    
    \node[vertex,label={below:$w_1$}] (w1) at (3.5,-1){};
    \node[vertex,label={below:$w_2$}] (w2) at (4.5,-1){};
    \draw (w2) -- (v3);
    
    \foreach \i in {1,2} {
    	\pgfmathtruncatemacro{\j}{\i+1}
    	\draw (v\i) -- (v\j);
	}
	\draw (v3) -- (vd1);
	
	\node[vertex,label={below:$w_3$}] (w3) at (5.5, -2) {};
	\node[vertex,label={below:$w_4$}] (w4) at (6.5, -2) {};
	
	\draw[thickedge] (v1) -- (w1) --  (w2) -- (w3) -- (w4) -- (vi);
    \draw (vi) -- (vd1);
  
\end{tikzpicture}

\begin{tikzpicture}[scale=1]

    \foreach \i in {1,2,3}{
    	\pgfmathtruncatemacro{\j}{\i+2}
    	\node[vertex,label={above:$v_\i$}] (v\i) at (\j,0){};
    }
    \node (vd1) at (2,0) {$\ldots$};
    \node (vd3) at (6,0) {$\ldots$};
    \draw (vd1) -- (v1);
    \draw (vd3) -- (v3);
    
    \node[vertex,label={below:$w_1$}] (w1) at (3.5,-1){};
    \node[vertex,label={below:$w_2$}] (w2) at (4.5,-1){};
    \draw (v1) -- (w1) -- (w2) -- (v3);
    
    \foreach \i in {1,2} {
    	\pgfmathtruncatemacro{\j}{\i+1}
    	\draw (v\i) -- (v\j);
	}
	
	\node[vertex,label={right:$w_3$}] (w3) at (5, -2) {};
	\node[vertex,label={right:$w_4$}] (w4) at (5.5, -1) {};
	
	\draw (w1) -- (w3) -- (w4) -- (v3);
	
	\node[elabel] (e1) at (4,-0.5) {$\emptyset$};
	\node[elabel] (e2) at (5,-1) {$\emptyset$};

\end{tikzpicture}

\begin{tikzpicture}[scale=1]

    \foreach \i in {1,2,3}{
    	\pgfmathtruncatemacro{\j}{\i+2}
    	\node[vertex,label={above:$v_\i$}] (v\i) at (\j,0){};
    }
    \node (vd1) at (2,0) {$\ldots$};
    \node (vd3) at (6,0) {$\ldots$};
    \draw (vd1) -- (v1);
    \draw (vd3) -- (v3);
	
    \node[vertex,label={below:$w_1$}] (w1) at (3.5,-1){};
    \node[vertex,label={below:$w_2$}] (w2) at (4.5,-1){};
    \draw (v1) -- (w1) -- (w2) -- (v3);
    
    \foreach \i in {1,2} {
    	\pgfmathtruncatemacro{\j}{\i+1}
    	\draw (v\i) -- (v\j);
	}
	
	\node[vertex,label={left:$w_3$}] (w3) at (2, -2) {};
	\node[vertex,label={left:$w_4$}] (w4) at (1.5, -1) {};
	
	\draw (w1) -- (w3) -- (w4) -- (v1);

\end{tikzpicture}

\begin{tikzpicture}[scale=1]

    \foreach \i in {1,2,3}{
    	\pgfmathtruncatemacro{\j}{\i+2}
    	\node[vertex,label={above:$v_\i$}] (v\i) at (\j,0){};
    }
	\node[vertex,label={above:$v_4$}] (v4) at (6.5,0){};
    \node (vd1) at (2,0) {$\ldots$};
    \node (vd4) at (7.5,0) {$\ldots$};
    \draw (vd1) -- (v1);
    \draw (vd4) -- (v4);
    
    \node[vertex,label={below:$w_1$}] (w1) at (3.5,-1){};
    \node[vertex,label={below:$w_2$}] (w2) at (4.5,-1){};
    \draw (v1) -- (w1) -- (w2) -- (v3);
    
    \foreach \i in {1,2,3} {
    	\pgfmathtruncatemacro{\j}{\i+1}
    	\draw (v\i) -- (v\j);
	}
	
	\node[vertex,label={below:$w_3$}] (w3) at (5.5, -2) {};
	\node[vertex,label={below:$w_4$}] (w4) at (6.5, -2) {};
	
	\draw (w2) -- (w3) -- (w4) -- (v4);
	

    \draw[dashed] (2.5,0.5) -- (2.5,0) arc (180:270:0.5) --(4,-0.5) arc (-90:0:0.5) -- (4.5,0.5);
\end{tikzpicture}

\begin{tikzpicture}[scale=1]
    \node[precoloured,label={above:$u_1$}] (u1) at (1,0){};
    \begin{scope}[xshift=1cm]
    \begin{scope}[shift={(135:-2cm)},rotate=90, xscale=-1]
    \node[precoloured] (u2) at (105:2cm){};
    \node[precoloured] (u3) at (75:2cm){};
    \end{scope}
    \end{scope}
    \node[precoloured] (u4) at ($ (u3) + (0,-1) $){};

    \begin{scope}[xshift=5cm, xscale=-1]
    \begin{scope}[shift={(135:-2cm)},rotate=90, xscale=-1]  
    \node[vertex,label={right:$v_j$}] (vj) at (105:2cm) {};
    \node[vertex,label={right:$v_{j+1}\in \hat B$}] (vj1) at (75:2cm){};
    \end{scope}
    \end{scope}
    \node[rotate=90] (vd2) at ($ (vj1) + (0,-1) $) {$\ldots$};
    \draw (vj1) -- (vd2);
    
    \node[vertex,label={left:$u\in \hat A$}] (vA) at (3,-2){};
    \node[vertex,paren,label={above:$v_0$}] (v0) at (2,0){};
    \node[vertex,label={above:$v_1$}] (v1) at (3,0){};
    \node[vertex,label={above:$v_2$}] (v2) at (4,0){};
    \node[vertex,label={above:$v_3$}] (v3) at (5,0){};
    
    \node[vertex] (w1) at (3.5,-1){};
    \node[vertex] (w2) at (4.5,-1){};
    
    
    \draw (v1) -- (w1) -- (w2) -- (v3);
    \draw (u1) -- (u2) -- (u3) -- (u4);
    \draw[thickedge] (u1) -- (vA) -- (vj1);
    \draw[bend right] (u3) to (vj1);
    \draw[bend right] (u4) to (vj1);
    \draw (u1) -- (v0) -- (v1) -- (v2) -- (v3);
    \draw (v3) to node[midway,sloped, fill=white] {$\ldots$} (vj) -- (vj1);

   \coordinate (BB1) at (current bounding box.north east);
   \coordinate (BB2) at (current bounding box.south west);
\end{tikzpicture}

\begin{tikzpicture}[scale=1]
    \useasboundingbox ($ (BB1) - (1,0) $) rectangle (BB2);

    \node[precoloured,label={above:$u_1$}] (u1) at (1,0){};
    \begin{scope}[xshift=1cm]
    \begin{scope}[shift={(135:-2cm)},rotate=90, xscale=-1]
    \node[precoloured] (u2) at (105:2cm){};
    \node[precoloured] (u3) at (75:2cm){};
    \node[precoloured] (u4) at (45:2cm){};

    \node[vertex,label={below left:$v\in \hat B$}] (vB) at (-0.75,0) {};
    \node[vertex,label={right:$u\in \hat A$}] (vA) at (0.75,0) {};
    \draw (vB) -- (u1) -- (u2) -- (u3) -- (u4) -- (vA);
    \draw (u2) -- (vB) -- (vA);
    \draw[thickedge] (u1) -- (vB) -- (vA) -- (u4);
    \end{scope}
    \end{scope}

    
    \node[vertex,paren,label={above:$v_0$}] (v0) at (2,0){};
    \node[vertex,label={above:$v_1$}] (v1) at (3,0){};
    \node[vertex,label={above:$v_2$}] (v2) at (4,0){};
    \node[vertex,label={above:$v_3$}] (v3) at (5,0){};
    \node (vd) at (6,0) {$\ldots$};

    \node[vertex] (w1) at (3.5,-1){};
    \node[vertex] (w2) at (4.5,-1){};
    \draw (v1) -- (w1) -- (w2) -- (v3);
    
    \draw (u1) -- (v0) -- (v1) -- (v2) -- (v3) -- (vd);
\end{tikzpicture}}%
    \captionsetup{width=0.9\textwidth}
    \caption{Showing that $\hat K$ cannot be an exceptional canvas of type~\ref{ex:AB} in Claim~\ref{claim:ifcanvthenunex}.\label{fig:AAnotexAB}}
\end{figure}

If $u \in A$ and $v \not \in B$, then given $g(u) \geq 5$, we have that $uv$ together with the endpoint of $P$ neighbouring $v$ form a 2-chord violating Lemma \ref{lem:2chords}. If $u \not \in A$ and $v \in B$, then $uv$ together with the endpoint of $P$ neighbouring $u$ form a 2-chord violating Lemma \ref{lem:2chords}: either way, a contradiction.
\end{cproof}

Next we argue that $\hat{K}$ is not a canvas. In particular, we argue that $\hat{K}$ fails to be a canvas precisely because $\hat{A}$ is not an independent set. We will show $\hat{K}$ satisfies all \emph{other} properties of canvases:  in particular, that $\hat{f}$ is positive for all $v \in V(\hat{G}) \setminus V(P)$. Since $w_{3-\ell}$ is removed via $\del$, showing that $\hat{f}(v) \geq 0$ for all $v \in V(\hat{G})\setminus V(P)$ implies that $w_{3-\ell}$ is removed via a legal operation.

\begin{claim}\label{claim:finalstructure1}
    $\hat{K}$ satisfies every canvas property with the exception of $\ref{canv:A}$ (and hence $w_{3-\ell}$ is removed via a legal application of $\del$). In particular, $\hat{A}$ is a set of vertices of girth at least five but $\hat{G}[\hat{A}]$ contains at least one edge.   
\end{claim}

\begin{cproof}
We first show that $\hat{K}$ satisfies every canvas property except \ref{canv:A}.

Note that Properties \ref{canv:acceptable} and \ref{canv:fint} hold, since $K$ is a canvas.  Recall that $\hat{A}$ is defined as $\{v \in V(\hat{G}) \setminus V(P) : \hat{f}(v) \leq 1\}$, and  $\hat{B}$ is defined as $\{v \in V(\hat{G}) \setminus V(P): \hat{f}(v) = 2 \textnormal{ and } g(v) = 3\}$. Hence Property \ref{canv:fB} holds by definition of $\hat{B}$, and Property \ref{canv:fCn3} holds by definition of $\hat{A}$. Property \ref{canv:fC3} holds by definition of $\hat{A}$ and $\hat{B}$. Thus it remains to show that Properties \ref{canv:B} and \ref{canv:fA} hold. We begin with Property \ref{canv:B}.

\begin{subclaim}
    Property \ref{canv:B} holds for $\hat K$.
\end{subclaim}
\begin{scproof}[Proof of subclaim]
Note that every vertex in $\hat{B}$ has girth three by definition, and hence it suffices to show $\hat{B}$ is independent. To that end, suppose not: then there exist vertices $u,u' \in \hat{B}$ with $uu' \in E(\hat{G})$. 
Since $B$ is independent, without loss of generality either $u \in \hat{B}\setminus B$ and $u' \in B$, or $\{u,u'\} \subseteq \hat{B} \setminus B$. 

First suppose $u \in \hat{B}\setminus B$ and $u' \in B$. If $u \neq v_{j+1}$, then by Claim \ref{claim:delGfine} $u \not \in V(\delta G)$ and hence $f(u) = 4$. Since $u \in \hat{B}$, we have that $u$ is adjacent in $G$ to at least two vertices in $V(R) \cup \{w_1,w_2\}$. By Claim \ref{claim:notadjtow1w2andR} and the fact that $g(w_1) = g(w_2) = 5$, $u$ is adjacent to two vertices in $R$, and so there exists an index $i \leq j-1 $ such that $u$ is adjacent to $v_i$. But then $v_iuu'$ is a 2-chord which, given $u' \in B$ and $B$ is independent, violates Lemma \ref{lem:2chords}. Thus we may assume $u = v_{j+1}$. Since $u \not \in B$, by Claim \ref{claim:delGfine} $R$ is defined according to \ref{case:f2233}. Since $\delta G$ has no 1-chord by Lemma \ref{lem:1chord}, $u' = v_{j+2}$. This contradicts the definition of $R$.

Thus we may assume both $u$ and $u'$ are in $\hat{B}\setminus B$.
If $u$ is not $v_{j+1}$, then by Claim \ref{claim:delGfine} we have that $f(u) = 4$. Since $u$ is in $\hat{B}\setminus \{v_{j+1}\}$, it is adjacent to two vertices in $V(R) \cup \{w_1, w_2\}$. As above, the girth of $w_1$ and $w_2$ Claim~\ref{claim:notadjtow1w2andR} mean that in fact $u$ is adjacent to two vertices in $V(R)$. 
Since $R$ is removed via the same sequence of operations as in \ref{handle:f2233} or \ref{handle:f223}, we have that $u$ is also in $\tilde B$ (where we use the notation from Section~\ref{sec:properties}: $(G-R,P, \tilde{A}, \tilde{B}, \tilde{f})$ is obtained from $K$ by removing $R$ via \ref{handle:f2233} or \ref{handle:f223}).  If $u'$ is \emph{also} not $v_{j+1}$, we have that $u'$ is also in $\tilde B$ and so it follows that $\tilde{G}[\tilde{B}]$ is not independent and the structure established in Lemma \ref{lem:BBstructure} exists. But since $\tilde{G}[\tilde{A}]$ is also not independent, Lemma \ref{lem:AAstructure} also applies. This is a contradiction, as Lemmas \ref{lem:BBstructure} and \ref{lem:AAstructure} are mutually exclusive.

Thus, without loss of generality $u \in \hat{B}\setminus B$ and $u' = v_{j+1}$. Let $i$ be the smallest index such that $u$ is adjacent to $v_i$ and $v_i \in V(R)$. 
As above, $u$ is adjacent to two vertices in $R$ which means that $i \leq j-1$. Since $R$ is defined according to case \ref{case:f2233} or \ref{case:f223}, $f(v_{i}) \leq 2$. But then $v_iuu'$ is a 2-chord violating Lemma \ref{lem:2chords}, a contradiction.
\end{scproof}

Next, we show Property \ref{canv:fA} holds for $\hat K$, and that every vertex in $\hat{A}$ has girth at least five.

\begin{subclaim}
    Property \ref{canv:fA} holds, and $g(v) \geq 5$ for all $v \in \hat{A}$.
\end{subclaim}
\begin{scproof}[Proof of subclaim]
Recall that $\hat{A}$ is the set $\{v \in V(\hat{G}) \setminus V(P) : \hat{f}(v)\leq 1\}$. First suppose that there exists a vertex $v \in \hat{A}$ with $\hat{f}(v) \leq 0$. By Claim \ref{claim:delGfine}, we have that $v \not \in V(\delta G)$, and hence either $f(v) = 4$ and $g(v)=3$, or $f(v) = 3$ and $g(v) = 4$, or $f(v) = 2$ and $g(v) \geq 5$.  Since $\hat{f}(v) \leq 0$, if $v$ has girth three, four, or at least five, then $v$ is adjacent to at (at least) four, three, or two vertices, respectively, in $V(R) \cup \{w_1,w_2\}$. Since $g(w_1) =5$, $v$ is not adjacent to both $w_1$ and $w_2$, and hence by Claim \ref{claim:notadjtow1w2andR}, $v$ is adjacent to at least four, three, or two vertices in $V(R)$ (depending on its girth). If $g(v) \geq 4$, this immediately gives a 2-chord violating Lemma \ref{lem:2chords}, and so we may assume $g(v) = 3$ and that $v$ is adjacent to four vertices in $R$. Let $i$ and $i'$ be the lowest and highest indices, respectively, with $vv_i,vv_{i'} \in E(G)$ and $v_i,v_{i'} \in V(R)$.  Since $i' \geq i+3$, we have that $v_ivv_{i'}$ is a 2-chord violating Lemma \ref{lem:2chords}. Thus Property \ref{canv:fA} holds, as desired.

Suppose next that $\hat{A}$ contains a vertex $v$ of girth at most four. Again by Claim \ref{claim:delGfine}, we have that $v \not \in V(\delta G)$, and hence either $f(v) = 4$ and $g(v)=3$, or $f(v) = 3$ and $g(v) = 4$. Since  $\hat{f}(v) = 1$, if $v$ has girth three or four, then $v$ is adjacent to three or two vertices, respectively, in $V(R) \cup \{w_1,w_2\}$. Again since $g(w_1) = 5$, $v$ is not adjacent to both $w_1$ and $w_2$, and hence by Claim \ref{claim:notadjtow1w2andR} $v$ is adjacent to three or two (depending on its girth) vertices in $V(R)$. 
As above, if $g(v) = 4$, this creates a 2-chord violating Lemma \ref{lem:2chords}, and so we may assume $g(v) = 3$. Let $v_i$ and $v_{i'}$ be the vertices of lowest and highest index in $V(R)$ adjacent to $v$ respectively. 
Since $v$ is adjacent to three vertices in $R$, we have that $i' \geq i+2$.  Hence $v_ivv_{i'}$ is a 2-chord violating Lemma \ref{lem:2chords} unless $i' = i+2$, $f(v_i) = 3$, $f(v_{i+1}) = 2$, and $f(v_{i+2}) = 3$. But $f(v_i)\ne 3$ by Observation \ref{obs:whereisthelist4vert}.
\end{scproof}
To complete the proof of Claim~\ref{claim:finalstructure1}, we must verify that $\hat{G}[\hat{A}]$ contains at least one edge.
First note that if $\hat{A}$ is independent, then $\hat{K}$ is a canvas. By Claim \ref{claim:ifcanvthenunex}, it is moreover unexceptional and so by the minimality of $K$, we have that $\hat{K}$ is weakly $\hat{f}$-degenerate and so there exists a legal sequence of operations $\sigma'$ that removes all vertices of $\hat{G}$. 
But by performing $\sigma$ and then $\sigma'$ starting from $(G-P, f_K)$, we conclude that $G-P$ is weakly $f_K$-degenerate. But then $K$ is weakly $f$-degenerate, a contradiction. 
Thus $\hat{G}[\hat{A}]$ contains at least one edge.
\end{cproof}

\begin{claim}\label{claim:finalstructure2}
    For every edge $w_3w_4$ in $\hat{G}[\hat{A}]$, up to swapping the labels of $w_3$ and $w_4$, exactly one of the following holds:
    \begin{enumerate}[label=\textup{(\roman*)}]
        \item\label{itm:w3w1} $w_3$ is adjacent to $w_1$ and $w_4$ is adjacent to a vertex $v_i\in V(R)$ with $i \geq 4$, or
        \item\label{itm:w3w2} $w_3$ is adjacent to $w_2$ and $w_4$ is adjacent to a vertex $v_i\in V(R)$ with $i \geq 4$.
    \end{enumerate}
\end{claim}
\begin{cproof}
    Note that by Claim \ref{claim:delGfine}, if $R$ is defined according to \ref{case:f2233} then $\hat{f}(v) = f(v)$ for all $v \in V(\delta G) \cap V(\hat{G})$ with $g(v) \geq 5$. 
    If $R$ is defined according to \ref{case:f223}, then $g(v_{j+1}) = 3$ and hence again by Claim \ref{claim:delGfine}, $\hat{f}(v) = f(v)$ for all $v \in V(\delta G) \cap V(\hat{G})$ with $g(v) \geq 5$.  Thus $(\hat{A}\setminus A) \cap V(\delta G) = \emptyset$.

    Let $w_3w_4$ be any edge in $\hat{G}[\hat{A}]$. We first show that both $w_3$ and $w_4$ are in $\hat A\setminus A$. 
    Suppose not; without loss of generality suppose that $w_3 \in A$. Recall that every vertex in $\hat{A} \setminus A$ is adjacent in $G$ to a vertex in $V(R) \cup \{w_1,w_2\}$.
\begin{itemize}
    \item If $w_4 \in A$, then $w_3w_4$ is a 1-chord violating Lemma \ref{lem:1chord}.
    \item If $w_4\in \hat{A} \setminus A$ and $w_4$ neighbours a vertex $w \in V(R)$, then $w_3w_4w$ is a 2-chord violating Lemma~\ref{lem:2chords}.
    \item Finally, suppose that $w_4 \in \hat{A} \setminus A$ and that $w_4$ neighbours a vertex $w \in \{w_1,w_2\}$. Let $v'$ be the neighbour of $w$ in $\{v_1,v_3\}$. Note that $w_3w_4wv'$ is a 3-chord. By Lemma \ref{lem:3chords}, since $g(w_4)=g(w) = 5$, we have that $w_3$ (and $v'$) neighbour a vertex in $A$. This contradicts that $w_3$ is in $A$, since $A$ is independent.


\end{itemize}
Hence, $w_3$ and $w_4$ are both in $\hat A\setminus A$, which as discussed above, is a subset of $V(G)\setminus V(\delta G)$. Let $w_3'$ and $w_4'$ be neighbours of $w_3$ and $w_4$, respectively, in $V(R) \cup \{w_1,w_2\}$. 
Note that since $g(w_1) = 5$, at least one of $w_3'$ and $w_4'$ is not in $\{w_1,w_2\}$ (and is thus in $R$). 
If \emph{both} $w_3'$ and $w_4'$ are in $R$, then since $R$ was removed via the same sequence of operations as \ref{handle:f2233} or \ref{handle:f223}, we have that $w_3w_4$ is also in $\tilde{G}[\tilde{A}]$. This contradicts Lemma \ref{lem:AAstructure} \ref{itm:AAunique}, since $w_3w_4 \neq w_1w_2$. Hence exactly one of $w_3'$ and $w_4'$ is in $V(R)$. Swapping the labels of $w_3$ and $w_4$ if necessary, suppose that $w_4'$ is in $R$.
Then $w_4' = v_i$ for some index $i$. We now argue that $i \geq 4$. 

Note that $i \neq  2$ since $G$ is a plane graph, $g(w_1)=5$, and so the interior of the 5-cycle $v_1v_2v_3w_2w_1v_1$ is empty by Lemma~\ref{lem:5cycles}. 
Moreover, $i \neq 3$: otherwise, in case \ref{itm:w3w1}, we have that the interior of the 5-cycle $w_1w_3w_4v_3w_2w_1$ is empty by Lemma \ref{lem:5cycles} (and the fact that $g(w_1) = 5$) and hence $\deg(w_2) = 2 = f(w_2)$, contradicting Lemma \ref{lem:nolowdegs}. In case \ref{itm:w3w2}, we have that $v_3w_2w_3w_4v_3$ is a 4-cycle, contradicting that $g(w_2) = 5$. Symmetrically, $i \neq 1$, and hence $i \geq 4$. See Figure~\ref{fig:finalstructure2}.
\end{cproof}

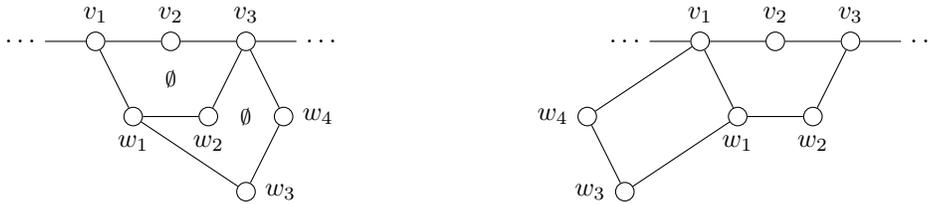
\begin{figure}[ht]\centering
    \subcaptionbox{Showing that $i\ne 3$ in the proof of Claim~\ref{claim:finalstructure2}\ref{itm:w3w1}: $\deg(w_2)=2$ contradicts Lemma~\ref{lem:nolowdegs}.}[0.4\textwidth]{\begin{tikzpicture}[scale=1]
    
    \foreach \i in {1,2,3}{
    	\pgfmathtruncatemacro{\j}{\i+2}
    	\node[vertex,label={above:$v_\i$}] (v\i) at (\j,0){};
    }
    \node (vd1) at (6,0) {$\ldots$};
	\node[vertex,label={above:$v_i\in R$}] (vi) at (7,0){};
    \node (vd0) at (2,0) {$\ldots$};
    \node (vdi) at (8,0) {$\ldots$};
    \draw (vd0) -- (v1);
    \draw (vdi) -- (vi);
	
    \node[vertex,label={below:$w_1$}] (w1) at (3.5,-1){};
    \node[vertex,label={below:$w_2$}] (w2) at (4.5,-1){};
    \draw (w1) -- (w2) -- (v3);
    
    \foreach \i in {1,2} {
    	\pgfmathtruncatemacro{\j}{\i+1}
    	\draw (v\i) -- (v\j);
	}
	\draw (v3) -- (vd1);
	
	\node[vertex,label={below:$w_3$}] (w3) at (4.5, -2) {};
	\node[vertex,label={below:$w_4$}] (w4) at (6.5, -2) {};
	
	\draw[thickedge] (v1) --  (w1) -- (w3) -- (w4) -- (vi);
    \draw (vi) -- (vd1);
  
\end{tikzpicture}

\begin{tikzpicture}[scale=1]

    \foreach \i in {1,2,3}{
    	\pgfmathtruncatemacro{\j}{\i+2}
    	\node[vertex,label={above:$v_\i$}] (v\i) at (\j,0){};
    }
    \node (vd1) at (6,0) {$\ldots$};
	\node[vertex,label={above:$v_i\in R$}] (vi) at (7,0){};
    \node (vd0) at (2,0) {$\ldots$};
    \node (vdi) at (8,0) {$\ldots$};
    \draw (vd0) -- (v1);
    \draw (vdi) -- (vi);
    
    \node[vertex,label={below:$w_1$}] (w1) at (3.5,-1){};
    \node[vertex,label={below:$w_2$}] (w2) at (4.5,-1){};
    \draw (w2) -- (v3);
    
    \foreach \i in {1,2} {
    	\pgfmathtruncatemacro{\j}{\i+1}
    	\draw (v\i) -- (v\j);
	}
	\draw (v3) -- (vd1);
	
	\node[vertex,label={below:$w_3$}] (w3) at (5.5, -2) {};
	\node[vertex,label={below:$w_4$}] (w4) at (6.5, -2) {};
	
	\draw[thickedge] (v1) -- (w1) --  (w2) -- (w3) -- (w4) -- (vi);
    \draw (vi) -- (vd1);
  
\end{tikzpicture}

\begin{tikzpicture}[scale=1]

    \foreach \i in {1,2,3}{
    	\pgfmathtruncatemacro{\j}{\i+2}
    	\node[vertex,label={above:$v_\i$}] (v\i) at (\j,0){};
    }
    \node (vd1) at (2,0) {$\ldots$};
    \node (vd3) at (6,0) {$\ldots$};
    \draw (vd1) -- (v1);
    \draw (vd3) -- (v3);
    
    \node[vertex,label={below:$w_1$}] (w1) at (3.5,-1){};
    \node[vertex,label={below:$w_2$}] (w2) at (4.5,-1){};
    \draw (v1) -- (w1) -- (w2) -- (v3);
    
    \foreach \i in {1,2} {
    	\pgfmathtruncatemacro{\j}{\i+1}
    	\draw (v\i) -- (v\j);
	}
	
	\node[vertex,label={right:$w_3$}] (w3) at (5, -2) {};
	\node[vertex,label={right:$w_4$}] (w4) at (5.5, -1) {};
	
	\draw (w1) -- (w3) -- (w4) -- (v3);
	
	\node[elabel] (e1) at (4,-0.5) {$\emptyset$};
	\node[elabel] (e2) at (5,-1) {$\emptyset$};

\end{tikzpicture}

\begin{tikzpicture}[scale=1]

    \foreach \i in {1,2,3}{
    	\pgfmathtruncatemacro{\j}{\i+2}
    	\node[vertex,label={above:$v_\i$}] (v\i) at (\j,0){};
    }
    \node (vd1) at (2,0) {$\ldots$};
    \node (vd3) at (6,0) {$\ldots$};
    \draw (vd1) -- (v1);
    \draw (vd3) -- (v3);
	
    \node[vertex,label={below:$w_1$}] (w1) at (3.5,-1){};
    \node[vertex,label={below:$w_2$}] (w2) at (4.5,-1){};
    \draw (v1) -- (w1) -- (w2) -- (v3);
    
    \foreach \i in {1,2} {
    	\pgfmathtruncatemacro{\j}{\i+1}
    	\draw (v\i) -- (v\j);
	}
	
	\node[vertex,label={left:$w_3$}] (w3) at (2, -2) {};
	\node[vertex,label={left:$w_4$}] (w4) at (1.5, -1) {};
	
	\draw (w1) -- (w3) -- (w4) -- (v1);

\end{tikzpicture}

\begin{tikzpicture}[scale=1]

    \foreach \i in {1,2,3}{
    	\pgfmathtruncatemacro{\j}{\i+2}
    	\node[vertex,label={above:$v_\i$}] (v\i) at (\j,0){};
    }
	\node[vertex,label={above:$v_4$}] (v4) at (6.5,0){};
    \node (vd1) at (2,0) {$\ldots$};
    \node (vd4) at (7.5,0) {$\ldots$};
    \draw (vd1) -- (v1);
    \draw (vd4) -- (v4);
    
    \node[vertex,label={below:$w_1$}] (w1) at (3.5,-1){};
    \node[vertex,label={below:$w_2$}] (w2) at (4.5,-1){};
    \draw (v1) -- (w1) -- (w2) -- (v3);
    
    \foreach \i in {1,2,3} {
    	\pgfmathtruncatemacro{\j}{\i+1}
    	\draw (v\i) -- (v\j);
	}
	
	\node[vertex,label={below:$w_3$}] (w3) at (5.5, -2) {};
	\node[vertex,label={below:$w_4$}] (w4) at (6.5, -2) {};
	
	\draw (w2) -- (w3) -- (w4) -- (v4);
	

    \draw[dashed] (2.5,0.5) -- (2.5,0) arc (180:270:0.5) --(4,-0.5) arc (-90:0:0.5) -- (4.5,0.5);
\end{tikzpicture}

\begin{tikzpicture}[scale=1]
    \node[precoloured,label={above:$u_1$}] (u1) at (1,0){};
    \begin{scope}[xshift=1cm]
    \begin{scope}[shift={(135:-2cm)},rotate=90, xscale=-1]
    \node[precoloured] (u2) at (105:2cm){};
    \node[precoloured] (u3) at (75:2cm){};
    \end{scope}
    \end{scope}
    \node[precoloured] (u4) at ($ (u3) + (0,-1) $){};

    \begin{scope}[xshift=5cm, xscale=-1]
    \begin{scope}[shift={(135:-2cm)},rotate=90, xscale=-1]  
    \node[vertex,label={right:$v_j$}] (vj) at (105:2cm) {};
    \node[vertex,label={right:$v_{j+1}\in \hat B$}] (vj1) at (75:2cm){};
    \end{scope}
    \end{scope}
    \node[rotate=90] (vd2) at ($ (vj1) + (0,-1) $) {$\ldots$};
    \draw (vj1) -- (vd2);
    
    \node[vertex,label={left:$u\in \hat A$}] (vA) at (3,-2){};
    \node[vertex,paren,label={above:$v_0$}] (v0) at (2,0){};
    \node[vertex,label={above:$v_1$}] (v1) at (3,0){};
    \node[vertex,label={above:$v_2$}] (v2) at (4,0){};
    \node[vertex,label={above:$v_3$}] (v3) at (5,0){};
    
    \node[vertex] (w1) at (3.5,-1){};
    \node[vertex] (w2) at (4.5,-1){};
    
    
    \draw (v1) -- (w1) -- (w2) -- (v3);
    \draw (u1) -- (u2) -- (u3) -- (u4);
    \draw[thickedge] (u1) -- (vA) -- (vj1);
    \draw[bend right] (u3) to (vj1);
    \draw[bend right] (u4) to (vj1);
    \draw (u1) -- (v0) -- (v1) -- (v2) -- (v3);
    \draw (v3) to node[midway,sloped, fill=white] {$\ldots$} (vj) -- (vj1);

   \coordinate (BB1) at (current bounding box.north east);
   \coordinate (BB2) at (current bounding box.south west);
\end{tikzpicture}

\begin{tikzpicture}[scale=1]
    \useasboundingbox ($ (BB1) - (1,0) $) rectangle (BB2);

    \node[precoloured,label={above:$u_1$}] (u1) at (1,0){};
    \begin{scope}[xshift=1cm]
    \begin{scope}[shift={(135:-2cm)},rotate=90, xscale=-1]
    \node[precoloured] (u2) at (105:2cm){};
    \node[precoloured] (u3) at (75:2cm){};
    \node[precoloured] (u4) at (45:2cm){};

    \node[vertex,label={below left:$v\in \hat B$}] (vB) at (-0.75,0) {};
    \node[vertex,label={right:$u\in \hat A$}] (vA) at (0.75,0) {};
    \draw (vB) -- (u1) -- (u2) -- (u3) -- (u4) -- (vA);
    \draw (u2) -- (vB) -- (vA);
    \draw[thickedge] (u1) -- (vB) -- (vA) -- (u4);
    \end{scope}
    \end{scope}

    
    \node[vertex,paren,label={above:$v_0$}] (v0) at (2,0){};
    \node[vertex,label={above:$v_1$}] (v1) at (3,0){};
    \node[vertex,label={above:$v_2$}] (v2) at (4,0){};
    \node[vertex,label={above:$v_3$}] (v3) at (5,0){};
    \node (vd) at (6,0) {$\ldots$};

    \node[vertex] (w1) at (3.5,-1){};
    \node[vertex] (w2) at (4.5,-1){};
    \draw (v1) -- (w1) -- (w2) -- (v3);
    
    \draw (u1) -- (v0) -- (v1) -- (v2) -- (v3) -- (vd);
\end{tikzpicture}}%
    \hspace*{0.05\textwidth}
    \subcaptionbox{Showing that $i\ne 1$ in the proof of Claim~\ref{claim:finalstructure2}\ref{itm:w3w1}. The 4-cycle contradicts that $g(w_1) = 5$ (Lemma \ref{lem:AAstructure}\ref{itm:AA3chord}).\label{fig:fs2ine1}}[0.4\textwidth]{\begin{tikzpicture}[scale=1]
    
    \foreach \i in {1,2,3}{
    	\pgfmathtruncatemacro{\j}{\i+2}
    	\node[vertex,label={above:$v_\i$}] (v\i) at (\j,0){};
    }
    \node (vd1) at (6,0) {$\ldots$};
	\node[vertex,label={above:$v_i\in R$}] (vi) at (7,0){};
    \node (vd0) at (2,0) {$\ldots$};
    \node (vdi) at (8,0) {$\ldots$};
    \draw (vd0) -- (v1);
    \draw (vdi) -- (vi);
	
    \node[vertex,label={below:$w_1$}] (w1) at (3.5,-1){};
    \node[vertex,label={below:$w_2$}] (w2) at (4.5,-1){};
    \draw (w1) -- (w2) -- (v3);
    
    \foreach \i in {1,2} {
    	\pgfmathtruncatemacro{\j}{\i+1}
    	\draw (v\i) -- (v\j);
	}
	\draw (v3) -- (vd1);
	
	\node[vertex,label={below:$w_3$}] (w3) at (4.5, -2) {};
	\node[vertex,label={below:$w_4$}] (w4) at (6.5, -2) {};
	
	\draw[thickedge] (v1) --  (w1) -- (w3) -- (w4) -- (vi);
    \draw (vi) -- (vd1);
  
\end{tikzpicture}

\begin{tikzpicture}[scale=1]

    \foreach \i in {1,2,3}{
    	\pgfmathtruncatemacro{\j}{\i+2}
    	\node[vertex,label={above:$v_\i$}] (v\i) at (\j,0){};
    }
    \node (vd1) at (6,0) {$\ldots$};
	\node[vertex,label={above:$v_i\in R$}] (vi) at (7,0){};
    \node (vd0) at (2,0) {$\ldots$};
    \node (vdi) at (8,0) {$\ldots$};
    \draw (vd0) -- (v1);
    \draw (vdi) -- (vi);
    
    \node[vertex,label={below:$w_1$}] (w1) at (3.5,-1){};
    \node[vertex,label={below:$w_2$}] (w2) at (4.5,-1){};
    \draw (w2) -- (v3);
    
    \foreach \i in {1,2} {
    	\pgfmathtruncatemacro{\j}{\i+1}
    	\draw (v\i) -- (v\j);
	}
	\draw (v3) -- (vd1);
	
	\node[vertex,label={below:$w_3$}] (w3) at (5.5, -2) {};
	\node[vertex,label={below:$w_4$}] (w4) at (6.5, -2) {};
	
	\draw[thickedge] (v1) -- (w1) --  (w2) -- (w3) -- (w4) -- (vi);
    \draw (vi) -- (vd1);
  
\end{tikzpicture}

\begin{tikzpicture}[scale=1]

    \foreach \i in {1,2,3}{
    	\pgfmathtruncatemacro{\j}{\i+2}
    	\node[vertex,label={above:$v_\i$}] (v\i) at (\j,0){};
    }
    \node (vd1) at (2,0) {$\ldots$};
    \node (vd3) at (6,0) {$\ldots$};
    \draw (vd1) -- (v1);
    \draw (vd3) -- (v3);
    
    \node[vertex,label={below:$w_1$}] (w1) at (3.5,-1){};
    \node[vertex,label={below:$w_2$}] (w2) at (4.5,-1){};
    \draw (v1) -- (w1) -- (w2) -- (v3);
    
    \foreach \i in {1,2} {
    	\pgfmathtruncatemacro{\j}{\i+1}
    	\draw (v\i) -- (v\j);
	}
	
	\node[vertex,label={right:$w_3$}] (w3) at (5, -2) {};
	\node[vertex,label={right:$w_4$}] (w4) at (5.5, -1) {};
	
	\draw (w1) -- (w3) -- (w4) -- (v3);
	
	\node[elabel] (e1) at (4,-0.5) {$\emptyset$};
	\node[elabel] (e2) at (5,-1) {$\emptyset$};

\end{tikzpicture}

\begin{tikzpicture}[scale=1]

    \foreach \i in {1,2,3}{
    	\pgfmathtruncatemacro{\j}{\i+2}
    	\node[vertex,label={above:$v_\i$}] (v\i) at (\j,0){};
    }
    \node (vd1) at (2,0) {$\ldots$};
    \node (vd3) at (6,0) {$\ldots$};
    \draw (vd1) -- (v1);
    \draw (vd3) -- (v3);
	
    \node[vertex,label={below:$w_1$}] (w1) at (3.5,-1){};
    \node[vertex,label={below:$w_2$}] (w2) at (4.5,-1){};
    \draw (v1) -- (w1) -- (w2) -- (v3);
    
    \foreach \i in {1,2} {
    	\pgfmathtruncatemacro{\j}{\i+1}
    	\draw (v\i) -- (v\j);
	}
	
	\node[vertex,label={left:$w_3$}] (w3) at (2, -2) {};
	\node[vertex,label={left:$w_4$}] (w4) at (1.5, -1) {};
	
	\draw (w1) -- (w3) -- (w4) -- (v1);

\end{tikzpicture}

\begin{tikzpicture}[scale=1]

    \foreach \i in {1,2,3}{
    	\pgfmathtruncatemacro{\j}{\i+2}
    	\node[vertex,label={above:$v_\i$}] (v\i) at (\j,0){};
    }
	\node[vertex,label={above:$v_4$}] (v4) at (6.5,0){};
    \node (vd1) at (2,0) {$\ldots$};
    \node (vd4) at (7.5,0) {$\ldots$};
    \draw (vd1) -- (v1);
    \draw (vd4) -- (v4);
    
    \node[vertex,label={below:$w_1$}] (w1) at (3.5,-1){};
    \node[vertex,label={below:$w_2$}] (w2) at (4.5,-1){};
    \draw (v1) -- (w1) -- (w2) -- (v3);
    
    \foreach \i in {1,2,3} {
    	\pgfmathtruncatemacro{\j}{\i+1}
    	\draw (v\i) -- (v\j);
	}
	
	\node[vertex,label={below:$w_3$}] (w3) at (5.5, -2) {};
	\node[vertex,label={below:$w_4$}] (w4) at (6.5, -2) {};
	
	\draw (w2) -- (w3) -- (w4) -- (v4);
	

    \draw[dashed] (2.5,0.5) -- (2.5,0) arc (180:270:0.5) --(4,-0.5) arc (-90:0:0.5) -- (4.5,0.5);
\end{tikzpicture}

\begin{tikzpicture}[scale=1]
    \node[precoloured,label={above:$u_1$}] (u1) at (1,0){};
    \begin{scope}[xshift=1cm]
    \begin{scope}[shift={(135:-2cm)},rotate=90, xscale=-1]
    \node[precoloured] (u2) at (105:2cm){};
    \node[precoloured] (u3) at (75:2cm){};
    \end{scope}
    \end{scope}
    \node[precoloured] (u4) at ($ (u3) + (0,-1) $){};

    \begin{scope}[xshift=5cm, xscale=-1]
    \begin{scope}[shift={(135:-2cm)},rotate=90, xscale=-1]  
    \node[vertex,label={right:$v_j$}] (vj) at (105:2cm) {};
    \node[vertex,label={right:$v_{j+1}\in \hat B$}] (vj1) at (75:2cm){};
    \end{scope}
    \end{scope}
    \node[rotate=90] (vd2) at ($ (vj1) + (0,-1) $) {$\ldots$};
    \draw (vj1) -- (vd2);
    
    \node[vertex,label={left:$u\in \hat A$}] (vA) at (3,-2){};
    \node[vertex,paren,label={above:$v_0$}] (v0) at (2,0){};
    \node[vertex,label={above:$v_1$}] (v1) at (3,0){};
    \node[vertex,label={above:$v_2$}] (v2) at (4,0){};
    \node[vertex,label={above:$v_3$}] (v3) at (5,0){};
    
    \node[vertex] (w1) at (3.5,-1){};
    \node[vertex] (w2) at (4.5,-1){};
    
    
    \draw (v1) -- (w1) -- (w2) -- (v3);
    \draw (u1) -- (u2) -- (u3) -- (u4);
    \draw[thickedge] (u1) -- (vA) -- (vj1);
    \draw[bend right] (u3) to (vj1);
    \draw[bend right] (u4) to (vj1);
    \draw (u1) -- (v0) -- (v1) -- (v2) -- (v3);
    \draw (v3) to node[midway,sloped, fill=white] {$\ldots$} (vj) -- (vj1);

   \coordinate (BB1) at (current bounding box.north east);
   \coordinate (BB2) at (current bounding box.south west);
\end{tikzpicture}

\begin{tikzpicture}[scale=1]
    \useasboundingbox ($ (BB1) - (1,0) $) rectangle (BB2);

    \node[precoloured,label={above:$u_1$}] (u1) at (1,0){};
    \begin{scope}[xshift=1cm]
    \begin{scope}[shift={(135:-2cm)},rotate=90, xscale=-1]
    \node[precoloured] (u2) at (105:2cm){};
    \node[precoloured] (u3) at (75:2cm){};
    \node[precoloured] (u4) at (45:2cm){};

    \node[vertex,label={below left:$v\in \hat B$}] (vB) at (-0.75,0) {};
    \node[vertex,label={right:$u\in \hat A$}] (vA) at (0.75,0) {};
    \draw (vB) -- (u1) -- (u2) -- (u3) -- (u4) -- (vA);
    \draw (u2) -- (vB) -- (vA);
    \draw[thickedge] (u1) -- (vB) -- (vA) -- (u4);
    \end{scope}
    \end{scope}

    
    \node[vertex,paren,label={above:$v_0$}] (v0) at (2,0){};
    \node[vertex,label={above:$v_1$}] (v1) at (3,0){};
    \node[vertex,label={above:$v_2$}] (v2) at (4,0){};
    \node[vertex,label={above:$v_3$}] (v3) at (5,0){};
    \node (vd) at (6,0) {$\ldots$};

    \node[vertex] (w1) at (3.5,-1){};
    \node[vertex] (w2) at (4.5,-1){};
    \draw (v1) -- (w1) -- (w2) -- (v3);
    
    \draw (u1) -- (v0) -- (v1) -- (v2) -- (v3) -- (vd);
\end{tikzpicture}}%
    \captionsetup{width=0.9\textwidth}
    \caption{Ruling out neighbours of $w_4$ in $R$ for Claim~\ref{claim:finalstructure2}. The structures for case~\ref{itm:w3w2} are similar.\label{fig:finalstructure2}}
\end{figure}

Let $w_3$ and $w_4$ be vertices as described in Claim~\ref{claim:finalstructure2}, and let $v_i$ the neighbour of $w_4$ in $V(R)$. In particular, choose $w_3w_4$ to minimize $i$. 
If $w_3$ is adjacent to $w_1$, let $Q = v_1w_1w_3w_4v_i$. If $w_3$ is adjacent to $w_2$, let $Q = v_1w_1w_2w_3w_4v_i$. Note that in both cases, the endpoints of $Q$ are $v_1$ and $v_i$. See Figure~\ref{fig:AAdetails}. Let $Q' = v_1v_2 \dots v_i$. The path $Q$ separates $G$ into two graphs $G_1, G_2$ (with $\delta G_2 = Q \cup Q'$). 

\begin{figure}[ht]\centering
    \subcaptionbox{Case~\ref{itm:w3w1}.\label{fig:fs2w3w1}}[0.45\textwidth]{\begin{tikzpicture}[scale=1]
    
    \foreach \i in {1,2,3}{
    	\pgfmathtruncatemacro{\j}{\i+2}
    	\node[vertex,label={above:$v_\i$}] (v\i) at (\j,0){};
    }
    \node (vd1) at (6,0) {$\ldots$};
	\node[vertex,label={above:$v_i\in R$}] (vi) at (7,0){};
    \node (vd0) at (2,0) {$\ldots$};
    \node (vdi) at (8,0) {$\ldots$};
    \draw (vd0) -- (v1);
    \draw (vdi) -- (vi);
	
    \node[vertex,label={below:$w_1$}] (w1) at (3.5,-1){};
    \node[vertex,label={below:$w_2$}] (w2) at (4.5,-1){};
    \draw (w1) -- (w2) -- (v3);
    
    \foreach \i in {1,2} {
    	\pgfmathtruncatemacro{\j}{\i+1}
    	\draw (v\i) -- (v\j);
	}
	\draw (v3) -- (vd1);
	
	\node[vertex,label={below:$w_3$}] (w3) at (4.5, -2) {};
	\node[vertex,label={below:$w_4$}] (w4) at (6.5, -2) {};
	
	\draw[thickedge] (v1) --  (w1) -- (w3) -- (w4) -- (vi);
    \draw (vi) -- (vd1);
  
\end{tikzpicture}

\begin{tikzpicture}[scale=1]

    \foreach \i in {1,2,3}{
    	\pgfmathtruncatemacro{\j}{\i+2}
    	\node[vertex,label={above:$v_\i$}] (v\i) at (\j,0){};
    }
    \node (vd1) at (6,0) {$\ldots$};
	\node[vertex,label={above:$v_i\in R$}] (vi) at (7,0){};
    \node (vd0) at (2,0) {$\ldots$};
    \node (vdi) at (8,0) {$\ldots$};
    \draw (vd0) -- (v1);
    \draw (vdi) -- (vi);
    
    \node[vertex,label={below:$w_1$}] (w1) at (3.5,-1){};
    \node[vertex,label={below:$w_2$}] (w2) at (4.5,-1){};
    \draw (w2) -- (v3);
    
    \foreach \i in {1,2} {
    	\pgfmathtruncatemacro{\j}{\i+1}
    	\draw (v\i) -- (v\j);
	}
	\draw (v3) -- (vd1);
	
	\node[vertex,label={below:$w_3$}] (w3) at (5.5, -2) {};
	\node[vertex,label={below:$w_4$}] (w4) at (6.5, -2) {};
	
	\draw[thickedge] (v1) -- (w1) --  (w2) -- (w3) -- (w4) -- (vi);
    \draw (vi) -- (vd1);
  
\end{tikzpicture}

\begin{tikzpicture}[scale=1]

    \foreach \i in {1,2,3}{
    	\pgfmathtruncatemacro{\j}{\i+2}
    	\node[vertex,label={above:$v_\i$}] (v\i) at (\j,0){};
    }
    \node (vd1) at (2,0) {$\ldots$};
    \node (vd3) at (6,0) {$\ldots$};
    \draw (vd1) -- (v1);
    \draw (vd3) -- (v3);
    
    \node[vertex,label={below:$w_1$}] (w1) at (3.5,-1){};
    \node[vertex,label={below:$w_2$}] (w2) at (4.5,-1){};
    \draw (v1) -- (w1) -- (w2) -- (v3);
    
    \foreach \i in {1,2} {
    	\pgfmathtruncatemacro{\j}{\i+1}
    	\draw (v\i) -- (v\j);
	}
	
	\node[vertex,label={right:$w_3$}] (w3) at (5, -2) {};
	\node[vertex,label={right:$w_4$}] (w4) at (5.5, -1) {};
	
	\draw (w1) -- (w3) -- (w4) -- (v3);
	
	\node[elabel] (e1) at (4,-0.5) {$\emptyset$};
	\node[elabel] (e2) at (5,-1) {$\emptyset$};

\end{tikzpicture}

\begin{tikzpicture}[scale=1]

    \foreach \i in {1,2,3}{
    	\pgfmathtruncatemacro{\j}{\i+2}
    	\node[vertex,label={above:$v_\i$}] (v\i) at (\j,0){};
    }
    \node (vd1) at (2,0) {$\ldots$};
    \node (vd3) at (6,0) {$\ldots$};
    \draw (vd1) -- (v1);
    \draw (vd3) -- (v3);
	
    \node[vertex,label={below:$w_1$}] (w1) at (3.5,-1){};
    \node[vertex,label={below:$w_2$}] (w2) at (4.5,-1){};
    \draw (v1) -- (w1) -- (w2) -- (v3);
    
    \foreach \i in {1,2} {
    	\pgfmathtruncatemacro{\j}{\i+1}
    	\draw (v\i) -- (v\j);
	}
	
	\node[vertex,label={left:$w_3$}] (w3) at (2, -2) {};
	\node[vertex,label={left:$w_4$}] (w4) at (1.5, -1) {};
	
	\draw (w1) -- (w3) -- (w4) -- (v1);

\end{tikzpicture}

\begin{tikzpicture}[scale=1]

    \foreach \i in {1,2,3}{
    	\pgfmathtruncatemacro{\j}{\i+2}
    	\node[vertex,label={above:$v_\i$}] (v\i) at (\j,0){};
    }
	\node[vertex,label={above:$v_4$}] (v4) at (6.5,0){};
    \node (vd1) at (2,0) {$\ldots$};
    \node (vd4) at (7.5,0) {$\ldots$};
    \draw (vd1) -- (v1);
    \draw (vd4) -- (v4);
    
    \node[vertex,label={below:$w_1$}] (w1) at (3.5,-1){};
    \node[vertex,label={below:$w_2$}] (w2) at (4.5,-1){};
    \draw (v1) -- (w1) -- (w2) -- (v3);
    
    \foreach \i in {1,2,3} {
    	\pgfmathtruncatemacro{\j}{\i+1}
    	\draw (v\i) -- (v\j);
	}
	
	\node[vertex,label={below:$w_3$}] (w3) at (5.5, -2) {};
	\node[vertex,label={below:$w_4$}] (w4) at (6.5, -2) {};
	
	\draw (w2) -- (w3) -- (w4) -- (v4);
	

    \draw[dashed] (2.5,0.5) -- (2.5,0) arc (180:270:0.5) --(4,-0.5) arc (-90:0:0.5) -- (4.5,0.5);
\end{tikzpicture}

\begin{tikzpicture}[scale=1]
    \node[precoloured,label={above:$u_1$}] (u1) at (1,0){};
    \begin{scope}[xshift=1cm]
    \begin{scope}[shift={(135:-2cm)},rotate=90, xscale=-1]
    \node[precoloured] (u2) at (105:2cm){};
    \node[precoloured] (u3) at (75:2cm){};
    \end{scope}
    \end{scope}
    \node[precoloured] (u4) at ($ (u3) + (0,-1) $){};

    \begin{scope}[xshift=5cm, xscale=-1]
    \begin{scope}[shift={(135:-2cm)},rotate=90, xscale=-1]  
    \node[vertex,label={right:$v_j$}] (vj) at (105:2cm) {};
    \node[vertex,label={right:$v_{j+1}\in \hat B$}] (vj1) at (75:2cm){};
    \end{scope}
    \end{scope}
    \node[rotate=90] (vd2) at ($ (vj1) + (0,-1) $) {$\ldots$};
    \draw (vj1) -- (vd2);
    
    \node[vertex,label={left:$u\in \hat A$}] (vA) at (3,-2){};
    \node[vertex,paren,label={above:$v_0$}] (v0) at (2,0){};
    \node[vertex,label={above:$v_1$}] (v1) at (3,0){};
    \node[vertex,label={above:$v_2$}] (v2) at (4,0){};
    \node[vertex,label={above:$v_3$}] (v3) at (5,0){};
    
    \node[vertex] (w1) at (3.5,-1){};
    \node[vertex] (w2) at (4.5,-1){};
    
    
    \draw (v1) -- (w1) -- (w2) -- (v3);
    \draw (u1) -- (u2) -- (u3) -- (u4);
    \draw[thickedge] (u1) -- (vA) -- (vj1);
    \draw[bend right] (u3) to (vj1);
    \draw[bend right] (u4) to (vj1);
    \draw (u1) -- (v0) -- (v1) -- (v2) -- (v3);
    \draw (v3) to node[midway,sloped, fill=white] {$\ldots$} (vj) -- (vj1);

   \coordinate (BB1) at (current bounding box.north east);
   \coordinate (BB2) at (current bounding box.south west);
\end{tikzpicture}

\begin{tikzpicture}[scale=1]
    \useasboundingbox ($ (BB1) - (1,0) $) rectangle (BB2);

    \node[precoloured,label={above:$u_1$}] (u1) at (1,0){};
    \begin{scope}[xshift=1cm]
    \begin{scope}[shift={(135:-2cm)},rotate=90, xscale=-1]
    \node[precoloured] (u2) at (105:2cm){};
    \node[precoloured] (u3) at (75:2cm){};
    \node[precoloured] (u4) at (45:2cm){};

    \node[vertex,label={below left:$v\in \hat B$}] (vB) at (-0.75,0) {};
    \node[vertex,label={right:$u\in \hat A$}] (vA) at (0.75,0) {};
    \draw (vB) -- (u1) -- (u2) -- (u3) -- (u4) -- (vA);
    \draw (u2) -- (vB) -- (vA);
    \draw[thickedge] (u1) -- (vB) -- (vA) -- (u4);
    \end{scope}
    \end{scope}

    
    \node[vertex,paren,label={above:$v_0$}] (v0) at (2,0){};
    \node[vertex,label={above:$v_1$}] (v1) at (3,0){};
    \node[vertex,label={above:$v_2$}] (v2) at (4,0){};
    \node[vertex,label={above:$v_3$}] (v3) at (5,0){};
    \node (vd) at (6,0) {$\ldots$};

    \node[vertex] (w1) at (3.5,-1){};
    \node[vertex] (w2) at (4.5,-1){};
    \draw (v1) -- (w1) -- (w2) -- (v3);
    
    \draw (u1) -- (v0) -- (v1) -- (v2) -- (v3) -- (vd);
\end{tikzpicture}}%
    \hspace*{0.05\textwidth}
    \subcaptionbox{Case~\ref{itm:w3w2}.\label{fig:fs2w3w2}}[0.45\textwidth]{\begin{tikzpicture}[scale=1]
    
    \foreach \i in {1,2,3}{
    	\pgfmathtruncatemacro{\j}{\i+2}
    	\node[vertex,label={above:$v_\i$}] (v\i) at (\j,0){};
    }
    \node (vd1) at (6,0) {$\ldots$};
	\node[vertex,label={above:$v_i\in R$}] (vi) at (7,0){};
    \node (vd0) at (2,0) {$\ldots$};
    \node (vdi) at (8,0) {$\ldots$};
    \draw (vd0) -- (v1);
    \draw (vdi) -- (vi);
	
    \node[vertex,label={below:$w_1$}] (w1) at (3.5,-1){};
    \node[vertex,label={below:$w_2$}] (w2) at (4.5,-1){};
    \draw (w1) -- (w2) -- (v3);
    
    \foreach \i in {1,2} {
    	\pgfmathtruncatemacro{\j}{\i+1}
    	\draw (v\i) -- (v\j);
	}
	\draw (v3) -- (vd1);
	
	\node[vertex,label={below:$w_3$}] (w3) at (4.5, -2) {};
	\node[vertex,label={below:$w_4$}] (w4) at (6.5, -2) {};
	
	\draw[thickedge] (v1) --  (w1) -- (w3) -- (w4) -- (vi);
    \draw (vi) -- (vd1);
  
\end{tikzpicture}

\begin{tikzpicture}[scale=1]

    \foreach \i in {1,2,3}{
    	\pgfmathtruncatemacro{\j}{\i+2}
    	\node[vertex,label={above:$v_\i$}] (v\i) at (\j,0){};
    }
    \node (vd1) at (6,0) {$\ldots$};
	\node[vertex,label={above:$v_i\in R$}] (vi) at (7,0){};
    \node (vd0) at (2,0) {$\ldots$};
    \node (vdi) at (8,0) {$\ldots$};
    \draw (vd0) -- (v1);
    \draw (vdi) -- (vi);
    
    \node[vertex,label={below:$w_1$}] (w1) at (3.5,-1){};
    \node[vertex,label={below:$w_2$}] (w2) at (4.5,-1){};
    \draw (w2) -- (v3);
    
    \foreach \i in {1,2} {
    	\pgfmathtruncatemacro{\j}{\i+1}
    	\draw (v\i) -- (v\j);
	}
	\draw (v3) -- (vd1);
	
	\node[vertex,label={below:$w_3$}] (w3) at (5.5, -2) {};
	\node[vertex,label={below:$w_4$}] (w4) at (6.5, -2) {};
	
	\draw[thickedge] (v1) -- (w1) --  (w2) -- (w3) -- (w4) -- (vi);
    \draw (vi) -- (vd1);
  
\end{tikzpicture}

\begin{tikzpicture}[scale=1]

    \foreach \i in {1,2,3}{
    	\pgfmathtruncatemacro{\j}{\i+2}
    	\node[vertex,label={above:$v_\i$}] (v\i) at (\j,0){};
    }
    \node (vd1) at (2,0) {$\ldots$};
    \node (vd3) at (6,0) {$\ldots$};
    \draw (vd1) -- (v1);
    \draw (vd3) -- (v3);
    
    \node[vertex,label={below:$w_1$}] (w1) at (3.5,-1){};
    \node[vertex,label={below:$w_2$}] (w2) at (4.5,-1){};
    \draw (v1) -- (w1) -- (w2) -- (v3);
    
    \foreach \i in {1,2} {
    	\pgfmathtruncatemacro{\j}{\i+1}
    	\draw (v\i) -- (v\j);
	}
	
	\node[vertex,label={right:$w_3$}] (w3) at (5, -2) {};
	\node[vertex,label={right:$w_4$}] (w4) at (5.5, -1) {};
	
	\draw (w1) -- (w3) -- (w4) -- (v3);
	
	\node[elabel] (e1) at (4,-0.5) {$\emptyset$};
	\node[elabel] (e2) at (5,-1) {$\emptyset$};

\end{tikzpicture}

\begin{tikzpicture}[scale=1]

    \foreach \i in {1,2,3}{
    	\pgfmathtruncatemacro{\j}{\i+2}
    	\node[vertex,label={above:$v_\i$}] (v\i) at (\j,0){};
    }
    \node (vd1) at (2,0) {$\ldots$};
    \node (vd3) at (6,0) {$\ldots$};
    \draw (vd1) -- (v1);
    \draw (vd3) -- (v3);
	
    \node[vertex,label={below:$w_1$}] (w1) at (3.5,-1){};
    \node[vertex,label={below:$w_2$}] (w2) at (4.5,-1){};
    \draw (v1) -- (w1) -- (w2) -- (v3);
    
    \foreach \i in {1,2} {
    	\pgfmathtruncatemacro{\j}{\i+1}
    	\draw (v\i) -- (v\j);
	}
	
	\node[vertex,label={left:$w_3$}] (w3) at (2, -2) {};
	\node[vertex,label={left:$w_4$}] (w4) at (1.5, -1) {};
	
	\draw (w1) -- (w3) -- (w4) -- (v1);

\end{tikzpicture}

\begin{tikzpicture}[scale=1]

    \foreach \i in {1,2,3}{
    	\pgfmathtruncatemacro{\j}{\i+2}
    	\node[vertex,label={above:$v_\i$}] (v\i) at (\j,0){};
    }
	\node[vertex,label={above:$v_4$}] (v4) at (6.5,0){};
    \node (vd1) at (2,0) {$\ldots$};
    \node (vd4) at (7.5,0) {$\ldots$};
    \draw (vd1) -- (v1);
    \draw (vd4) -- (v4);
    
    \node[vertex,label={below:$w_1$}] (w1) at (3.5,-1){};
    \node[vertex,label={below:$w_2$}] (w2) at (4.5,-1){};
    \draw (v1) -- (w1) -- (w2) -- (v3);
    
    \foreach \i in {1,2,3} {
    	\pgfmathtruncatemacro{\j}{\i+1}
    	\draw (v\i) -- (v\j);
	}
	
	\node[vertex,label={below:$w_3$}] (w3) at (5.5, -2) {};
	\node[vertex,label={below:$w_4$}] (w4) at (6.5, -2) {};
	
	\draw (w2) -- (w3) -- (w4) -- (v4);
	

    \draw[dashed] (2.5,0.5) -- (2.5,0) arc (180:270:0.5) --(4,-0.5) arc (-90:0:0.5) -- (4.5,0.5);
\end{tikzpicture}

\begin{tikzpicture}[scale=1]
    \node[precoloured,label={above:$u_1$}] (u1) at (1,0){};
    \begin{scope}[xshift=1cm]
    \begin{scope}[shift={(135:-2cm)},rotate=90, xscale=-1]
    \node[precoloured] (u2) at (105:2cm){};
    \node[precoloured] (u3) at (75:2cm){};
    \end{scope}
    \end{scope}
    \node[precoloured] (u4) at ($ (u3) + (0,-1) $){};

    \begin{scope}[xshift=5cm, xscale=-1]
    \begin{scope}[shift={(135:-2cm)},rotate=90, xscale=-1]  
    \node[vertex,label={right:$v_j$}] (vj) at (105:2cm) {};
    \node[vertex,label={right:$v_{j+1}\in \hat B$}] (vj1) at (75:2cm){};
    \end{scope}
    \end{scope}
    \node[rotate=90] (vd2) at ($ (vj1) + (0,-1) $) {$\ldots$};
    \draw (vj1) -- (vd2);
    
    \node[vertex,label={left:$u\in \hat A$}] (vA) at (3,-2){};
    \node[vertex,paren,label={above:$v_0$}] (v0) at (2,0){};
    \node[vertex,label={above:$v_1$}] (v1) at (3,0){};
    \node[vertex,label={above:$v_2$}] (v2) at (4,0){};
    \node[vertex,label={above:$v_3$}] (v3) at (5,0){};
    
    \node[vertex] (w1) at (3.5,-1){};
    \node[vertex] (w2) at (4.5,-1){};
    
    
    \draw (v1) -- (w1) -- (w2) -- (v3);
    \draw (u1) -- (u2) -- (u3) -- (u4);
    \draw[thickedge] (u1) -- (vA) -- (vj1);
    \draw[bend right] (u3) to (vj1);
    \draw[bend right] (u4) to (vj1);
    \draw (u1) -- (v0) -- (v1) -- (v2) -- (v3);
    \draw (v3) to node[midway,sloped, fill=white] {$\ldots$} (vj) -- (vj1);

   \coordinate (BB1) at (current bounding box.north east);
   \coordinate (BB2) at (current bounding box.south west);
\end{tikzpicture}

\begin{tikzpicture}[scale=1]
    \useasboundingbox ($ (BB1) - (1,0) $) rectangle (BB2);

    \node[precoloured,label={above:$u_1$}] (u1) at (1,0){};
    \begin{scope}[xshift=1cm]
    \begin{scope}[shift={(135:-2cm)},rotate=90, xscale=-1]
    \node[precoloured] (u2) at (105:2cm){};
    \node[precoloured] (u3) at (75:2cm){};
    \node[precoloured] (u4) at (45:2cm){};

    \node[vertex,label={below left:$v\in \hat B$}] (vB) at (-0.75,0) {};
    \node[vertex,label={right:$u\in \hat A$}] (vA) at (0.75,0) {};
    \draw (vB) -- (u1) -- (u2) -- (u3) -- (u4) -- (vA);
    \draw (u2) -- (vB) -- (vA);
    \draw[thickedge] (u1) -- (vB) -- (vA) -- (u4);
    \end{scope}
    \end{scope}

    
    \node[vertex,paren,label={above:$v_0$}] (v0) at (2,0){};
    \node[vertex,label={above:$v_1$}] (v1) at (3,0){};
    \node[vertex,label={above:$v_2$}] (v2) at (4,0){};
    \node[vertex,label={above:$v_3$}] (v3) at (5,0){};
    \node (vd) at (6,0) {$\ldots$};

    \node[vertex] (w1) at (3.5,-1){};
    \node[vertex] (w2) at (4.5,-1){};
    \draw (v1) -- (w1) -- (w2) -- (v3);
    
    \draw (u1) -- (v0) -- (v1) -- (v2) -- (v3) -- (vd);
\end{tikzpicture}}%
    \captionsetup{width=0.9\textwidth}
    \caption{The structure given by Claim~\ref{claim:finalstructure2} for an edge $w_3w_4\in \hat G[\hat A]$ chosen to minimize $i$. The edges of the path $Q$ are bold. $\delta G_2$ is the cycle $v_1\dotsb v_i w_4\dotsb w_1v_1$.\label{fig:AAdetails}}
\end{figure}
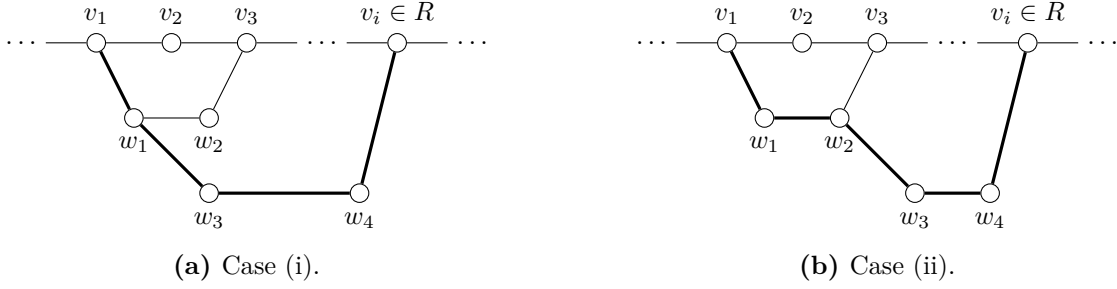

Since $G$ is a plane graph, our choice of $i$ gives us the following. 

\begin{obs}\label{obs:goodichoice}
    There does not exist a vertex $w \in V(G_2) \setminus V(\delta G_2)$ with $g(w) \geq 5$ (equivalently, $f(w) = 2$) such that $w$ is adjacent to a vertex in $V(Q')$ and to $w_3$.
\end{obs}
Similarly, Lemma \ref{lem:3chords} and Observation \ref{obs:whereistheAvert} (noting that $Q' \subseteq R$ since $R$ is defined according to case~\ref{case:f2233} or~\ref{case:f223}) give the following.
\begin{obs}\label{obs:goodichoice2}
    There does not exist a vertex $w \in V(G_2) \setminus V(\delta G_2)$ with $g(w) \geq 5$ (equivalently, $f(w) = 2$) such that $w$ is adjacent to a vertex in $V(Q')$ and to $w_4$.
\end{obs}
In fact, these observations can be strengthened as follows. 
\begin{claim}
Either $Q = v_1w_1w_2w_3w_4v_i$ and $G_2$ has no interior vertex of girth at least five adjacent to two vertices in $V(\delta G_2)$, or $Q = v_1w_1w_3w_4v_i$ and the only interior vertex of $G_2$ of girth at least five adjacent to two vertices in $V(\delta G_2)$ is $w_2$.  
\end{claim}
\begin{cproof}
Let $w$ be an interior vertex of $G_2$ with $g(w) \geq 5$. Note that since the interior of $v_1v_2v_3w_2w_1v_1$ is empty by Lemma \ref{lem:5cycles} and the fact that $g(w_1) = 5$, $w$ is not adjacent to $v_1$ or $v_2$. If $w$ is adjacent to $w_1$, then ($Q = v_1w_1w_3w_4v_i$ and) $w$ is not adjacent to $w_4$ since $g(w) \geq 5$. If $w$ is adjacent to a vertex $r \in V(Q')$, then  $w \neq w_2$, we have that $v_1w_1wr$ is a 3-chord which, given Observation \ref{obs:whereistheAvert}, contradicts Lemma \ref{lem:3chords} unless $r = v_3$. But then $v_2$ is in the interior of $v_1w_1vv_3v_1$, which contradicts Lemma \ref{lem:3chords}.

Thus we may assume $w$ is not adjacent to $w_1$, $v_1,$ or $v_2$. By Lemma \ref{lem:2chords}, $w$ is not adjacent to two vertices in $Q'$. If $w$ is adjacent to $w_3$, it is not adjacent to a vertex in $Q'$ by Observation \ref{obs:goodichoice}. If $w$ is adjacent to $w_4$, it is not adjacent to a vertex in $Q'$ by Observation \ref{obs:goodichoice2}.
\end{cproof}

Lemmas \ref{lem:2chords}, \ref{lem:3chords}, and~\ref{lem:AAstructure}\ref{itm:onlynbrsw1w2} now easily give the following, which  establishes more of the structure of $\delta G_2$.
\begin{obs}\label{obs:notmanyQ'nbrs}
$N(w_1) \cap V(Q') = \{v_1 \}$. $N(w_2) \cap V(Q') = \{v_3\}$. $N(w_3) \cap V(Q') = \emptyset$. $N(w_4) \cap V(Q') = \{v_i\}$.
\end{obs}

We need one more claim.

\begin{claim}\label{claim:i>4}
   If $Q = v_1w_1w_2w_3w_4v_i$, then $i > 4$.
\end{claim}
\begin{cproof}
Note that by Claim \ref{claim:finalstructure2}, $i\geq 4$. Suppose for a contradiction that $Q = v_1w_1w_2w_3w_4v_i$ and $i = 4$, as depicted in Figure~\ref{fig:i>4}. 
Note that since $v_4 \in V(R)$, by Observation \ref{obs:whereistheAvert}, neither $v_3$ nor $v_4$ are in $A$. By Lemma~\ref{lem:5cycles} and the fact that $g(w_2) = 5$, the 5-cycles $v_1v_2v_3w_2w_1v_1$ and $v_3v_4w_4w_3w_2v_3$ have empty interiors, and hence $\deg(v_3) = 3$.  Since $g(w_2) = g(v_2) = 5$, it follows that $g(v_3) = 5$, as well.  Let $(G'', f'')$ be the graph and function obtained from $(G, f)$ by performing the operations $\delsave(v_1,v_0),\del(v_2)$. Note this sequence of operations is legal: the first operation is legal since $f(v_1) = 2$ and if $v_0$ exists, then by Lemma \ref{lem:1chord} $v_1$ is not adjacent to a vertex in $P$. The second operation is legal since $N(v_2) = \{v_1,v_3\}$ and $f(v_3) = 2$.

\begin{figure}[ht]\centering
    \begin{tikzpicture}[scale=1]
    
    \foreach \i in {1,2,3}{
    	\pgfmathtruncatemacro{\j}{\i+2}
    	\node[vertex,label={above:$v_\i$}] (v\i) at (\j,0){};
    }
    \node (vd1) at (6,0) {$\ldots$};
	\node[vertex,label={above:$v_i\in R$}] (vi) at (7,0){};
    \node (vd0) at (2,0) {$\ldots$};
    \node (vdi) at (8,0) {$\ldots$};
    \draw (vd0) -- (v1);
    \draw (vdi) -- (vi);
	
    \node[vertex,label={below:$w_1$}] (w1) at (3.5,-1){};
    \node[vertex,label={below:$w_2$}] (w2) at (4.5,-1){};
    \draw (w1) -- (w2) -- (v3);
    
    \foreach \i in {1,2} {
    	\pgfmathtruncatemacro{\j}{\i+1}
    	\draw (v\i) -- (v\j);
	}
	\draw (v3) -- (vd1);
	
	\node[vertex,label={below:$w_3$}] (w3) at (4.5, -2) {};
	\node[vertex,label={below:$w_4$}] (w4) at (6.5, -2) {};
	
	\draw[thickedge] (v1) --  (w1) -- (w3) -- (w4) -- (vi);
    \draw (vi) -- (vd1);
  
\end{tikzpicture}

\begin{tikzpicture}[scale=1]

    \foreach \i in {1,2,3}{
    	\pgfmathtruncatemacro{\j}{\i+2}
    	\node[vertex,label={above:$v_\i$}] (v\i) at (\j,0){};
    }
    \node (vd1) at (6,0) {$\ldots$};
	\node[vertex,label={above:$v_i\in R$}] (vi) at (7,0){};
    \node (vd0) at (2,0) {$\ldots$};
    \node (vdi) at (8,0) {$\ldots$};
    \draw (vd0) -- (v1);
    \draw (vdi) -- (vi);
    
    \node[vertex,label={below:$w_1$}] (w1) at (3.5,-1){};
    \node[vertex,label={below:$w_2$}] (w2) at (4.5,-1){};
    \draw (w2) -- (v3);
    
    \foreach \i in {1,2} {
    	\pgfmathtruncatemacro{\j}{\i+1}
    	\draw (v\i) -- (v\j);
	}
	\draw (v3) -- (vd1);
	
	\node[vertex,label={below:$w_3$}] (w3) at (5.5, -2) {};
	\node[vertex,label={below:$w_4$}] (w4) at (6.5, -2) {};
	
	\draw[thickedge] (v1) -- (w1) --  (w2) -- (w3) -- (w4) -- (vi);
    \draw (vi) -- (vd1);
  
\end{tikzpicture}

\begin{tikzpicture}[scale=1]

    \foreach \i in {1,2,3}{
    	\pgfmathtruncatemacro{\j}{\i+2}
    	\node[vertex,label={above:$v_\i$}] (v\i) at (\j,0){};
    }
    \node (vd1) at (2,0) {$\ldots$};
    \node (vd3) at (6,0) {$\ldots$};
    \draw (vd1) -- (v1);
    \draw (vd3) -- (v3);
    
    \node[vertex,label={below:$w_1$}] (w1) at (3.5,-1){};
    \node[vertex,label={below:$w_2$}] (w2) at (4.5,-1){};
    \draw (v1) -- (w1) -- (w2) -- (v3);
    
    \foreach \i in {1,2} {
    	\pgfmathtruncatemacro{\j}{\i+1}
    	\draw (v\i) -- (v\j);
	}
	
	\node[vertex,label={right:$w_3$}] (w3) at (5, -2) {};
	\node[vertex,label={right:$w_4$}] (w4) at (5.5, -1) {};
	
	\draw (w1) -- (w3) -- (w4) -- (v3);
	
	\node[elabel] (e1) at (4,-0.5) {$\emptyset$};
	\node[elabel] (e2) at (5,-1) {$\emptyset$};

\end{tikzpicture}

\begin{tikzpicture}[scale=1]

    \foreach \i in {1,2,3}{
    	\pgfmathtruncatemacro{\j}{\i+2}
    	\node[vertex,label={above:$v_\i$}] (v\i) at (\j,0){};
    }
    \node (vd1) at (2,0) {$\ldots$};
    \node (vd3) at (6,0) {$\ldots$};
    \draw (vd1) -- (v1);
    \draw (vd3) -- (v3);
	
    \node[vertex,label={below:$w_1$}] (w1) at (3.5,-1){};
    \node[vertex,label={below:$w_2$}] (w2) at (4.5,-1){};
    \draw (v1) -- (w1) -- (w2) -- (v3);
    
    \foreach \i in {1,2} {
    	\pgfmathtruncatemacro{\j}{\i+1}
    	\draw (v\i) -- (v\j);
	}
	
	\node[vertex,label={left:$w_3$}] (w3) at (2, -2) {};
	\node[vertex,label={left:$w_4$}] (w4) at (1.5, -1) {};
	
	\draw (w1) -- (w3) -- (w4) -- (v1);

\end{tikzpicture}

\begin{tikzpicture}[scale=1]

    \foreach \i in {1,2,3}{
    	\pgfmathtruncatemacro{\j}{\i+2}
    	\node[vertex,label={above:$v_\i$}] (v\i) at (\j,0){};
    }
	\node[vertex,label={above:$v_4$}] (v4) at (6.5,0){};
    \node (vd1) at (2,0) {$\ldots$};
    \node (vd4) at (7.5,0) {$\ldots$};
    \draw (vd1) -- (v1);
    \draw (vd4) -- (v4);
    
    \node[vertex,label={below:$w_1$}] (w1) at (3.5,-1){};
    \node[vertex,label={below:$w_2$}] (w2) at (4.5,-1){};
    \draw (v1) -- (w1) -- (w2) -- (v3);
    
    \foreach \i in {1,2,3} {
    	\pgfmathtruncatemacro{\j}{\i+1}
    	\draw (v\i) -- (v\j);
	}
	
	\node[vertex,label={below:$w_3$}] (w3) at (5.5, -2) {};
	\node[vertex,label={below:$w_4$}] (w4) at (6.5, -2) {};
	
	\draw (w2) -- (w3) -- (w4) -- (v4);
	

    \draw[dashed] (2.5,0.5) -- (2.5,0) arc (180:270:0.5) --(4,-0.5) arc (-90:0:0.5) -- (4.5,0.5);
\end{tikzpicture}

\begin{tikzpicture}[scale=1]
    \node[precoloured,label={above:$u_1$}] (u1) at (1,0){};
    \begin{scope}[xshift=1cm]
    \begin{scope}[shift={(135:-2cm)},rotate=90, xscale=-1]
    \node[precoloured] (u2) at (105:2cm){};
    \node[precoloured] (u3) at (75:2cm){};
    \end{scope}
    \end{scope}
    \node[precoloured] (u4) at ($ (u3) + (0,-1) $){};

    \begin{scope}[xshift=5cm, xscale=-1]
    \begin{scope}[shift={(135:-2cm)},rotate=90, xscale=-1]  
    \node[vertex,label={right:$v_j$}] (vj) at (105:2cm) {};
    \node[vertex,label={right:$v_{j+1}\in \hat B$}] (vj1) at (75:2cm){};
    \end{scope}
    \end{scope}
    \node[rotate=90] (vd2) at ($ (vj1) + (0,-1) $) {$\ldots$};
    \draw (vj1) -- (vd2);
    
    \node[vertex,label={left:$u\in \hat A$}] (vA) at (3,-2){};
    \node[vertex,paren,label={above:$v_0$}] (v0) at (2,0){};
    \node[vertex,label={above:$v_1$}] (v1) at (3,0){};
    \node[vertex,label={above:$v_2$}] (v2) at (4,0){};
    \node[vertex,label={above:$v_3$}] (v3) at (5,0){};
    
    \node[vertex] (w1) at (3.5,-1){};
    \node[vertex] (w2) at (4.5,-1){};
    
    
    \draw (v1) -- (w1) -- (w2) -- (v3);
    \draw (u1) -- (u2) -- (u3) -- (u4);
    \draw[thickedge] (u1) -- (vA) -- (vj1);
    \draw[bend right] (u3) to (vj1);
    \draw[bend right] (u4) to (vj1);
    \draw (u1) -- (v0) -- (v1) -- (v2) -- (v3);
    \draw (v3) to node[midway,sloped, fill=white] {$\ldots$} (vj) -- (vj1);

   \coordinate (BB1) at (current bounding box.north east);
   \coordinate (BB2) at (current bounding box.south west);
\end{tikzpicture}

\begin{tikzpicture}[scale=1]
    \useasboundingbox ($ (BB1) - (1,0) $) rectangle (BB2);

    \node[precoloured,label={above:$u_1$}] (u1) at (1,0){};
    \begin{scope}[xshift=1cm]
    \begin{scope}[shift={(135:-2cm)},rotate=90, xscale=-1]
    \node[precoloured] (u2) at (105:2cm){};
    \node[precoloured] (u3) at (75:2cm){};
    \node[precoloured] (u4) at (45:2cm){};

    \node[vertex,label={below left:$v\in \hat B$}] (vB) at (-0.75,0) {};
    \node[vertex,label={right:$u\in \hat A$}] (vA) at (0.75,0) {};
    \draw (vB) -- (u1) -- (u2) -- (u3) -- (u4) -- (vA);
    \draw (u2) -- (vB) -- (vA);
    \draw[thickedge] (u1) -- (vB) -- (vA) -- (u4);
    \end{scope}
    \end{scope}

    
    \node[vertex,paren,label={above:$v_0$}] (v0) at (2,0){};
    \node[vertex,label={above:$v_1$}] (v1) at (3,0){};
    \node[vertex,label={above:$v_2$}] (v2) at (4,0){};
    \node[vertex,label={above:$v_3$}] (v3) at (5,0){};
    \node (vd) at (6,0) {$\ldots$};

    \node[vertex] (w1) at (3.5,-1){};
    \node[vertex] (w2) at (4.5,-1){};
    \draw (v1) -- (w1) -- (w2) -- (v3);
    
    \draw (u1) -- (v0) -- (v1) -- (v2) -- (v3) -- (vd);
\end{tikzpicture}%
    \captionsetup{width=0.9\textwidth}
    \caption{Ruling out $Q = v_1w_1w_2w_3w_4v_i$ and $i = 4$ in Claim~\ref{claim:i>4}.\label{fig:i>4}}
\end{figure}
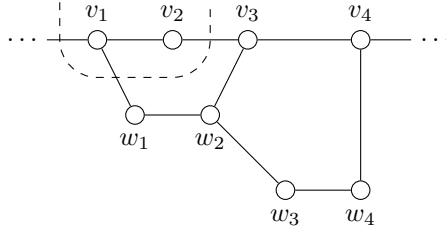

Let $A''$ be the set of vertices in $V(G'') \setminus V(P)$ with $f''(v) \leq 1$, and $B''$, the set of vertices  $v \in V(G'') \setminus V(P)$ with $g(v) = 3$ and $f(v) = 2$. Since $v_2v_1$ is a subpath of $R$, $f''(v) \geq \tilde{f}(v)$ for all $v \in V(G'') \setminus \{v_3\}$. Hence by Lemma \ref{lem:AAstructure}, since $w_2 \not \in A''$, every edge induced by $A''$ includes $v_3$.  But $N_{G''}(v_3) = \{w_2,v_4\}$. Since $f(v_4) = 2$, $v_4 \not \in A''$, and since $w_2 \not \in N(v_1) \cup N(v_2)$, also $w_2 \not \in A''$. Hence $A''$ is independent, and so by Observation \ref{obs:almostcanvas} and the fact that $v_1v_2 \subseteq V(R)$, we have that $(G''-P, A'', B'', f'')$ is a canvas. Since the interior of the 5-cycle $v_1v_2v_3w_2w_1v_1$ is empty, it is easy to see it is unexceptional. This is a contradiction, as then $G''$ is weakly $f''$ degenerate; and hence $G$ is weakly $f$-degenerate. 
\end{cproof}

The proof Theorem \ref{thm:inductive} is completed below.

\begin{claim}
$K$ is not a counterexample to Theorem \ref{thm:inductive}.
\end{claim}
\begin{cproof}
By Claims \ref{claim:finalstructure2} and \ref{claim:i>4}, either $Q = v_1w_1w_3w_4v_i$ and $i \geq 4$, or $Q = v_1w_1w_2w_3w_4v_i$ with $i > 4$. Let $f_K(v) = f(v) - |N(v) \cap V(P)|$ for all $v \in V(G_1) \setminus V(P)$. Since $v_2 \not \in V(G_1)$, we have that $v(G_1) < v(G)$, and hence by the minimality of $G$, there exists a sequence of legal degeneracy operations $\sigma$ allowing us to remove all vertices of $G_1-P$, starting from $(G_1-P, f_K)$. 

Recall that $Q = G_1 \cap G_2$. Let $(G_2-Q,f_2)$ be the graph and function obtained from $(G_2, f)$ by removing $Q$ via the same sequence of operations as in $\sigma$. Let 
\[ \sigma' := (\del(v_2), \del(v_3), \dots, \del(v_{i-1})), \]
and let $(G_3, f_3)$ be the graph and function obtained from $(G_2-Q, f_2)$ by performing the sequence of operations $\sigma'$. 
First, we claim this sequence of operations is legal.
By Observation \ref{obs:notmanyQ'nbrs}, if $Q = v_1w_1w_3w_4v_i$, then $f_2(v_2) = 0$, $f_2(v_{i-1}) = 1$, and $f_2(v_{i'}) = 2$ for all $2 < i' < i-1$. If $Q = v_1w_1w_2w_3w_4v_i$, then similarly $f_2(v_2) = 0$, $f(v_3) = 1$, $f_2(v_{i-1}) = 1$, and $f_2(v_{i'}) = 2$ for all $3 < i' < i-1$. Note that here we are using the fact that $v_{i-1}$ is not adjacent to \emph{both} $v_i$ and $w_2$, since  $i > 4$ by Claim \ref{claim:i>4}.

Finally, let $G_4 = G_2-Q'$. Let $f_4(v) = f_3(v) + |N(v) \cap (V(Q)\setminus \{v_1,v_i\})|$ for all $v \in G_3$. Let $A_4$ be the set of vertices in $V(G_4) \setminus V(Q)$ with $f_4(v) \leq 1$, and $B_4$ the set of vertices $\{ v \in G_4\setminus V(Q): g(v) = 3 \textnormal{ and } f_4(v) = 2\}$. 

We claim $K_4 = (G_4, Q-Q', A_4, B_4, f_4)$ is an unexceptional canvas. Since $Q' \subseteq R$ and $Q'$ is removed via the same operations as in \ref{handle:f2233} or \ref{handle:f223}, by Observation \ref{obs:almostcanvas}, $K_4$ is a canvas unless $A_4$ is not independent. By Lemma \ref{lem:AAstructure} and the fact that neither $w_1$ nor $w_2$ is in $A_4$, we conclude that $K_4$ is a canvas. Since $Q-Q'$ contains only vertices of girth five, it is not an exceptional canvas of type \ref{ex:AB} or \ref{ex:B}. Since $v_3w_2$ exists in $G$, there is no vertex in $V(G_2) \supset V(G_4)$ adjacent to both $w_1$ and $w_4$, and hence it is not an exceptional canvas of type \ref{ex:A}. 

Thus there exists a legal sequence of degeneracy operations $\sigma''$ that removes all vertices of $G_4$. But then $\sigma \sigma' \sigma''$ is a legal sequence of degeneracy operations that removes all vertices in $G$, and hence $G-P$ is weakly $f_K$-degenerate, and so $G$ is weakly $f$-degenerate\textemdash a contradiction.
\end{cproof}

\bibliographystyle{habbrv}
\bibliography{bibliog}

\noindent
\end{document}